\documentclass[11pt]{amsart}
\usepackage{hyperref}
\usepackage{amsmath}
\usepackage{amssymb}
\usepackage{xypic}
\usepackage{yhmath}
\usepackage{bm}
\usepackage{xcolor}
\usepackage{setspace}
\usepackage{geometry}
\def\beqnn{\begin{eqnarray*}}\def\eeqnn{\end{eqnarray*}}

\newtheorem{theorem}{Theorem}[section]
\newtheorem{lemma}[theorem]{Lemma}
\newtheorem{proposition}[theorem]{Proposition}
\newtheorem{corollary}[theorem]{Corollary}

\theoremstyle{claim}
\newtheorem{claim}[theorem]{Claim}

\theoremstyle{definition}
\newtheorem{definition}[theorem]{Definition}
\newtheorem{example}[theorem]{Example}

\newtheorem{remark}[theorem]{Remark}

\theoremstyle{problem}
\newtheorem{problem}[theorem]{Problem}
\theoremstyle{conjecture}
\newtheorem{conjecture}[theorem]{Conjecture}
\theoremstyle{question}

\numberwithin{equation}{section}

\begin{document}

\begin{center}
\title[Complex integral means spectrum and Brennan conjecture]{Complex exponential integral means spectra of univalent functions and the Brennan conjecture}
\end{center}

 \author[Jianjun Jin]{Jianjun Jin}

 \address{School of Mathematics, Hefei University of Technology, Xuancheng Campus, Xuancheng 242000, P.R.China}

\email{jin@hfut.edu.cn, jinjjhb@163.com}



\begin{abstract}\begin{spacing}{1.1}
In this paper we investigate the complex exponential integral means spectra of univalent functions in the unit disk. We show that all integral means spectrum (IMS) functionals for complex exponents on the universal Teichm\"uller space, the closure of the universal Teichm\"uller curve, and the universal asymptotic Teichm\"uller space are continuous. We also show that the complex exponential integral means spectrum of any univalent function admitting a quasiconformal extension to the extended complex plane is strictly less than the universal integral means spectrum. These extend some related results in our recent work \cite{Jin}. Here we employ a different and more direct approach to prove the continuity of the IMS functional on the universal asymptotic Teichm\"uller space. Additionally, we completely determine the integral means spectra of all univalent rational functions in the unit disk. As a consequence, we show that the Brennan conjecture is true for this class of univalent functions. Finally, we present some remarks and raise some problems and conjectures regarding IMS functionals on Teichm\"uller spaces, univalent rational functions, and a multiplication operator whose norm is closely related to the Brennan conjecture. \end{spacing}\end{abstract}

\maketitle

\begin{spacing}{1.3}
{\footnotesize  \scshape  \addtocontents{toc}{\protect\color{blue}} 
  \tableofcontents}
\end{spacing}

\section{Introduction}
Let $\Delta=\{z:|z|<1\}$ be the unit disk in the complex plane $\mathbb{C}$. We denote by $\widehat{\mathbb{C}}=\mathbb{C}\cup\{\infty\}$ the extended complex plane and let $\mathbb{R}$ be the real line. 
For  a complex number $z$, we use $\arg z$ to denote the unique argument of $z$ satisfying that $\arg z \in (-\pi, \pi]$. Let $\Omega$ be a bounded simply connected domain in $\mathbb{C}$ with $0\in \Omega$.  Let $f$ be an analytic function in $\Omega$ with $f(0)=1$ and $f(z)\neq 0$ for all $z\in \Omega$. In this paper, we always use $\log f$ to denote the unique single-valued branch of the logarithm of $f$ in $\Omega$ with $\arg f(0)=0$.  

We denote by $\mathcal{U}$ the class of all univalent functions (conformal mappings) in $\Delta$. We let $\mathcal{S}$ be the class of all univalent functions $f$ in $\Delta$ with $f(0)=f'(0)-1=0$ and let $\mathcal{S}_b$ be the subclass of $\mathcal{S}$ which consists of all bounded univalent functions.  

Let $\tau \in \mathbb{C}$. The {\em integral means spectrum} $\beta_f(\tau)$ for $f\in \mathcal{S}$ is defined as
\begin{equation}\label{defi}
\beta_f(\tau):=\limsup_{r \rightarrow 1^{-}}\frac{\log \int_{-\pi}^{\pi} |[f'(re^{i\theta})]^{\tau}| d\theta}{|\log(1-r)|}. 
\end{equation}

\begin{remark}
For $f\in \mathcal{S}$, since $f'(z)$ is never zero in $\Delta$ and $f'(0)=1$, we define the complex power $[f'(z)]^{\tau}$ as 
$$[f'(z)]^{\tau}=\exp(\tau \log f'(z)).$$
For real $\tau$, the integral means measure the boundary expansion and compression associated with a given conformal map. Allowing $\tau$ to be a complex exponent, we should also take into account the rotation.
\end{remark}

The {\em universal integral means spectra} $B(\tau)$ and $B_b(\tau)$ are defined as
\begin{equation}
B(\tau)=\sup\limits_{f\in \mathcal{S}} \beta_f(\tau)\,\, \textup{and}\,\,B_b(\tau)=\sup\limits_{f\in \mathcal{S}_b} \beta_f(\tau).\nonumber
\end{equation}

Because of its important relation to the fine properties of harmonic measures in the complex plane, see \cite{M},  the universal integral means spectra of  univalent functions have been studied extensively in recent years. It is an important and difficult problem to find out the exact values of the universal integral means spectra $B(\tau)$ and $B_b(\tau)$.  We shall first review some main known results and open problems for real $\tau$ in this topic.  It was first observed by Makarov in \cite{M} that
\begin{theorem}\label{mak}
$B(\tau)=\max \{B_b(\tau), 3\tau-1\},\, \tau\in \mathbb{R}.$
\end{theorem}

For large $\tau$, Feng and MacGregor proved in \cite{FM}  that $
B(\tau)=3\tau-1,\,\, \tau\geq \frac{2}{5}.$ Also for large $\tau$, see \cite{Po-2},  we have $B_b(\tau)=\tau-1,\,\, \tau\geq 2.$ By considering the lacunary series,  Kayumov showed in \cite{Kay} that
$ B(\tau)>\frac{\tau^2}{5},\,\, 0<\tau\leq\frac{2}{5}.$ In the paper \cite{CM}, after establishing an important result in the theory of harmonic measures, Carleson and Makarov obtained that
\begin{theorem}\label{CM}
There is a constant $\tau_{*}<0$ such that
$ B_b(\tau)=B(\tau)=|\tau|-1,\,\, \tau\leq \tau_{*}.$
\end{theorem}
It is only known that $\tau_{*}\leq -2$.
It is conjectured that $\tau_{*}=-2$. This is equivalent to the celebrated Brennan conjecture that $B(-2)=B_b(-2)=1$, which was raised in \cite{Br}. There have been many studies on the Brennan conjecture, see for example, \cite{Be}, \cite{Bish}, \cite{BVZ}, \cite{GU-1}, \cite{GU-2}, \cite{HS-1}, \cite{SS}, \cite{To}. The current best upper bound estimate about $B(-2)$ is provided by Hedenmalm and Shimorin in \cite{HS-1}. In the paper \cite{CJ}, Carleson and Jones conjectured that $B_b(1)=\frac{1}{4}$. A more generalized conjecture was given by Kraetzer in \cite{K} that
\begin{conjecture}\label{conj} Let $\tau \in \mathbb{R}$. Then 
\begin{equation}B_b(\tau)=\begin{cases}
\frac{\tau^2}{4}, \;\; \; \;\; \;  \;  \text{if} \;\, \tau \in [-2, 2], \\
|\tau|-1, \; \text{if} \;\, \tau \in (-\infty, -2) \cup (2, +\infty). \nonumber
\end{cases}
\end{equation}
\end{conjecture}
 
For the complex exponents $\tau$, there are only several known results about $B(\tau)$ and $B_b(\tau)$. In an unpublished preprint \cite{Bin-1}, Binder extended Makarov's Theorem \ref{mak} to the complex setting.
\begin{theorem}\label{Binder}
(1) If $\textup{Re}(\tau)>0$, then $B(\tau)=\max\{B_b(\tau), |\tau|+2\textup{Re}(\tau)-1\}.$ 

(2) If $\textup{Re}({\tau})\leq 0$, then $B(\tau)=B_b(\tau)$.  
\end{theorem}

In a 1987 paper \cite{BP-1}, Becker and Pommerenke extended the Brennan conjecture to the complex case, asking whether $B_b(\tau)=1$ whenever $|\tau|=2$.  Baranov and Hedenmalm proved in \cite{BH} that
\begin{theorem}
$$B_b(2-\tau)\leq 1-\textup{Re}(\tau) +[\frac{9e^2}{2}+o(1)]|\tau|^2 \log \frac{1}{|\tau|}, \,\,\, \textup{as}\,\,\, |\tau| \rightarrow 0.$$ 
\end{theorem}

Binder continued to study the integral means spectrum in \cite{Bin-2} for complex exponents and proved the following 
\begin{theorem}For each $\theta \in (-\pi,\pi]$, there exists $T_{\theta}>0$ such that $B_b(te^{i\theta})=t-1$ for $t \geq T_{\theta}$.  
\end{theorem}

Also in \cite{Bin-2}, Binder conjectured that the following generalization of Kraetzer's conjecture still holds. 
\begin{conjecture}Let $\tau \in \mathbb{C}$. Then 
\begin{equation}
B_b(\tau)=\begin{cases}
\frac{\tau^2}{4}, \;\; \; \;\; \;  \;  \text{if} \;\, |\tau|<2, \\
|\tau|-1, \; \text{if} \;\, |\tau|\geq 2.  \nonumber
\end{cases}
\end{equation}
\end{conjecture}
More known results on the universal integral means spectrum and related topics can be found in the monograph \cite[{Chapter VIII}]{GM} and recent survey \cite{HSo} given by Hedenmalm and Sola. 

From the fractal approximation principle(see \cite{M} and \cite{CJ}), we have
\begin{theorem}\label{app}
For each $\tau \in \mathbb{C}$, we have
$
B_b(\tau)=\sup\limits_{f\in \mathcal{S}_q}\beta_f(\tau).
$\end{theorem}
\begin{remark}Here, $\mathcal{S}_q$ is the class of all univalent functions $f$ that belong to $\mathcal{S}_b$ and admit a quasiconformal extension to $\widehat{\mathbb{C}}$.  In the rest of the paper, when a univalent function $f$ belongs to $\mathcal{S}_q$,  we also use $f$ to denote its quasiconformal extension.\end{remark}

There have been some studies on the integral means spectra of univalent functions admitting a quasiconformal extension, see for example \cite{HS-2}, \cite{Hed-1}, \cite{Hed-2}, \cite{Iv}, \cite{Jin-1}, \cite{PS} and \cite{Pr}. 
From Theorem \ref{app}, we can see the problem that determining the exact values of $B_b(\tau)$ as a global extremal problem in the class $\mathcal{S}_q$.  To enrich the understanding of the theory of integral means spectra of univalent functions, in the present paper, we continue to consider the following problem, which differs from previous studies that finding a better estimation for $B_b(\tau)$.
\begin{problem}\label{pro-r}
{\bf (1)} For fixed $\tau \in \mathbb{C}$ with $\tau\neq 0$, 
what topology can be endowed upon the class $\mathcal{S}_q$ such that the integral means spectrum $\beta_f(\tau)$ is continuous on $\mathcal{S}_q$?
{\bf (2)} For given $\tau\neq 0$, does there exist at least one extremal function for $B_b(\tau)$? {\bf (3)} If the extremal functions for $B_b(\tau)$ exist, in which subset of $\mathcal{S}$ do they lie?  \end{problem}
Here and later, we say a function $f\in \mathcal{S}$ is an extremal function for $B_b(\tau)$ if $\beta_f(\tau)=B_b(\tau)$. In the recent work \cite{Jin}, we have studied Problem \ref{pro-r} for real $\tau$. It has been pointed out that the functional $I: f\rightarrow \beta_f(\tau), f\in \mathcal{S}_q$, is not continuous on $\mathcal{S}_q$ under the locally uniformly convergence topology, see \cite[Remark 5.8]{Jin}.  In \cite{Jin}, we introduced and studied the integral means spectrum (IMS) functionals on the Teichm\"uller spaces. It has been proved in \cite{Jin} that all IMS functionals on the universal Teichm\"uller space, the closure of the universal Teichm\"uller curve, and the universal asymptotic Teichm\"uller space are continuous. It was also shown in \cite{Jin} that  the real exponential integral means spectrum of any univalent function admitting a quasiconformal extension to $\widehat{\mathbb{C}}$ is strictly less than the universal integral means spectrum. This result means that, for $\tau\in \mathbb{R}$ with $\tau\neq 0$, the extremal functions for $B_b(\tau)$ cannot be from the class $\mathcal{S}_q$. The first main purpose of this paper is to extend these results to the complex setting. We shall prove that 

\begin{theorem}\label{Em-1}
For each $\tau\in \mathbb{C}$,  the IMS functional $I_{{T}}: [\mu]_T  \mapsto  \beta_{f_{\mu}}(\tau), \,\, [\mu]_T \in {T},$
is continuous.
\end{theorem}

\begin{theorem}\label{Em-2}
For each $\tau\in \mathbb{C}$, the IMS functional $I_{{AT}}: [\mu]_{AT}  \mapsto  \beta_{f_{\mu}}(\tau), \,\, [\mu]_{AT} \in {AT},$ is continuous.
\end{theorem}

\begin{theorem}\label{Em-3}
For each $\tau\in \mathbb{C}$, the IMS functional $I_{\overline{\mathcal{T}}}: \phi  \mapsto \beta_{f_{\phi}}(\tau), \,\, \phi \in \overline{\mathcal{T}},$ is continuous.
\end{theorem} 

\begin{theorem}\label{Em-4}
Let $\tau\in \mathbb{C}$ with $\tau \neq 0$. Then we have  $\beta_{f}(\tau)<B_b(\tau)$ for any $f \in \mathcal{S}_q$.
 \end{theorem}

The rest of paper is organized as follows. In the next section, we will fix some notations used in this paper and recall some basic properties about univalent functions, quasiconformal mappings, and the definitions of quasiconformal Teichm\"uller spaces. Also, Theorem \ref{Em-1}-\ref{Em-4} will be restated in Section 2.  We will provide some lemmas in Section 3.  We shall give the proof of Theorem \ref{Em-1} and \ref{Em-2} in Section 4. We will prove Theorem \ref{Em-3} and \ref{Em-4} in Section 5. In Section 6, we study the integral means spectra of univalent rational functions in $\Delta$ and prove that the Brennan conjecture is true for this class of functions. We also check that the Brennan conjecture holds for any univalent function from the unit disk onto the interior of a bounded polygon.  In Section 7, for IMS functionals on Teichm\"uller spaces, univalent rational functions, and a multiplication operator whose norm is closely related to the Brennan conjecture, we finally present some remarks and raise some problems and conjectures.

\section{Preliminaries and restatement of the first main results}
In this section, we first fix some notations that will be used later, and then recall some basic definitions and properties of univalent functions(conformal mappings), quasiconformal mappings, and quasiconformal Teichm\"uller spaces.  For the main references,  see \cite{Du, Po, L, LV} and \cite{AG1, GN, EGN-1, EMS, EGN-2, L, Zhur}.

We let $\Delta^{*}=\widehat{\mathbb{C}}-\overline{\Delta}$ be the exterior of $\Delta$ and let $\mathbb{T}=\partial\Delta=\partial\Delta^{*}$ to be  the unit circle.  We let $\mathcal{A}(\Delta)$ denote the class of all analytic functions in $\Delta$.  We use the notation $\Delta(r)$ to denote the disk centered at $0$ with radius $r$. For two compact sets $X, Y$ of the complex plane, we define the distance of  $X$ and $Y$, denoted by ${\text{dist}}\{X, Y\}$, as
$
{\textup{dist}}\{X, Y\}:=\min_{x\in X, y\in Y}|x-y|. $
Let $\Omega$ be a simply connected proper subdomain of $\mathbb{C}$.  We shall use $\rho_{\Omega}$ to denote the hyperbolic metric with curvature $-4$ in $\Omega$.  That is $\rho_{\Omega}(w)=|g'(w)|/({1-|g(w)|^2}),\,\, w\in \Omega.$
Here $g$ is a univalent function from $\Omega$ to $\Delta$.  In particular, $\rho_{\Delta}(z)=(1-|z|^2)^{-1}, z\in \Delta.$ 

\subsection{Univalent functions and quasiconformal mappings} 
It is well known that\begin{proposition}\label{pro-1}
Let $f\in \mathcal{U}$. Then, for any $z\in \Delta$,
\begin{equation}  {\textup{dist}}(f(z), \partial f(\Delta))\leq (1-|z|^2)|f'(z)|\leq 4 {\textup{dist}}(f(z), \partial f(\Delta)). \nonumber \end{equation}
\end{proposition}
We let ${E}_j$ be  the Banach space of functions $\phi \in \mathcal{A}(\Delta)$ with the norm
  \begin{equation}
 \|\phi\|_{E_j}:= \sup_{z\in \Delta} |\phi(z)|(1-|z|^2)^{j}<\infty, \,\,\,\, j=1,2. \nonumber
 \end{equation}
 Let $f$ be a locally univalent function in an open domain $\Omega$ of $\mathbb{C}$. The {\em Pre-Schwarzian derivative} $N_f$ of $f$, and the {\em Schwarzian derivative} $S_f$ of $f$ are defined as
 $$N_f(z):=\frac{f''(z)}{{f'(z)}}, \, z\in \Omega,$$ and 
 $$S_f(z):=[N_f(z)]'-\frac{1}{2}[N_f(z)]^2=\frac{f'''(z)}{f'(z)}-\frac{3}{2}\left[\frac{f''(z)}{f'(z)}\right]^2, \,z\in \Omega.$$
Let $g$ be another locally univalent function in $f(\Omega)$. Then we have
\begin{equation}\label{chain}N_{g \circ f}(z)=N_{g}(f(z))[f'(z)]+N_f(z),\,\, z\in \Omega,  \end{equation}
and \begin{equation}\label{chain-1}S_{g \circ f}(z)=S_{g}(f(z))[f'(z)]^2+S_f(z),\,\, z\in \Omega. \end{equation}
It is well known that $$|N_f(z)|(1-|z|^2)\leq 6\,\,{\text {and}}\,\, |S_f(z)|(1-|z|^2)^2\leq 6,$$ for all $f\in \mathcal{U}$.  This means that
  $\|N_f\|_{E_1}\leq 6\,\,{\text {and}}\,\, \|S_f\|_{E_2}\leq 6$ for any $f\in \mathcal{U}$. We define the classes $\mathbf{N}$ and $\mathbf{S}$ as
$$
 \mathbf{N}=\{\phi \in \mathcal{A}(\Delta): \, \phi=N_f(z), f\in \mathcal{S}\}, $$
 and 
$$ \mathbf{S}=\{\phi \in \mathcal{A}(\Delta): \, \phi=S_f(z), f\in \mathcal{S}\}.
$$
Then $\mathbf{N} \subset E_1$ and $\mathbf{S} \subset E_2$. Moreover, we have
\begin{proposition}\label{pro-2}
$\mathbf{N}$ and $\mathbf{S}$ are closed in $E_1$ and  $E_2$, respectively.
\end{proposition}

We say a sense-preserving homeomorphism $f$,  from an open domain $\Omega$ in
$\mathbb{C}$ to another one,  is a quasiconformal mapping if it has locally square
integral distributional derivatives and satisfies the Beltrami equation $\bar{\partial}f=\mu_f  \partial{f}$ with
$$\|\mu_f\|_{\infty}=\mathop{\text {ess sup}}\limits_{z\in \Omega} |\mu_f(z)|<1.$$
Here the function $\mu_f(z)$ is called the {\em  Beltrami coefficient} of $f$ and
$$\bar{\partial}f=f_{\bar{z}} :=\frac{1}{2}\left(\frac{\partial}{\partial x}+i\frac{\partial}{\partial y}\right)f,$$
$$\partial f=f_{z} :=\frac{1}{2}\left(\frac{\partial}{\partial x}-i\frac{\partial}{\partial y}\right)f.$$
Let $f$ be a quasiconformal mapping from one open domain $\Omega_1$ to another domain $\Omega_2$. If $g$ is another quasiconformal mapping from $\Omega_1$ to $\Omega_3$. Then the Beltrami coefficients of $f, g$ and $g \circ f^{-1}$ satisfy the following chain rule.
\begin{equation}\label{dila}
\mu_{g \circ f^{-1}}\circ f(z) =\frac{1}{\chi}\frac{\mu_g(z)-\mu_f(z)}{1-\overline{\mu_f(z)}\mu_g(z)},\,\, \chi=\frac{\overline{\partial f}}{\partial f}, \,\,\, z\in \Omega_1.
\end{equation}
 Let $f$ be a bounded univalent function in a Jordan domain $\Omega$ of $\mathbb{C}$ admitting a quasiconformal extension (still denoted by $f$) to $\widehat{\mathbb{C}}$.  The {\em boundary dilatation} of $f$, denoted by $b(f)$, is defined as
\begin{equation}\label{boun}
b(f):=\inf\{\|\mu_f|_{\Omega^{*}-E}\|_{\infty}: \, E {\text { is a compact set in}}\,\, \Omega^{*}\}. 
\end{equation}
Here $ \Omega^{*}=\widehat{\mathbb{C}}-\overline{\Omega}$ is seen as an open set in the Riemann sphere $\widehat{\mathbb{C}}$ under the spherical distance and $b(f)$ is the infimum of $\|\mu_f|_{\Omega^{*}-E}\|_{\infty}$ over all compact subsets $E$ contained in $\Omega^{*}$. 

\subsection{Quasiconformal Teichm\"uller spaces} 
Let $\Omega$ be a bounded Jordan domain in $\mathbb{C}$ with $0\in \Omega$, and let $\Omega^{*}=\widehat{\mathbb{C}}-\overline{\Omega}$.
We use $M(\Omega^{*})$ to denote the open unit ball of the Banach space $L^{\infty}({\Omega}^{*})$ of essentially bounded measurable functions in ${\Omega}^{*}$. For $\mu \in M({\Omega}^{*})$, let $f_{\mu}$ be the quasiconformal mapping in the extended complex plane $\widehat{\mathbb{C}}$ with complex dilatation equal to $\mu$ in $\Omega^{*}$, equal to $0$ in $\Omega$, normalized $f_{\mu}(0)=0, \, f'_{\mu}(0)=1, \, f_{\mu}(\infty)=\infty$.  We say two elements $\mu$ and $\nu$ in $M(\Omega^{*})$ are equivalent, denoted by $\mu\sim \nu$, if $f_{\mu}|_{\Omega}=f_{\nu}|_{\Omega}$. The equivalence class of $\mu$  is denoted by $[\mu]_{T(\Omega)}$. Then $T(\Omega)=M(\Omega^{*})/\sim$ is one model of Teichm\"uller space of the domain $\Omega$.

For any $\mu, \nu \in M(\Omega^{*})$, we let
$$\sigma(\mu, \nu)(\zeta):=\frac{\mu(\zeta)-\nu(\zeta)}{1-{\overline{\nu(\zeta)}}\mu(\zeta)},\, \zeta\in \Omega^{*}.$$
The (hyperbolic) distance $d_{\Omega}(\mu, \nu)$ of two elements $\mu, \nu$ in $M(\Omega^{*})$ is defined as
$$d_{\Omega}(\mu, \nu):=\frac{1}{2}\log \frac{1+\|\sigma(\mu, \nu)\|_{\infty}}{1-\|\sigma(\mu, \nu)\|_{\infty}}=\tanh^{-1}\|\sigma(\mu, \nu)\|_{\infty}.$$
The Teichm\"uller distance $d_T([\mu]_{T(\Omega)}, [\nu]_{T(\Omega)})$ of two points $[\mu]_{T(\Omega)}$, $[\nu]_{T(\Omega)}$ in $T(\Omega)$ is defined as
\begin{eqnarray}
d_T([\mu]_{T(\Omega)}, [\nu]_{T(\Omega)})&=&\inf\limits_{\mu_1\sim \mu, \nu_1\sim \nu}\Big\{ \tanh^{-1}\|\sigma(\mu_1, \nu_1)\|_{\infty}\Big\}\nonumber \\
&=& \inf\limits_{\mu_1\sim \mu, \nu_1\sim \nu}\Big\{\frac{1}{2}\log \frac{1+\|\sigma(\mu_1, \nu_1)\|_{\infty}}{1-\|\sigma(\mu_1, \nu_1)\|_{\infty}}\Big\}.\nonumber 
\end{eqnarray}

We say $\mu$ and $\nu$ in $M({\Omega}^{*})$ are asymptotically equivalent, denoted by $\mu \approx \nu$, if there exists $\widetilde{\nu}$ in ${M}(\Omega^{*})$ such that $[\widetilde{\nu}]_{T(\Omega)}=[{\nu}]_{T(\Omega)}$ and $\widetilde{\nu}(\zeta)-{\mu}(\zeta) \rightarrow 0$ as $\textup{dist}(\zeta, \partial\Omega)\rightarrow 0^{+}$.  The asymptotic equivalence
of $\mu$ will be denoted by $[\mu]_{AT(\Omega)}$. The {\em asymptotic Teichm\"uller space} $AT(\Omega)$ is the set of
all asymptotic equivalence classes $[\mu]_{AT(\Omega)}$ of elements $\mu$ in $M({\Omega}^{*})$. The Teichm\"uller distance $d_{AT}([\mu]_{AT(\Omega)}, [\nu]_{AT(\Omega)})$ of two points $[\mu]_{AT(\Omega)}$, $[\nu]_{AT(\Omega)}$ in $AT(\Omega)$ is defined as
\begin{eqnarray}
d_{AT}([\mu]_{AT(\Omega)}, [\nu]_{AT(\Omega)})&=&\inf\limits_{\mu_1\approx \mu, \nu_1\approx \nu} \Big\{\tanh^{-1}h^{*}(\sigma(\mu_1, \nu_1)\Big\}\nonumber \\
&=& \inf\limits_{\mu_1\approx \mu, \nu_1\approx \nu}\Big\{\frac{1}{2} \log \frac{1+h^{*}(\sigma(\mu_1, \nu_1))}{1-h^{*}(\sigma(\mu_1, \nu_1))}\Big\}.\nonumber 
\end{eqnarray}
Here, for $\mu \in M(\Omega^{*})$, $h^{*}(\mu)$ is defined as 
\begin{equation}\label{bn}
h^{*}(\mu)=\inf\{\|\mu|_{\Omega^{*}-E}\|_{\infty}: \, E {\text { is a compact set in}}\,\,\Omega^{*}\}.\end{equation}
\begin{remark}
We note that $h^*(\mu)=b(f_{\mu})$ for $\mu \in M(\Omega^{*})$, here $b(f_{\mu})$ is defined as in (\ref{boun}).
\end{remark}
\begin{remark}
We shall mainly consider the Teichm\"uller spaces defined on $\Delta$. We next will use $T$, $AT$ to denote $T(\Delta)$, $AT(\Delta)$, respectively.  $T$ and $AT$ are known as {\em universal Teichm\"uller space} and {\em universal asymptotic Teichm\"uller space}, respectively.  For the sake of simplicity, the equivalence classes $[\mu]_{T(\Delta)}$ and $[\mu]_{AT(\Delta)}$ will be denoted by $[\mu]_T$ and $[\mu]_{AT}$, respectively.
\end{remark}

\begin{remark}Note that for any $f\in \mathcal{S}_q$, we can find a $\mu\in M(\Delta^*)$ such that $f=\sigma \circ f_{\mu}$, $\sigma$ is a Möbius transformation. Then we check from Theorem \ref{app} that $$B_b(\tau)=\sup\limits_{{[\mu]_T}\in T} \beta_{f_{\mu}}(\tau)$$ for each $\tau\in \mathbb{C}$.
\end{remark}

\subsection{Restatement of the first main results} We will study the following IMS functional defined on $T$ and prove that
\begin{theorem}[=\bf{Theorem 1}]
For each $\tau\in \mathbb{C}$, the IMS functional $I_{{T}}: [\mu]_T  \mapsto \beta_{f_{\mu}}(\tau)$ on $T$ is continuous.
\end{theorem}

For the IMS functional on the universal asymptotic Teichm\"uller space,  we shall show that
\begin{theorem}[=\bf{Theorem 2}]
For each $\tau\in \mathbb{C}$, the IMS functional $I_{{AT}}: [\mu]_{AT} \mapsto  \beta_{f_\mu}(\tau)$ is well-defined and continuous on $AT$.
\end{theorem}

We set
 $$\mathcal{T}:=\{\phi: \phi=N_{f}(z), f\in \mathcal{S}_q\}.$$
$\mathcal{T}$ is seen as one model of {\em universal Teichm\"uller curve}, see \cite{Bers, T}.  It is known that $\mathcal{T}$ is an open connected subset of $E_1$, see \cite{Zhur}. Since $\mathbf{N}$ is closed in $E_1$, then the closure $\overline{\mathcal{T}}$ of $\mathcal{T}$ is contained in $\mathbf{N}$.  We view $\overline{\mathcal{T}}$ as a model for the closure of the universal Teichm\"uller curve.
For any $\phi \in \overline{\mathcal{T}}$, there is a unique univalent function $f_{\phi}(z)$ with $f_{\phi}\in \mathcal{S}$ and such that
$\phi(z)=N_{f_{\phi}}(z).$ Actually,  we can take
\begin{equation}\label{exa}f_{\phi}(z)=\int_{0}^{z} e^{\int_{0}^{\zeta}\phi(w) dw}d\zeta, \,\,\,z\in\Delta.\end{equation}

We shall prove that
\begin{theorem}[=\bf{Theorem 3}]
For each $\tau\in \mathbb{C}$, the IMS functional $I_{\overline{\mathcal{T}}}: \phi  \mapsto  \beta_{f_{\phi}}(\tau)$ is continuous  on $\overline{\mathcal{T}}$.
\end{theorem}

Also, we will prove the following result, which indicates that the IMS functional on the class $\mathcal{S}_q$ satisfies the maximum modulus principle. 
\begin{theorem}[=\bf{Theorem 4}]
Let $\tau\in \mathbb{C}$ with $\tau \neq 0$. Then we have  $\beta_{f}(\tau)<B_b(\tau)$ for any $f \in \mathcal{S}_q$.
 \end{theorem}

\section{Five lemmas}

In this section, we will recall three known lemmas and establish two new ones. First, we will use the following criterion for the integral means spectrum,  see \cite{HSo}, \cite{HS-1}, \cite{SS}.
We define the Hilbert space $\mathcal{H}_{\alpha}^2(\Delta)$ as
$$\mathcal{H}_{\alpha}^2(\Delta)=\{\phi \in \mathcal{A}(\Delta) : \|\phi\|_{\alpha}^2:=(\alpha+1)\iint_{\Delta}|\phi (z)|^2(1-|z|^2)^{\alpha}\frac{dxdy}{\pi}<\infty\}.$$
Then
\begin{lemma}\label{cri} 
Let $\alpha>-1$.  For each $\tau\in \mathbb{C}$, we have
$$\beta_{f}(\tau)=\inf\{ \alpha+1: \,\,[f'(z)]^{\tau/2}\in \mathcal{H}_{\alpha}^2(\Delta)\}.$$
\end{lemma}
We need the following result, which is a special case of Proposition 2.14 in \cite{Jin}. 
\begin{lemma}\label{l--1}
Let $f$ belong to $\mathcal{S}_q$ with $f(\infty)=\infty$. Let $\mathbf{h}$ be a bounded univalent function in $f(\Delta)$ with $\mathbf{h}(0)=\mathbf{h}'(0)-1=0$. We assume that $\mathbf{h}$ admits a quasiconformal extension (still denoted by $\mathbf{h}$) to $\widehat{\mathbb{C}}$ with $\mathbf{h}(\infty)=\infty$ and $b(\mathbf{h})=0$.  Then, for any $\varepsilon\in (0, \frac{1}{3}(1-\|\mu_f\|_{\infty}))$, there are two constants $C(f, {\bf{h}})>0$, $\delta>0$ such that
$$ |N_{\mathbf{h}}(\zeta)|{\textup{dist}}(\zeta, f(\mathbb{T}))<C(f, \mathbf{h})\varepsilon,$$
for all  $\zeta\in f(\Delta)$ with ${\textup{dist}}(\zeta, f(\mathbb{T}))<\delta$.
\end{lemma}
We will use the following lemma.
\begin{lemma}\label{l--2}
Let $f$ belong to $\mathcal{S}_q$ with $f(\infty)=\infty$. Assume that $\mathbf{h}$ is a bounded univalent function in $\Omega:=f(\Delta)$ with $\mathbf{h}(0)=\mathbf{h}'(0)-1=0$ and admits a quasiconformal extension (still denoted by $\mathbf{h}$) to $\widehat{\mathbb{C}}$ with $\mathbf{h}(\infty)=\infty$. Then we have
\begin{equation}
|N_{\mathbf{h}}(\zeta)|\rho_{\Omega}^{-1}(\zeta)\leq 8\|\mu_{\mathbf{h}}\|_{\infty},
\end{equation}
for all $\zeta\in \Omega$.
\end{lemma}

\begin{remark}
This lemma can be proved by the arguments in the proof of Proposition 3.3 in \cite{Jin}. 
\end{remark}

The following two lemmas are also needed.
\begin{lemma}\label{key-est}
Let $f, g \in \mathcal{S}$ and let $\tau\in \mathbb{C}$ with $\tau \neq 0$.  If there is a constant $r_0\in (0, 1) $ such that
\begin{equation}\label{condition}\sup\limits_{|z|\in (r_0, 1)}|N_g(z)-N_f(z)|(1-|z|^2) <\varepsilon, \nonumber \end{equation}
for some positive number $\varepsilon$. Then there exist two positive numbers $C_1(r_0, \tau, \varepsilon)$ and $C_2(r_0, \tau, \varepsilon)$ such that
\begin{equation}
C_1(r_0, \tau, \varepsilon)\Big( \frac{1-|z|}{1+|z|}\Big)^{\frac{|\tau|\varepsilon}{2}}\leq |[{\mathbf{h}}' \circ f(z)]^{\tau}| \leq  C_2(r_0, \tau, \varepsilon)\Big( \frac{1+|z|}{1-|z|}\Big)^{\frac{|\tau|\varepsilon}{2}},\nonumber
\end{equation}
for all $|z|\in (r_0,1)$.  Here $\mathbf{h}=g \circ f^{-1}$.\end{lemma}

\begin{proof}
In view of $\mathbf{h}=g \circ f^{-1}$, we obtain from (\ref{chain}) that
 \begin{equation}
N_g(z)-N_f(z)=\frac{{\mathbf{h}}'' \circ f(z) }{{\mathbf{h}}' \circ f(z)} \cdot f'(z).\nonumber
\end{equation}
Let $$G(z):=\frac{{\mathbf{h}}'' \circ f(z)}{{\mathbf{h}}' \circ f(z)} \cdot f'(z),\,\, H(z):={\mathbf{h}}' \circ f(z).$$
It is easy to see that 
$$\log [{\mathbf{h}}' \circ f(z)]^{\tau}=\tau \log H(z)\,\, \textup{and}\,\,  [\log H(z)]'=G(z).$$ 
Let $z=|z|e^{i \arg z}$ be such that $|z|\in (r_0,1)$, then
\begin{equation}\label{ggg-1}
\log [{\mathbf{h}}' \circ f(z)]^{\tau}=\tau \log H(z)=\tau\int_{z_0}^z G(\zeta)d\zeta+ \tau\log H(z_0), 
\end{equation}
here, $z_0=r_0 e^{i \arg z}$ and the integral is taken on the radial path from $z_0$ to $z$. 

On the other hand, note that $$|G(z)(1-|z|^2)| <\varepsilon,$$ for all $|z| \in (r_0, 1)$, hence we obtain that
\begin{eqnarray}\label{ggg-2}
\Big| \int_{z_0}^z G(\zeta)d\zeta\Big| &=& \Big| \int_{r_0}^{|z|} G(t e^{i \arg z}) e^{i \arg z}dt \Big| \\
 &= & \Big| \int_{r_0}^{|z|} G(t e^{i \arg z})(1-t^2) \cdot\frac{e^{i \arg z}}{1-t^2}dt  \Big| \nonumber  \\
 &\leq & \int_{r_0}^{|z|} \frac{\varepsilon}{1-t^2}dt =\frac{\varepsilon}{2}\Big[\log \frac{1+|z|}{1-|z|}-\log \frac{1+r_0}{1-r_0}\Big]. \nonumber \end{eqnarray}
We denote
$$\mathbf{M}_0=\max\limits_{|z|=r_0}| \log  H(z)|=\max\limits_{|z|=r_0} |\log{\mathbf{h}}' \circ f(z)|.$$Then, from the fact $\Big|\log |[{\mathbf{h}}' \circ f(z)]^{\tau}|\Big|\leq \Big|\log [{\mathbf{h}}' \circ f(z)]^{\tau} \Big|$, and (\ref{ggg-1}), (\ref{ggg-2}), we get that
\begin{equation}
\Big|\log |[{\mathbf{h}}' \circ f(z)]^{\tau}|\Big| \leq \frac{|\tau|\varepsilon}{2}\Big[\log \frac{1+|z|}{1-|z|}-\log \frac{1+r}{1-r}\Big]+|\tau|\mathbf{M}_0,\nonumber
\end{equation}
for all $|z|\in (r_0,1)$.  It follows that
\begin{equation}
e^{-|\tau|\mathbf{M}_0}\Big( \frac{1+r_0}{1-r_0}\Big)^{\frac{|\tau|\varepsilon}{2}} \Big( \frac{1-|z|}{1+|z|}\Big)^{\frac{|\tau|\varepsilon}{2}}\leq |[{\mathbf{h}}' \circ f(z)]^{\tau}| \leq  e^{|\tau|\mathbf{M}_0}\Big( \frac{1-r_0}{1+r_0}\Big)^{\frac{|\tau|\varepsilon}{2}} \Big( \frac{1+|z|}{1-|z|}\Big)^{\frac{|\tau|\varepsilon}{2}}, \nonumber
\end{equation}
for all $|z|\in (r_0,1)$.  This proves the lemma.
\end{proof}

\begin{lemma}\label{key-lemma}
Let $f, g \in \mathcal{S}$ and let $\tau \in \mathbb{C}$ with $\tau\neq 0$.  (1) If $\beta_f(\tau):=\beta>0$,  and there is a constant $r_0\in (0, 1) $ such that 
\begin{equation}\sup\limits_{|z|\in (r_0, 1)}|N_g(z)-N_f(z)|(1-|z|^2) <\varepsilon/|\tau|,  \nonumber \end{equation}
for $\varepsilon\in (0, \beta)$, then we have
\begin{equation}
|\beta_g(\tau)-\beta_f(\tau)|\leq \varepsilon. \nonumber
\end{equation}
(2) If $\beta_f(\tau)=0$, and there is a constant $r_0\in (0,1)$ such that
\begin{equation}\sup\limits_{|z|\in (r_0, 1)}|N_g(z)-N_f(z)|(1-|z|^2) <\varepsilon/|\tau|, \nonumber \end{equation}
for $\varepsilon>0$,  then we have $\beta_g(\tau)\leq \varepsilon$.
\end{lemma}

\begin{proof}
(1) Let $\mathbf{h}=g\circ f^{-1}$. First, by Lemma \ref{key-est},  we have
\begin{equation}\label{we-1}
C_1(r_0, \tau, \varepsilon)\Big( \frac{1-|z|}{1+|z|}\Big)^{\frac{\varepsilon}{2}}\leq |{\mathbf{h}}' \circ f(z)| \leq  C_2(r_0, \tau, \varepsilon) \Big( \frac{1+|z|}{1-|z|}\Big)^{\frac{\varepsilon}{2}}, 
\end{equation}
for $|z|\in (r_0,1)$.  On the other hand, since $\beta_{f}(\tau)=\beta >0$, we see from Lemma \ref{cri} that, for $\varepsilon\in (0, \beta)$,
\begin{equation}\label{we-2}
\iint_{\Delta}|[f'(z)]^{\tau}| (1-|z|^{2})^{-1+\beta+\varepsilon/2}dxdy<\infty,
\end{equation}
and
\begin{equation}\label{we-3}
\iint_{\Delta}|[f'(z)]^{\tau}| (1-|z|^{2})^{-1+\beta-\varepsilon/2}dxdy=\infty.
\end{equation}
Hence it follows from the second inequality of (\ref{we-1}) and (\ref{we-2}) that
\begin{eqnarray}
\lefteqn{\iint_{\Delta-\Delta(r_0)}|[g'(z)]^{\tau}| (1-|z|^{2})^{-1+\beta+\varepsilon}dxdy} \nonumber \\&&=\iint_{\Delta-\Delta(r_0)}|[\mathbf{h}'\circ f(z)]^{\tau}| |[f'(z)]^{\tau}| (1-|z|^{2})^{-1+\beta+\varepsilon}dxdy  \nonumber \\
&&\leq [C_2(r_0, \tau, \varepsilon)] \iint_{\Delta-\Delta(r_0)}\Big( \frac{1+|z|}{1-|z|}\Big)^{\varepsilon/2} |[f'(z)]^{\tau}| (1-|z|^{2})^{-1+\beta+\varepsilon}dxdy \nonumber \\
&&\leq 2^{\varepsilon}[C_2(r_0, \tau, \varepsilon)] \iint_{\Delta-\Delta(r_0)} |[f'(z)]^{\tau}| (1-|z|^{2})^{-1+\beta+\varepsilon/2}dxdy<\infty. \nonumber
\end{eqnarray}
Then it is easy to see from Lemma \ref{cri} that $\beta_{g}(\tau)\leq \beta+\varepsilon.$  Also, from the first  inequality of (\ref{we-1}) and (\ref{we-3}), we have
\begin{eqnarray}
\lefteqn{\iint_{\Delta-\Delta(r_0)}|[g'(z)]^{\tau}| (1-|z|^{2})^{-1+\beta-\varepsilon}dxdy} \nonumber \\&&=\iint_{\Delta-\Delta(r_0)}|[\mathbf{h}'\circ f(z)]^{\tau}| |[f'(z)]^{\tau}| (1-|z|^{2})^{-1+\beta-\varepsilon}dxdy  \nonumber \\
&&\geq [C_1(r_0, \tau, \varepsilon)] \iint_{\Delta-\Delta(r_0)}\Big( \frac{1-|z|}{1+|z|}\Big)^{\varepsilon/2} |[f'(z)]^{\tau}| (1-|z|^{2})^{-1+\beta-\varepsilon}dxdy \nonumber \\
&&\geq 2^{-\varepsilon}[C_1(r_0, \tau, \varepsilon)]\iint_{\Delta-\Delta(r_0)} |[f'(z)]^{\tau}| (1-|z|^{2})^{-1+\beta-\varepsilon/2}dxdy=\infty. \nonumber
\end{eqnarray}
This implies that $\beta_{g}(\tau)\geq \beta-\varepsilon$. Hence we have $|\beta_{g}(\tau)-\beta|\leq \varepsilon$.  This proves (1) of Lemma \ref{key-lemma}.

(2) When $\beta_f(\tau)=0$, for $\varepsilon>0$, repeating the above arguments by only using the second inequality of (\ref{we-1}) and (\ref{we-2}), we can prove that $\beta_g(\tau)\leq \varepsilon.$  This proves (2) of Lemma \ref{key-lemma} and the proof of Lemma \ref{key-lemma} is done. \end{proof}

\section{Proof of Theorem \ref{Em-1} and \ref{Em-2}}

Note that $d_{AT}([\mu]_{AT}, [\nu]_{AT}) \leq d_{T}([\mu]_T, [\nu]_T)$ for any $\mu, \nu \in M(\Delta^{*})$, we see that the statement that $I_{AT}$ is continuous on $AT$ implies Theorem \ref{Em-1}. Hence we only need to prove Theorem \ref{Em-2}.  By studying the Pre-Schwarzian derivative model of the universal asymptotic Teichm\"uller space, we have proved in \cite{Jin} that Theorem \ref{Em-2} holds for real $\tau$.  We will use a more direct way to show that Theorem \ref{Em-2} holds for all complex numbers $\tau$. More precisely, we will prove  Theorem \ref{Em-2} by using some arguments from the theory of Teichm\"uller spaces and some results from the extremal theory of quasiconformal mappings.  

\subsection{Isomorphism mapping between Teichm\"uller spaces} We first recall some arguments from the theory of quasiconformal Teichm\"uller spaces. For a fixed $\omega\in M(\Delta^{*})$, let $\mathbf{\Omega}:=f_{\omega}(\Delta)$ and $\mathbf{\Omega}^{*}=\widehat{\mathbb{C}}-\overline{\mathbf{\Omega}}$. For any $\mu \in M(\Delta^{*})$, we define
$$\sharp(\mu)(\zeta):=\mu_{f_{\mu}\circ f_{\omega}^{-1}} (\zeta)=\Big[\frac{1}{\chi} \frac{\mu-\omega}{1-\overline{\omega}\mu}\Big]\circ f_{\omega}^{-1}(\zeta),\,\, \zeta \in \mathbf{\Omega}^{*}.$$
Here $\chi=\frac{\overline{\partial f_{\omega}}}{\partial f_{\omega}}$. We will next write $\mu^{\sharp}(\zeta)=\sharp(\mu)(\zeta)$. Then, from $f_{\mu}=f_{\mu^{\sharp}}\circ f_{\omega}$,  we obtain 
$$\mu(z)=\mu_{f_{\mu^{\sharp}}\circ f_{\omega}}(z)= \frac{{\omega}(z)+(\mu^{\sharp}\circ f_{\omega}(z)) \chi}{1+\overline{\omega(z)}(\mu^{\sharp}\circ f_{\omega}(z)) \chi},\,\, z\in \Delta^{*}.$$
This means that the mapping $\sharp: \mu(z) \mapsto \mu^{\sharp}(\zeta)$ is one-to-one from $M({\Delta}^{*})$ to $M({\mathbf{\Omega}}^{*})$ and,  for any $\nu \in M({\mathbf{\Omega}}^{*})$, 
$$\sharp^{-1}(\nu)(z)=\frac{{\omega}(z)+(\nu \circ f_{\omega}(z)) \chi}{1+\overline{\omega(z)}(\nu \circ f_{\omega}(z)) \chi},\,\,\chi=\frac{\overline{\partial f_{\omega}}}{\partial f_{\omega}},\,\, z\in \Delta^{*}.$$
We shall write $^{\sharp}\nu(z)=\sharp^{-1}(\nu)(z)$ for $\nu \in M({\mathbf{\Omega}}^{*}).$  Meanwhile, it is easy to see that  the mapping $\dagger: [\mu(z)]_{T} \mapsto [\mu^{\sharp}(\zeta)]_{T(\bf{\Omega})}$ is one-to-one from $T$ to $T({\mathbf{\Omega}})$. 

By a direct computation, for any $\mu, \nu \in M(\Delta^{*})$, we have
 \begin{equation}\label{sig-1}
\left|\sigma(\mu, \nu)(z)\right|=\left |\frac{\mu^{\sharp}-\nu^{\sharp}}{1-{\overline{\nu^{\sharp}}\mu^{\sharp}}}\circ f_{\omega}(z)\right |=\left|\sigma(\mu^{\sharp}, \nu^{\sharp})(\zeta)\right|, \, \zeta=f_{\omega}(z),\, z\in \Delta^{*}. 
\end{equation}
Similarly, for any $\mu, \nu \in M(\bf{\Omega}^{*})$, we have
 \begin{equation}\label{sig-2}
\left|\sigma(\mu, \nu)(\zeta)\right|=\left |\frac{^{\sharp}\mu-{^{\sharp}}\nu}{1-{\overline{^{\sharp}\nu}{^{\sharp}}\mu}}\circ f_{\omega}^{-1}(\zeta)\right |=\sigma({^{\sharp}}\mu, {^{\sharp}}\nu)(z), \, z\in \Delta^{*}, \, \zeta=f_{\omega}(z). 
\end{equation}
It follows that
\begin{eqnarray}
d_{T}([\mu]_{T}, [\nu]_T)&=& \inf\limits_{\mu_1\sim \mu, \nu_1\sim \nu}\Big \{\tanh^{-1}\|\sigma(\mu_1, \nu_1)\|_{\infty}\Big\}\nonumber \\
&=& \inf\limits_{\mu_1^{\sharp}\sim \mu^{\sharp}, \nu_1^{\sharp}\sim \nu^{\sharp}} \Big\{\tanh^{-1}\|\sigma(\mu_1^{\sharp}, \nu_1^{\sharp})\|_{\infty}\Big\}\nonumber \\
&=& d_{T}([\mu^{\sharp}]_{T(\mathbf{\Omega})}, [\nu^{\sharp}]_{T(\mathbf{\Omega})}). \nonumber
\end{eqnarray}

We have shown that
\begin{proposition}\label{lll}
The mapping $\dagger: [\mu(z)]_{T} \mapsto [\mu^{\sharp}(\zeta)]_{T(\mathbf{\Omega})}$ is an isometric isomorphism from $T$ to $T({\mathbf{\Omega}})$.
\end{proposition}

It is known that
\begin{proposition}\label{ppp}
Let $\Omega$ be a bounded Jordan domain in $\mathbb{C}$ with $0\in \Omega$. Let $\mu, \nu \in M({\Omega}^{*})$. Then $\mu \approx \nu$ in $M({\Omega}^{*})$, if and only if there exist $\widetilde{\mu}$, $\widetilde{\nu}$ in $M({\Omega}^{*})$ such that $\widetilde{\mu}\sim \mu$, $\widetilde{\nu}\sim \nu$ and $\widetilde{\mu}(\zeta)-\widetilde{\nu}(\zeta) \rightarrow 0$ as $\textup{dist}(\zeta, \partial {\Omega}) \rightarrow 0$. 
\end{proposition}

If $\mu\approx \nu$ in $M(\Delta^{*})$, then, by Proposition \ref{ppp}, there exist $\widetilde{\mu}$, $\widetilde{\nu}$ in $M({\Delta}^{*})$ such that $\widetilde{\mu}\sim \mu$, $\widetilde{\nu}\sim \nu$ in $M(\Delta^{*})$ and $\widetilde{\mu}(z)-\widetilde{\nu}(z) \rightarrow 0$ as $|z| \rightarrow 1^{+}$.  It follows from (\ref{sig-2}) that $\widetilde{\mu}^{\sharp}(\zeta)-\widetilde{\nu}^{\sharp}(\zeta) \rightarrow 0$ as $\textup{dist}(\zeta, \partial {\bf{\Omega}})\rightarrow 0$.  From Proposition \ref{lll}, we see that $\widetilde{\mu}^{\sharp}\sim \mu^{\sharp}$, $\widetilde{\nu}^{\sharp}\sim \nu^{\sharp}$ in $M(\bf{\Omega}^{*})$. By Proposition \ref{ppp} again, we obtain that $\mu^{\sharp}\approx \nu^{\sharp}$ in $M(\bf{\Omega}^{*})$. Similarly, by using Proposition \ref{ppp}, Lemma \ref{lll} and (\ref{sig-1}), we can show that $\mu\approx \nu$ in $M(\bf{\Omega}^{*})$ implies that ${^{\sharp}}\mu\approx {^{\sharp}}\nu$ in $M(\Delta^{*})$.  Moreover, for any $\mu, \nu$ in $M(\Delta^{*})$, we have 
\begin{eqnarray}
d_{AT}([\mu]_{AT},[\nu]_{AT})&=& \inf\limits_{\mu_1\approx \mu, \nu_1\approx \nu}\Big \{\tanh^{-1}\|\sigma(\mu_1, \nu_1)\|_{\infty}\Big\}\nonumber \\
&=& \inf\limits_{\mu_1^{\sharp}\approx \mu^{\sharp}, \nu_1^{\sharp}\approx \nu^{\sharp}} \Big\{\tanh^{-1}\|\sigma(\mu_1^{\sharp}, \nu_1^{\sharp})\|_{\infty}\Big\}\nonumber \\
&=& d_{AT}([\mu^{\sharp}]_{T(\mathbf{\Omega})}, [\nu^{\sharp}]_{T(\mathbf{\Omega})}). \nonumber
\end{eqnarray}

We have shown that 
\begin{proposition}\label{assp}
The mapping $\ddag: [\mu(z)]_{AT} \mapsto [\mu^{\sharp}(\zeta)]_{AT(\mathbf{\Omega})}$ is an isometric isomorphism from $AT$ to $AT({\mathbf{\Omega}})$.
\end{proposition}
\begin{remark}\label{ass-l}
In particular, for $\nu \in M(\Delta^{*})$, we see that
$\nu \approx \omega$ in $M(\Delta^{*})$ if and only if ${\nu}^{\sharp} \approx 0$ in $M(\mathbf{\Omega}^{*})$, and 
$$
d_{AT}([\omega]_{AT},[\nu]_{AT})=d_{AT}([0]_{AT(\mathbf{\Omega})}, [\nu^{\sharp}]_{AT(\mathbf{\Omega})}). 
$$
\end{remark}

\subsection{Extremal theory of quasiconformal mappings} We recall some definitions and results on the extremal theory of quasiconformal mappings. For a bounded Jordan domain $\Omega$  in $\mathbb{C}$ with $0\in \Omega$, let $\mu\in M(\Omega)$. We define
$$k_0([\mu]_{T(\Omega)}) = \inf\{\|\nu\|_{\infty} : \nu \sim \mu\},$$
and
$$h_0([\mu]_{T(\Omega)})=\inf\{h^{*}(\nu): {\nu} \sim \mu\}.$$
Here $h^{*}(\nu)$ is defined as in (\ref{bn}).  It is obvious that $h_0([\mu]_{T(\Omega)})\leq k_0([\mu]_{T(\Omega)})$.  We say $[\mu]_T$ is a {\em Strebel point} if $h_0([\mu]_{T(\Omega)})<k_0([\mu]_{T(\Omega)})$, otherwise, it is called a {\em non-Strebel point}.
Let $$H_0([\mu]_{AT(\Omega)})=\inf\{h^{*}(\nu): {\nu} \approx \mu\}.$$
It is known that $H_0([\mu]_{AT(\Omega)})=h_0([\mu]_{T(\Omega)})$.  We say some $\widetilde{\mu}\in [\mu]_{AT({\Omega})}$ is {\em non-Strebel extremal} in $[\mu]_{AT({\Omega})}$ if it satisfies that \[H_0([\mu]_{AT(\Omega)})=h_0([\widetilde{\mu}]_{T(\Omega)})=k_0([\widetilde{\mu}]_{T(\Omega)})=\|\widetilde{\mu}\|_{\infty}.\]  It follows that
\[H_0([\mu]_{AT(\Omega)})=h_0([\widetilde{\mu}]_{T(\Omega)})=k_0([\widetilde{\mu}]_{T(\Omega)})=h^{*}(\widetilde{\mu})=\|\widetilde{\mu}\|_{\infty}\] when $\widetilde{\mu}$ is non-Strebel extremal in $[\mu]_{AT({\Omega})}$. From \cite[Lemma 4.1]{Yao}, we know that
\begin{proposition}\label{Y}
For any $\mu\in M({{\Omega}}^{*})$, there always exists some $\widetilde{\mu}$ which is non-Strebel extremal in $[\mu]_{AT({{\Omega}})}$.
\end{proposition}

\subsection{A claim} Let $\Omega$ be a bounded Jordan domain in $\mathbb{C}$ with $0\in \Omega$. For any $\mu, \nu\in M(\Omega^{*})$, we let  
$${{d}}^{*}([\mu]_{T(\Omega)}, [\nu]_{T(\Omega)}):=\inf\limits_{{\mu_1\sim \mu}, \nu_1\sim\nu}\Big\{\tanh^{-1}h^{*}(\sigma(\mu_1, \nu_1))\Big\},$$  
$${\bf{d}}^{*}([\mu]_{AT(\Omega)}, [\nu]_{T(\Omega)}):=\inf\limits_{{\mu_1\approx\mu}, \nu_1\sim\nu}\Big\{\tanh^{-1}h^{*}(\sigma(\mu_1, \nu_1))\Big\}.$$  
We first note that $${{d}}^{*}([\mu]_{T(\Omega)}, [\nu]_{T(\Omega)})=\inf\limits_{{\mu_1\sim \mu}}\Big\{\tanh^{-1}h^{*}(\sigma(\mu_1, \nu))\Big\},$$
so that $${\bf{d}}^{*}([\mu]_{AT(\Omega)}, [\nu]_{T(\Omega)})=\inf\limits_{{\mu_1\approx\mu}}\Big\{\tanh^{-1}h^{*}(\sigma(\mu_1, \nu))\Big\}.$$ 

Also, for any $\nu_1, \nu_2$ with $\nu_1 \approx \nu_2 $ in $M(\Omega^{*})$, we observe that 
$${\bf{d}}^{*}([\mu]_{AT(\Omega)}, [\nu_1]_{T(\Omega)})={\bf{d}}^{*}([\mu]_{AT(\Omega)}, [\nu_2]_{T(\Omega)}).$$  
It follows that
\begin{eqnarray}\label{mu-2}d_{AT}([\mu]_{AT(\Omega)}, [\nu]_{AT(\Omega)})&=&\inf\limits_{\mu_1\approx \mu, \nu_1\approx \nu}\Big\{\tanh^{-1}h^{*}(\sigma(\mu_1, \nu_1))\Big\}\nonumber\\
&=& {\bf{d}}^{*}([\mu]_{AT(\Omega)}, [\nu]_{T(\Omega)})\nonumber\\
&=&\inf\limits_{\mu_1\approx\mu}\Big\{\tanh^{-1}h^{*}(\sigma(\mu_1, \nu))\Big\}
\nonumber\\
&=&\inf\limits_{\nu_1\approx\nu}\Big\{\tanh^{-1}h^{*}(\sigma(\mu, \nu_1))\Big\}.
\end{eqnarray} 

Now, given fixed $\omega\in M(\Delta^{*})$, let still $\mathbf{\Omega}=f_{\omega}(\Delta)$ and $\mathbf{\Omega}^{*}=\widehat{\mathbb{C}}-\overline{\mathbf{\Omega}}$. For any $\nu\in M(\Delta^{*})$, let $\widehat{{\nu}}$ be non-Strebel extremal in $[{\nu}^{\sharp}]_{AT(\mathbf{\Omega})}$, and let $d:=d_{AT}([\omega]_{AT}, [\nu]_{AT})$. Then from (\ref{mu-2}) and Proposition \ref{Y}, we have
$$
d_{AT}([0]_{AT(\bf {\Omega)}}, [\nu^{\sharp}]_{AT(\bf{\Omega})})=d_{\bf{\Omega}}(0,\widehat{{\nu}})=\tanh^{-1}\|\widehat{{\nu}}\|_{\infty}. 
$$
On the other hand, from Remark \ref{ass-l} of Proposition \ref{assp},  we know that 
$$
d=d_{AT}([\omega]_{AT}, [\nu]_{AT})=d_{AT}([0]_{AT(\bf {\Omega)}}, [\nu^{\sharp}]_{AT(\bf{\Omega})})=\tanh^{-1}\|\widehat{{\nu}}\|_{\infty}. 
$$
Consequently, from (\ref{sig-2}), we have $$d_{\Delta}(\omega, \widetilde{\nu})=d_{\bf{\Omega}}(0,\widehat{{\nu}})=\tanh^{-1}\|\widehat{{\nu}}\|_{\infty}=d,$$
so that   
$$d=d_{\Delta}(\omega, \widetilde{\nu})=\tanh^{-1}\|\sigma(\omega, \widetilde{\nu})\|_{\infty}.$$
Here $\widetilde{\nu}={^{\sharp}}\widehat{{\nu}}\in [\mu]_{AT}$, and $f_{\widetilde{\nu}}= f_{\widehat{{\nu}}}\circ f_{\omega}$. Replace $\omega$ by $\mu$ in the above arguments, we obtain that

\begin{claim}\label{cla-1}
For any $\mu, \nu \in M(\Delta^{*})$, let $d:=d_{AT}([\mu]_{AT},[\nu]_{AT})$.  Then there is a $\widetilde{\nu}\in [\nu]_{AT}$ such that $d=d_{\Delta}(\mu, \widetilde{\nu})=\tanh^{-1}\|\sigma({\mu}, \widetilde{\nu})\|_{\infty},$
and $b({\mathbf{h}})=\|\mu_{\mathbf{h}}\|_{\infty}$. Here, $\mathbf{h}=f_{\widetilde{\nu}}\circ f_{\mu}^{-1}$.
\end{claim}

\subsection{Proof of Theorem \ref{Em-2}} We shall first prove the following result, which implies that the IMS functional $I_{AT}$ is well-defined.   
\begin{proposition}\label{well}
Let $\mu, \nu \in {M}({\Delta}^{*})$. For each $\tau\in \mathbb{C}$,  if $\mu \approx \nu$ in ${M}({\Delta}^{*})$, then $\beta_{f_{\mu}}(\tau)=\beta_{f_{\nu}}(\tau)$.
\end{proposition}

\begin{remark}
In particular, $\beta_{f_{\mu}}(\tau)=0$ for any $\tau\in\mathbb{C}$  if $\mu\approx 0$ in $M(\Delta^{*})$. We call $f_{\mu}$ is an asymptotically conformal mapping in $\Delta$ when $\mu\approx 0$ in $M(\Delta^{*})$. In general, we say a function $f\in \mathcal{S}_b$ is an {\em asymptotically conformal mapping} in $\Delta$ if it admits a quasiconformal extension (still denoted by $f$) to $\widehat{\mathbb{C}}$ with $\mu_f(z) \rightarrow 0$ as $|z|\rightarrow 1^{+}$, see \cite{GS}, \cite{Po-1}, \cite{Po-3}. From Proposition \ref{well}, we see that $$B_{b}(\tau)=\sup\limits_{[\mu]_{AT}\in AT}\beta_{f_{\mu}}(\tau)$$ for each $\tau\in \mathbb{C}$.\end{remark}

\begin{proof}[Proof of Proposition \ref{well}]
The case $\tau=0$ is obvious, we will only consider $\tau\neq 0$.  When $\beta_{f_{\mu}}(\tau)=\beta >0$.
Since $\mu \approx \nu$, then we  know there is a $\widetilde{\nu}$ such that $\widetilde{\nu} \sim \nu$ in ${M}({\Delta}^{*})$ and $\mu(z)-\widetilde{\nu}(z) \rightarrow 0$ as $|z| \rightarrow 1^{+}$. Let  $\mathbf{h}=f_{\widetilde{\nu}}\circ f_{\mu}^{-1}$. Then from (\ref{dila}), we see that $b(\mathbf{h})=0$. It follows from Lemma \ref{l--1} and Proposition \ref{pro-1} that 
\begin{eqnarray}
\|N_{f_{\mu}}(z)-N_{f_{\nu}}(z)\|_{E_1}&=&\|N_{f_{\mu}}(z)-N_{f_{\widetilde{\nu}}}(z)\|_{E_1} \nonumber \\
&=& |(N_{\mathbf{h}}\circ f_{\mu}(z))|\cdot |f'_{\mu}(z)|(1-|z|^2) \nonumber \\
&\leq&4|(N_{\mathbf{h}}\circ f_{\mu}(z))|\textup{dist}(f_{\mu}(z), f_{\mu}(\mathbb{T}))\rightarrow 0,\,\, {\text as}\,\,\, |z| \rightarrow 1^{-}. \nonumber
\end{eqnarray}
Hence,  for any $\varepsilon\in (0, \beta)$, there is an $r_0\in (0,1)$ such that  $$\sup\limits_{|z|\in (r_0, 1)}|N_{f_{\mu}}(z)-N_{f_{\nu}}(z)|(1-|z|^2) <\varepsilon/|\tau|.$$
From (1) of Lemma \ref{key-lemma}, we obtain that $|\beta_{f_{\mu}}(\tau)-\beta_{f_{\nu}}(\tau)|\leq \varepsilon.$ This means that $\beta_{f_{\mu}}(\tau)=\beta_{f_{\nu}}(\tau)$.
When $\beta_{f_{\mu}}(\tau)=0$. For any $\varepsilon>0$, by using (2) of Lemma \ref{key-lemma}, we can similarly prove that $\beta_{f_{\nu}}(\tau)\leq \varepsilon.$ This implies that $\beta_{f_{\nu}}(\tau)=0$.
The proposition is proved.
\end{proof}

We proceed to prove Theorem \ref{Em-2}. First, it is easy to see that Theorem \ref{Em-2} holds for $\tau=0$. We next will assume that $\tau\neq 0$.  For any $\mu \in M(\Delta^{*})$, we first consider the case when $\beta_{f_{\mu}}(\tau)=\beta>0$. To prove the functional $I_{AT}$ is continuous on $[\mu]_{AT}$, it suffices to show that, for any small $\varepsilon>0$, there is a positive constant $\delta>0$ such that  
$|\beta_{f_\mu}(\tau)-\beta_{f_{\nu}}(\tau)|\leq \varepsilon,$
for all $[\nu]_{AT}$ with $d_{AT}([\mu]_{AT}, [\nu]_{AT})<\delta.$ 

For given $\mu \in {M}(\Delta^{*})$, let $\nu\in M(\Delta^{*})$ with $d:=d_{AT}([\mu]_{AT},[\nu]_{AT})$. Then,
from Claim \ref{cla-1}, we know that there is a $\widetilde{\nu}\in [\nu]_{AT}$ such that $d=d_{\Delta}(\mu, \widetilde{\nu})=\tanh^{-1}\|\sigma({\mu}, \widetilde{\nu})\|_{\infty},$
and $b({\mathbf{h}})=\|\mu_{\mathbf{h}}\|_{\infty}$. Here, $\mathbf{h}=f_{\widetilde{\nu}}\circ f_{\mu}^{-1}$.  It follows that 
\begin{equation}b(\mathbf{h})=\|\mu_{\mathbf{h}}\|_{\infty}=\|\sigma({\mu}, \widetilde{\nu})\|_{\infty}=\tanh d\leq d,\nonumber \end{equation}
and 
\begin{equation}
N_{f_{\widetilde{\nu}}}(z)-N_{f_{\mu}}(z)=N_{\mathbf{h}}\circ f_{\mu}(z)\cdot f_{\mu}'(z), \, z\in \Delta.\nonumber \end{equation}
Then, by using Lemma \ref{l--2} with respect to $f_{\mu}$ and $\mathbf{h}$, we see that 
\begin{eqnarray}\lefteqn{\sup\limits_{z\in \Delta} |N_{f_{\widetilde{\nu}}}(z)-N_{f_{\mu}}(z)|(1-|z|^2)}\nonumber \\
&&=\sup\limits_{z\in \Delta} |N_{\mathbf{h}}\circ f_{\mu}(z)|\cdot |f_{\mu}'(z)|(1-|z|^2)\nonumber \\
&&=\sup\limits_{\zeta\in f_{\mu}(\Delta)}|N_{\mathbf{h}}(\zeta)|\rho_{f_{\mu}(\Delta)}^{-1}(\zeta)\leq 8\|\mu_{\mathbf{h}}\|_{\infty}.\nonumber 
\end{eqnarray}
Consequently, for any small $\varepsilon>0$, we take $\delta=\frac{\varepsilon}{8|\tau|}$, then,  when $$d=d_{AT}([\mu]_{AT},[\nu]_{AT})<\delta,$$ we obtain that $\|\mu_{\mathbf{h}}\|_{\infty}=b(\bf{h})<\delta$ and 
\begin{eqnarray}\sup\limits_{z\in \Delta}|N_{f_{\widetilde{\nu}}}(z)-N_{f_{\mu}}(z)|(1-|z|^2)\leq 8\|\mu_{\mathbf{h}}\|_{\infty}<\varepsilon/|\tau|.\nonumber 
\end{eqnarray} 
It follows from (1) of Lemma \ref{key-lemma} that 
$$|\beta_{f_\mu}(\tau)-\beta_{f_{\nu}}(\tau)|=|\beta_{f_\mu}(\tau)-\beta_{f_{\widetilde{\nu}}}(\tau)|\leq \varepsilon.$$
This proves the functional $I_{AT}$ is continuous on $[\mu]_{AT}$ when $\beta_{f_{\mu}}(\tau)>0$.   

We now consider the case when $\beta_{f_{\mu}}(\tau)=0$. For any small $\varepsilon>0$,  by using the same arguments above, we still take $\delta=\frac{\varepsilon}{8|\tau|}$, then, when $d=d_{AT}([\mu]_{AT},[\nu]_{AT})<\delta$, we can obtain that 
$\|\mu_{\mathbf{h}}\|_{\infty}<\delta$ and so that 
\begin{eqnarray}\sup\limits_{z\in \Delta}|N_{f_{\widetilde{\nu}}}(z)-N_{f_{\mu}}(z)|(1-|z|^2)<\varepsilon/|\tau|,\nonumber 
\end{eqnarray}
Hence from (2) of Lemma \ref{key-lemma}, we have $\beta_{f_{\nu}}(\tau)\leq \varepsilon$.  This implies that $I_{AT}$ is continuous on $[\mu]_{AT}$ in this case. The proof of Theorem \ref{Em-2} is complete.

\section{Proof of Theorem \ref{Em-3} and \ref{Em-4}}

\subsection{Proof of Theorem \ref{Em-3}} Since $\beta_{f_{\mu}}(\tau)=0$ for any $M(\Delta^{*})$ when $\tau=0$, then, obviously, Theorem \ref{Em-3} holds for the case $\tau=0$.  So we will assume that $\tau\neq0$.  For any $\phi\in \overline{\mathcal{T}}$, we will take $f_{\phi}$ as in (\ref{exa}). For given $\psi\in  \overline{\mathcal{T}}$. When $\beta_{f_{\psi}}(\tau)=\beta>0$,  to prove $I_{\overline{\mathcal{T}}}$ is continuous at the point $\psi \in  \overline{\mathcal{T}}$,  it suffices to prove that, for small $\varepsilon>0$, we can find a constant $\delta>0$ such that $|\beta_{f_{\phi}}(\tau)-\beta_{f_{\psi}}(\tau)|\leq \varepsilon,$
for any $\phi \in\overline{\mathcal{T}}$ with $\|\phi-\psi\|_{E_1}<\delta$.

Actually,  for any $\varepsilon\in (0, \beta)$, we can take $\delta=\varepsilon/|\tau|$. Now, we assume that $\phi \in\overline{\mathcal{T}}$ satisfy that $\|\phi-\psi\|_{E_1}<\delta=\varepsilon/|\tau|$. Then we can choose a constant $r_0\in (0,1)$ such that 
\begin{equation}
\sup\limits_{|z|\in (r_0,1)} |N_{f_{\phi}}(z)-N_{{f}_{\psi}}(z)|(1-|z|^2)<\varepsilon/|\tau|. \nonumber
\end{equation}
Consequently, by (1) of Lemma \ref{key-lemma}, we obtain that  $|\beta_{f_{\phi}}(\tau)-\beta_{f_{\psi}}(\tau)|\leq \varepsilon.$
This proves that $I_{\overline{\mathcal{T}}}$ is continuous at the point $\psi \in \overline{\mathcal{T}}$ when $\beta_{f_{\psi}}(\tau)>0$.

When $\beta_{f_{\psi}}(\tau)=0$, for any $\varepsilon>0$, we still take $\delta=\varepsilon/|\tau|$. Similarly, by using (2) of Lemma \ref{key-lemma}, we can obtain that $\beta_{f_{\phi}}(\tau)\leq \varepsilon$ for any $\phi \in\overline{\mathcal{T}}$ with  $\|\phi-\psi\|_{E_1}<\delta$.  This implies that $I_{\overline{\mathcal{T}}}$ is continuous at the point $\psi \in \overline{\mathcal{T}}$ when $\beta_{f_{\psi}}(\tau)=0$. Now, we finish the proof of Theorem \ref{Em-3}.

\subsection{Proof of Theorem \ref{Em-4}}

To prove Theorem \ref{Em-4}, we first show the following result. 
\begin{proposition}\label{m-l}
Let $f\in \mathcal{S}_q$. If $\beta_{f}(\tau_0)>0$ for some $\tau_0\in \mathbb{C}$,  then the function $\gamma(t):=\beta_{f}(\tau_0+t\tau_0)$, $t\in \mathbb{R}_{+}= [0, +\infty)$ is strictly increasing on $\mathbb{R}_{+}$. 
\end{proposition}
\begin{remark}\label{remark}
Let $f\in \mathcal{S}_q$. From Proposition \ref{m-l}, we easily see that, for any fixed $\tau \in \mathbb{C}$, the function $\gamma_{\tau}(t):=\beta_{f}(\tau+t\tau)$ 
is nondecreasing on $\mathbb{R}_{+}$. 
\end{remark}

\begin{proof}
We let $\beta_{f}(\tau_0)=\beta>0$.  Let $\varepsilon<\beta$ be a positive number, which will be fixed later. Then we see from the definition of integral means spectrum of $f$ that  
there is a sequence $\{{\bf r}_n\}_{n=1}^{\infty}$  with ${\bf r}_n<1$ and ${\bf r}_n \rightarrow 1$ as $n\rightarrow \infty$, and such that
\begin{equation}
\int_{-\pi}^{\pi} |[f'({\bf r}_ne^{i\theta})]^{\tau_0}| d\theta > (1-{\bf r}_n)^{-(\beta-\varepsilon)}.\nonumber
\end{equation}
We set

$$\mathbf{A}_n:=\int_{-\pi}^{\pi} |[f'({\bf r}_ne^{i\theta})]^{\tau_0}| d\theta, \, \,\mathbf{D}_n:= (1-{\bf r}_n)^{-(\beta-\varepsilon)},\,\, n\in \mathbb{N},$$
and define 

$${\mathbf{E}}_n:=\{\theta:  |[f'({\bf r}_ne^{i\theta})]^{\tau_0}|> \frac{\mathbf{D}_n}{2\pi}, \, \theta\in (-\pi, \pi]\}, $$
$$\mathbf{F}_n:=\{\theta:  |[f'({\bf r}_ne^{i\theta})]^{\tau_0}|\leq \frac{\mathbf{D}_n}{2\pi},\, \theta\in (-\pi, \pi]\}.$$
It is obvious that $\mathbf{E}_n \bigcup \mathbf{F}_n=(-\pi, \pi]$ and $\mathbf{E}_n \bigcap \mathbf{F}_n=\emptyset.$ We denote 

$$\mathcal{I}_{\mathbf{E}}=\int_{\mathbf{E}_n}  |[f'({\bf r}_ne^{i\theta})]^{\tau_0}| d\theta, \,\, \mathcal{I}_{\mathbf{F}}=\int_{\mathbf{F}_n} |[f'({\bf r}_ne^{i\theta})]^{\tau_0}| d\theta.$$
Then we see that $\mathcal{I}_{\mathbf{E}}\geq \frac{1}{2}\bf{A}_n$ or $\mathcal{I}_{\mathbf{F}}\geq \frac{1}{2}\bf{A}_n$. 

{\bf {Case I.}} If $\mathcal{I}_{\mathbf{E}}\geq \frac{1}{2}\bf{A}_n$, then, for any $t>0$,  we have 
\begin{eqnarray}
\int_{-\pi}^{\pi}|[f'({\bf r}_ne^{i\theta})]^{\tau_0+t\tau_0}| d\theta &=& \int_{-\pi}^{\pi}|[f'({\bf r}_ne^{i\theta})]^{\tau_0}|\cdot |[f'({\bf r}_n e^{i \theta})]^{\tau_0}|^t d\theta  \nonumber \\
& \geq &  \int_{\mathbf{E}_n}|[f'({\bf r}_ne^{i\theta})]^{\tau_0}|\cdot |[f'({\bf r}_n e^{i \theta})]^{\tau_0}|^t d\theta \nonumber \\
&>& \Big(\frac{\mathbf{D}_n}{2\pi}\Big)^{t} \int_{\mathbf{E}_n}|[f'({\bf r}_ne^{i\theta})]^{\tau_0}|d\theta \nonumber \\
&\geq &\frac{1}{2} \Big(\frac{\mathbf{D}_n}{2\pi}\Big)^{t}\mathbf{A}_n \nonumber \\
&>& 2^{-1-t}\pi^{-t} [\mathbf{D}_n]^{1+t}. \nonumber
\end{eqnarray}
That is 
\begin{eqnarray}\label{i}
\int_{-\pi}^{\pi}|[f'({\bf r}_ne^{i\theta})]^{\tau_0+t\tau_0}|d\theta&>& 2^{-1-t}\pi^{-t}{(1-{\bf r}_n)^{-(\beta-\varepsilon)(1+t)}}. 
\end{eqnarray}

{\bf {Case II.}} If $\mathcal{I}_{\mathbf{F}} \geq \frac{1}{2}\mathbf{A}_n$, we set 
$$\mathbf{G}_n:=\{\theta:  |[f'({\bf r}_ne^{i\theta})]^{\tau_0}|\leq \frac{1}{8}\frac{\mathbf{D}_n}{2\pi},\, \theta\in (-\pi, \pi]\},$$
$${\mathbf{H}_n}:=\{\theta:  \frac{1}{8}\frac{\mathbf{D}_n}{2\pi}<|[f'({\bf r}_ne^{i\theta})]^{\tau_0}|\leq \frac{\mathbf{D}_n}{2\pi},\,  \theta\in (-\pi, \pi]\}.$$
We easily see that $\mathbf{G}_n \bigcup \mathcal{H}_n=\mathbf{F}_n$ and $\mathbf{G}_n \bigcap \mathbf{H}_n=\emptyset.$ Then we have
\begin{eqnarray}
\int_{\mathbf{H}_n}|[f'({\bf r}_ne^{i\theta})]^{\tau_0}| d\theta&=&\int_{\mathbf{F}_n}|[f'({\bf r}_ne^{i\theta})]^{\tau_0}| d\theta-\int_{\mathbf{G}_n}|[f'({\bf r}_ne^{i\theta})]^{\tau_0}| d\theta \nonumber \\
&\geq & \frac{1}{2}\mathbf{A}_n- 2\pi \frac{1}{8}\frac{\mathbf{D}_n}{2\pi}\nonumber \\
&\geq &  \frac{3}{8}\mathbf{D}_n. \nonumber
\end{eqnarray}
Consequently,  we have 
\begin{eqnarray}
\int_{-\pi}^{\pi}|[f'({\bf r}_ne^{i\theta})]^{\tau_0+t\tau_0}| d\theta &=&\int_{-\pi}^{\pi}|[f'({\bf r}_ne^{i\theta})]^{\tau_0}|\cdot |[f'({\bf r}_n e^{i \theta})]^{\tau_0}|^t d\theta  \nonumber \\
& \geq &  \int_{\mathbf{H}_n}|[f'({\bf r}_ne^{i\theta})]^{\tau_0}|\cdot |[f'({\bf r}_n e^{i \theta})]^{\tau_0}|^t d\theta \nonumber \\
&>& \frac{3}{8}\mathbf{D}_n \left(\frac{\mathbf{D}_n}{16\pi}\right)^{t} \nonumber\\
&>& 16^{-1-t}\pi^{-t}[\mathbf{D}_n]^{1+t}. \nonumber 
\end{eqnarray} 
Therefore,
\begin{eqnarray}\label{ii}
\int_{-\pi}^{\pi}|[f'({\bf r}_ne^{i\theta})]^{\tau_0+t\tau_0}| d\theta&>& 16^{-1-t}\pi^{-t}{(1-{\bf r}_n)^{-(\beta-\varepsilon)(1+t)}}. 
\end{eqnarray}
Thus, it follows from (\ref{i}) and (\ref{ii}) that 
\begin{eqnarray}\label{iii}
\int_{-\pi}^{\pi}|[f'({\bf r}_ne^{i\theta})]^{\tau_0+t\tau_0}| d\theta&>& 16^{-1-t}\pi^{-t}{(1-{\bf r}_n)^{-(\beta-\varepsilon)(1+t)}}.
\end{eqnarray}

Now, for any $t>0$, we take $\varepsilon=\frac{1}{2}\frac{\beta t}{1+t}\in (0, \beta)$, so that
$(\beta-\varepsilon)(1+t)=\beta+\frac{1}{2}\beta t.$
Hence, we see from (\ref{iii}) that 
\begin{eqnarray}
\int_{-\pi}^{\pi}|[f'({\bf r}_ne^{i\theta})]^{\tau_0+t\tau_0}| d\theta&>& 16^{-1-t}\pi^{-t}{(1-{\bf r}_n)^{-(\beta+\frac{1}{2}\beta t)}}. \nonumber 
\end{eqnarray}
This implies that 
$$\beta_{f}(\tau_0+t\tau_0)\geq \beta+\frac{1}{2}\beta t>\beta=\beta_{f}(\tau_0).$$
Proposition \ref{m-l} is proved.  
\end{proof}

We next finish the proof of Theorem \ref{Em-4}.  For $f\in \mathcal{S}_q$, $\tau\in \mathbb{C}$. First, it is easy to see that  Theorem \ref{Em-4} obviously holds if $\beta_{f}(\tau)=0$.  Hence we next assume that $\beta_{f}(\tau)>0$ for some $\tau\neq 0$. 
We define a locally univalent function $\mathbf{h}$ on the domain $\Omega:=f(\Delta)$ as 
\begin{equation}
\mathbf{h}(w)=\int_{0}^{w} [f'(\mathbf{g}(\zeta))]^{\varepsilon}d\zeta,\,\, w\in f(\Delta). \nonumber
\end{equation}
Here $\mathbf{g}=f^{-1}$ and $\varepsilon>0$ is a small number. Then,  ${\bf{h}}(0)={\bf{h}}'(0)-1=0$ and 
\begin{equation}\label{f-d}
\mathbf{h}'\circ f(z)=[f'(z)]^{\varepsilon}, z\in \Delta,
\end{equation}
and 
\begin{equation}\label{s-d}
\mathbf{h}''\circ f(z) \cdot f'(z)=\varepsilon[f'(z)]^{\varepsilon-1}f''(z), z\in \Delta. 
\end{equation}
It follows from (\ref{f-d}) and (\ref{s-d}) that
$$\frac{\mathbf{h}''\circ f(z)}{\mathbf{h}'\circ f(z)}\cdot f'(z)=\varepsilon N_f(z),\, z\in \Delta,$$ so that 
\begin{eqnarray}
|N_{\mathbf{h}}(w)|\rho_{\Omega}^{-1}(w)&=&\Big|\frac{\mathbf{h}''(w)}{\mathbf{h}'(w)} \Big|\frac{1-|\mathbf{g}(w)|^2}{|\mathbf{g}'(w)|}\nonumber \\
&=&\Big|\frac{\mathbf{h}''\circ f(z)}{\mathbf{h}'\circ f(z)}\Big| \cdot |f'(z)|(1-|z|^2)=\varepsilon |N_{f}(z)|\rho_{\Delta}^{-1}(z) \leq 6\varepsilon.  \nonumber
\end{eqnarray}
Set $\mathbf{F}(z)=\mathbf{h}\circ f(z), z\in \Delta$. We see that 
$$\|N_{\bf F}(z)-N_{f}(z)\|_{E_1}=|N_{\mathbf{h}}(w)|\rho_{\Omega}^{-1}(w)\leq 6\varepsilon.$$
Hence we see from \cite{Bec}, \cite{AG1} or \cite{Zhur} that  $\mathbf{F}$ is bounded univalent in $\Delta$ and admits a quasiconformal extension to $\widehat{\mathbb{C}}$ when $\varepsilon$ small enough. Now, we let $\varepsilon$ small enough so that $\mathbf{F}$ belongs to $\mathcal{S}_q$. On the other hand, we note that, for any $r\in (0, 1), \theta\in (-\pi, \pi]$, 
\begin{eqnarray}
[\mathbf{F}'(re^{i\theta})]^{\tau}&=&
 [\mathbf{h}'\circ f(re^{i\theta})]^{\tau} \cdot [f'(re^{i\theta})]^{\tau} \nonumber \\
&=& [f'(re^{i\theta})]^{\varepsilon\tau} \cdot [f'(re^{i\theta})]^{\tau}=[f'(re^{i\theta})]^{\tau+\varepsilon\tau}. \nonumber 
\end{eqnarray}
Thus we have $\beta_{\mathbf{F}}(\tau)$=$\beta_{f}(\tau+\varepsilon\tau)$. It follows from Proposition \ref{m-l} that $\beta_{\mathbf{F}}(\tau)>\beta_{f}(\tau)$, so that $\beta_{f}(\tau)<B_b(\tau)$ for all $f\in \mathcal{S}_q$ when $\beta_{f}(\tau)>0$.  The proof of Theorem \ref{Em-4} is finished.  
 
 \begin{remark}Note that for any $f\in \mathcal{S}_q$, there is always an $F\in \mathcal{S}$ such that $N_{F}=\phi\in \partial\mathcal{T}$ and $\phi=aN_f, a>1$.  We let $\mathbf{h}=F\circ f^{-1}$. Then $\mathbf{h}$ is univalent in $f(\Delta)$ and
 $$(a-1)N_f=N_{F}(z)-N_{f}(z)=\frac{\mathbf{h}''\circ f(z)}{\mathbf{h}'\circ f(z)}\cdot f'(z).$$
Integrate both sides of the above equation from $0$ to $z$, we obtain that
$$(a-1)\log f'(z)=\log \mathbf{h}'\circ f(z).$$
Hence $\mathbf{h}'\circ f(z))=[f'(z)]^{a-1}$. It follows that for any $r\in (0,1), \theta\in (-\pi, \pi]$,
\begin{eqnarray}
[F'(re^{i\theta})]^{\tau}&=&
 [\mathbf{h}'\circ f(re^{i\theta})]^{\tau} \cdot [f'(re^{i\theta})]^{\tau}=[f'(re^{i\theta})]^{a\tau}, \tau\in\mathbb{C}. \nonumber 
\end{eqnarray}
From Remark \ref{remark}, we have $\beta_{F}(\tau)=\beta_{f}(a\tau)\geq \beta_f(\tau)$. This implies that 
$$\sup_{\phi \in \partial\mathcal{T}} \beta_{f_{\phi}}(\tau)\geq \sup_{\phi \in \mathcal{T}} \beta_{f_{\phi}}(\tau)=B_b(\tau).$$ On the other hand, from Theorem \ref{Binder}, we know that $B(\tau)=B_b(\tau)$ if ${\text{Re}}(\tau)\leq 0$. Hence we have $$B_b(\tau)\geq \sup_{\phi \in \partial\mathcal{T}}\beta_{f_{\phi}}(\tau)$$ when ${\text{Re}}(\tau)\leq 0$. We have checked that 
 \begin{corollary}
If ${\text{Re}}{(\tau)}\leq 0$, then $B_b(\tau)=\sup\limits_{\phi\in \partial{\mathcal{T}}}\beta_{f_{\phi}}(\tau).$
 \end{corollary}
\end{remark}

\section{Integral means spectra of the univalent rational functions}

In this section, we study the integral means spectra of univalent rational functions in the unit disk $\Delta$. Since directly studying the integral means spectrum of univalent rational functions in $\Delta$ is inconvenient, we enlarge the scope to the general class of locally univalent functions in $\Delta$. The locally univalent functions allow for the definition of a generalized integral means spectrum, see (\ref{defi-n}) below. In particular, this generalized spectrum coincides with the classical one for the univalent rational functions. We then completely determine the generalized integral means spectra for a certain class of locally functions, thereby proving that univalent rational functions in the unit disk satisfy the Brennan conjecture. We also check that the Brennan conjecture is true for any univalent function from the unit disk onto the interior of a bounded polygon. 

In the rest of the paper, the notation $P_n(z)$ denotes a polynomial of degree $n$, $n\in \mathbb{N}\cup \{0\}$. In particular, $P_n$ is a complex constant when $n=0$. We denote by $\mathcal{R}$ the class of all rational functions $R$ in $\Delta$ with $R(0)=R'(0)-1=0$ and which have no poles in $\Delta$. Let $\Omega$ be a subset of the complex plane. If $R$ is analytic in $\Omega$, we say $R$ has no {\em zeros}, or one zero in $\Omega$ if the equation $R(z)=0$ has no roots, or one root in $\Omega$, respectively. If $R$ is analytic in $\Omega \setminus \Omega_1$, $\Omega_1$ is a finite (or empty) subset of $\Omega$, we will say $R$ has no {\em critical points}, or one critical point in $\Omega$, if there is no point $z\in \Omega$, or one point $z\in \Omega$ such that $R$ is analytic in $z$ and $R'(z)=0$, respectively.  A rational function $R\in \mathcal{R}$ can be uniquely represented as 
$$R(z)=\frac{P_n(z)}{P_m(z)}=\frac{z+a_2z^2+\cdots+a_nz^n}{1+b_1z+\cdots+b_m z^m}.$$
Here, $P_n$ and $P_m$ have no common factors except 1 and $P_m(z)$ has no zeros in $\Delta$.

\subsection{Certain locally univalent functions and the main result of this section}We use $\mathcal{U}_{R}$ to denote the class of all univalent rational functions $R$ belonging to $\mathcal{R}$. 
We denote by $\mathcal{U}_{P}$  the class which consists of all univalent polynomials contained in $\mathcal{U}_{R}$. 
It is easy to see that $\mathcal{U}_P$ is contained in $\mathcal{S}_b$. It should be pointed out that the univalent polynomials are dense in $\mathcal{S}$, see \cite{Du} or \cite{AR}. This fact is one of the reasons for our investigation of the classes  $\mathcal{U}_P$ and $\mathcal{U}_R$. When $R(z)$ belongs to $\mathcal{U}_R$, we know from the Koebe distortion theorem that
\begin{equation}\label{con-1}|R(z)|\leq \frac{|z|}{(1-|z|)^2}.\end{equation}
In view of this fact, we shall consider the following three classes of rational functions contained in $\mathcal{R}$.    
\begin{definition}
${\bullet}$ Class ${\mathcal{L}_{I}}$. The class  ${\mathcal{L}_{I}}$ consists of all functions $R\in \mathcal{R}$ that can be written as
\begin{equation}\label{I-1}
 R(z)=\frac{P_n(z)}{P_m(z)}.
\end{equation}
Here, $P_n$ and $P_m$ have no common factors except $1$, and $P_{m}(z)$ has no zeros in $\overline{\Delta}=\Delta \cup \mathbb{T}$.  It is easy to see that $\mathcal{U}_P \subset \mathcal{L}_{I}.$

${\bullet}$ Class ${\mathcal{L}_{II}}$. The class  ${\mathcal{L}_{II}}$ consists of all functions $R\in \mathcal{R}$ that can be written as
\begin{equation}\label{II-1}
R(z)=\frac{P_n(z)}{\prod\limits_{j=1}^{l}(z-e^{i\theta_j})P_m(z)}.  
\end{equation}
Here, $P_{m}$ has no zeros in $\overline{\Delta}$, $l\in \mathbb{N}, \theta_j \in (-\pi,\pi]$, $\theta_{j_1}\neq \theta_{j_2}$ whenever $j_1\neq j_2$, the numerator and denominator of $R$ have no common factors except $1$. 

${\bullet}$ Class ${\mathcal{L}_{III}}$. The class  ${\mathcal{L}_{III}}$ consists of all functions $R \in \mathcal{R}$ that can be written as
\begin{equation}\label{III-1}
R(z)=\frac{P_n(z)}{\prod\limits_{j=1}^{l}(z-e^{i\theta_j})^2\prod\limits_{k=1}^{t}(z-e^{i\widetilde{\theta}_k})P_m(z)}. 
\end{equation}
Here, $P_{m}$ has no zeros in $\overline{\Delta}$, $l\in \mathbb{N}, \theta_j \in (-\pi,\pi]$, $\theta_{j_1}\neq \theta_{j_2}$ whenever $j_1\neq j_2$, $t\in \mathbb{N}\cup\{0\}$, $\widetilde{\theta}_k \in (-\pi,\pi]$, $\widetilde{\theta}_{k_1}\neq \widetilde{\theta}_{k_2}$ whenever $k_1\neq k_2$, and ${\theta}_{j}\neq \widetilde{\theta}_{k}$ for all $1\leq j \leq l,$ $1\leq k \leq t$. The numerator and denominator of $R$ have no common factors except $1$. In particular, the case $t=0$ means that the denominator of $R$ is $\prod\limits_{j=1}^{l}(z-e^{i\theta_j})^2P_m(z).$
\end{definition}
\begin{remark}
We see that $\mathcal{U}_R \subset \mathcal{L}_{I} \cup  {\mathcal{L}_{II}} \cup  {\mathcal{L}_{III}}$.
\end{remark}

As we shall see later, the integral means spectra of functions $R\in\mathcal{U}_R$ that have no critical points on $\mathbb{T}$, are different from those of functions $R$ that have at least one critical point on $\mathbb{T}$. We will list some examples of univalent rational functions $R(z)$ which has no critical points, or has at least one critical point on $\mathbb{T}$. 
We first consider the univalent polynomials. 
\begin{example}\label{p1}
Take $P_1(z)=z$, then $P'_1(z)=1$ so that $P_1$ has no critical points on $\mathbb{T}$.  
\end{example}

\begin{example}\label{p2}
Take $P_2(z):=z-\frac{1}{2}z^{2}$, then $P'_2(z)=1-z$ so that $P_2$ has one critical point on $\mathbb{T}$.  
\end{example}

\begin{example}\label{p3}
Take $P_3(z):=z-\frac{1}{3}z^{3}$, then $P'_3(z)=(1-z)(1+z)$ so that $P_3$ has two critical points on $\mathbb{T}$.  
\end{example}

\begin{example}\label{p4}
Take $P_5(z)=z-\frac{1}{6}z^3-\frac{1}{10}z^5$, then $P'_5(z)=-\frac{1}{2}(z^2-1)(z^2+2)$ so that $P_5$ has two critical points on $\mathbb{T}$ and two other ones not on $\mathbb{T}$.   
\end{example}

\begin{remark}
$P_1, P_2, P_3, P_5$ are all contained in $\mathcal{U}_P \cap \mathcal{L}_{I}$.
\end{remark}

Next, we consider the univalent rational functions, which are unbounded in $\Delta$.
\begin{example}\label{r1}
Take $R_1(z)=\frac{z}{1-z},$ then $R_1'(z)=\frac{1}{(1-z)^2}$ so that $R_1$ has no critical points on $\mathbb{T}$.
\end{example}

\begin{example}\label{r2}
Take $R_2(z)=\frac{z^2+3z}{3(1-z)},$ then $R_2'(z)=-\frac{(z-3)(z+1)}{3(1-z)^2}$ so that $R_2$ has one critical point on $\mathbb{T}$. 
\end{example}

\begin{example}\label{r3}
Take $R_3(z)=\frac{z-\frac{1}{2}z^2}{(1-z)^2},$ then $R_3'(z)=\frac{1}{(1-z)^3}$ so that $R_3$ has no critical points on $\mathbb{T}$.
\end{example}

\begin{example}\label{r4}
Take $R_4(z)=\kappa(z):=\frac{z}{(1-z)^2},$ here $\kappa$ is the famous Koebe function, then $R_4'(z)=\frac{1+z}{(1-z)^3}$ so that $R_4$ has one critical point on $\mathbb{T}$.
\end{example}

\begin{remark}
Both $R_1$ and $R_2$ belong to $\mathcal{U}_R \cap \mathcal{L}_{II}$ while both $R_3$ and $R_4$ belong to $\mathcal{U}_R \cap \mathcal{L}_{III}$. 
\end{remark}

\begin{remark} When $R \in \mathcal{U}_R$. If $R$ has a critical point $z_1$ on $\mathbb{T}$, because $R$ is univalent, it follows that \begin{equation}\label{con-2}|R'(z)|\geq \frac{1-|z|}{(1+|z|)^3}.\end{equation}
Then $R'(z)$ must have the form
$$R'(z)=(z-z_1)\frac{P_{\bf n}(z)}{P_{\bf m}(z)}.$$
Here, $P_{\bf n}$ and $P_{\bf m}$ have no common factors except 1, $P_{\bf n}(z_1)\neq 0$,  $P_{\bf m}(z)=0$ has no zeros in $\Delta$.  

Generally, if $R$ has at least one critical point on $\mathbb{T}$, then $R'(z)$ can be written as the following unique form.
\begin{equation} R'(z)=\frac{[(z-z_1)(z-z_2)\cdots(z-z_s)]P_{\bf n}(z)}{P_{\bf m}(z)}.\nonumber \end{equation}
Here, $s\in \mathbb{N}$, and $z_j, j=1,2,\cdots, s$ are all distinct critical points of $R$ on $\mathbb{T}$. The numerator and denominator of $R'$ have no common factors except $1$. $P_{\bf n}(z)$ has no zeros on $\overline{\Delta}$, and $P_{\bf m}(z)$ has no zeros in $\Delta$. \end{remark}

Now, in view of above fact, we introduce another subclass of $\mathcal{R}$ as follows.
\begin{definition}For a function $R \in \mathcal{R}$, we say it belongs to the class $\mathcal{R}_O$, if $R'$ can be written uniquely in the following form:
 \begin{equation}\label{resp-1}R'(z)=\frac{[(z-z_1)(z-z_2)\cdots(z-z_s)]P_{\bf n}(z)}{P_{\bf m}(z)}:=\frac{\Pi(z)P_{\bf n}(z)}{P_{\bf m}(z)}.\end{equation}
Here $s\in \mathbb{N}\cup \{0\}$. When $s\geq 1$, $z_j, j=1,2,\cdots, s$ are all distinct critical points on $\mathbb{T}$ of $R$. The numerator and denominator of $R'$ have no common factors except $1$. $P_{\bf n}(z)$ has no zeros  on $\overline{\Delta}$, and $P_{\bf m}(z)$ has no zeros in $\Delta$. The case $s=0$ means that $\Pi(z)\equiv1$ and $R$ has no critical points on $\mathbb{T}$. \end{definition}
\begin{remark}
We see that $\mathcal{R}_O \cap {\mathcal{L}_{I}}$ is non-empty, by Example \ref{p1}-\ref{p4}, and $\mathcal{R}_O \cap {\mathcal{L}_{II}}$ is non-empty, by Example \ref{r1}, \ref{r2}, and  $\mathcal{R}_O \cap {\mathcal{L}_{III}}$ is non-empty, by Example \ref{r3}, \ref{r4}. Also, it is easy to see that $\mathcal{U}_R$ is contained in $\mathcal{R}_O \cap ({\mathcal{L}_{I}} \cup  {\mathcal{L}_{II}} \cup  {\mathcal{L}_{III}}).$  
\end{remark}

\begin{remark}\label{rem-1}For $R\in \mathcal{R}_O$, we let, as in (\ref{resp-1}),
 \begin{equation}R'(z)=\frac{[(z-z_1)(z-z_2)\cdots(z-z_s)]P_{\bf n}(z)}{P_{\bf m}(z)}.\nonumber \end{equation}
Furthermore, 
({\bf a}) If $R$ belongs to the class $\mathcal{R}_{O}\cap\mathcal{L}_{I}$ for some $s\geq 0$,  we let, as in (\ref{I-1}), 
$$R(z)=\frac{P_n(z)}{P_m(z)}.$$ 
Then
\begin{equation}\label{jj}R'(z)=\frac{P'_n(z)P_m(z)-P_n(z)P'_m(z)}{[P_m(z)]^2}=\frac{[(z-z_1)(z-z_2)\cdots(z-z_s)]P_{\bf n}(z)}{P_{\bf m}(z)}.\end{equation}
Here, as is required, the numerator and denominator of the last rational function in (\ref{jj}) have no common factors except $1$ and  $P_{\bf m}(z)$ has no zeros in $\overline{\Delta}.$  

({\bf b}) If $R$ belongs to the class $\mathcal{R}_{O}\cap\mathcal{L}_{II}$ for some $s\ge 0, l\geq 1$, we let, as in (\ref{II-1}), 
$$R(z)=\frac{P_n(z)}{\prod\limits_{j=1}^{l}(z-e^{i\theta_j})P_m(z)}:=\frac{1}{\Pi_1}\frac{P_n(z)}{P_m(z)}.$$ 
Then 
\begin{eqnarray}\label{jj-1}R'(z)&=&\frac{\Pi_1[P'_n(z)P_m(z)-P_{n}(z)P'_m(z)]-\Pi_1'P_n(z)P_m(z)}{\Pi_1^2[P_m(z)]^2}\\
&=&\frac{[(z-z_1)(z-z_2)\cdots(z-z_s)]P_{\bf n}(z)}{\Pi_1^2 P_{b}(z)}. \nonumber \end{eqnarray}
Here, the numerator and denominator of the last rational function in (\ref{jj-1}) have no common factors except $1$ and $P_{\bf m}(z)=\Pi_1^2 P_{b}(z)=\Pi_1^2P_{{\bf m}-2l}(z)$, and $P_b(z)=P_{{\bf m}-2l}(z)$ has no zeros in $\overline{\Delta}$. 

({\bf c}) If $R$ belongs to the class $\mathcal{R}_{O}\cap\mathcal{L}_{III}$  for some $s\geq 0, l\geq 1, t\geq 0$,  we let, as in (\ref{III-1}),  
$$R(z)=\frac{P_n(z)}{\prod\limits_{j=1}^{l}(z-e^{i\theta_j})^2\prod\limits_{k=1}^{t}(z-e^{i\widetilde{\theta}_k})P_m(z)}:=\frac{1}{\Pi_2}\frac{P_n(z)}{P_m(z)}.$$ 
Then 
\begin{eqnarray}\label{jj-2}R'(z)&=&\frac{\Pi_2[P'_n(z)P_m(z)-P_{n}(z)P'_m(z)]-\Pi_2'P_n(z)P_m(z)}{\Pi_2^2[P_m(z)]^2}   \\
&=& \frac{[(z-z_1)(z-z_2)\cdots(z-z_s)]P_{\bf n}(z)}{\prod\limits_{j=1}^{l}(z-e^{i\theta_j})^3\prod\limits_{k=1}^{t}(z-e^{i\widetilde{\theta}_k})^2P_{b}(z)}.\nonumber\end{eqnarray}
Here, the numerator and denominator of the last rational function in (\ref{jj-2}) have no common factors except $1$ and $$P_{\bf m}(z)=\prod\limits_{j=1}^{l}(z-e^{i\theta_j})^3\prod\limits_{k=1}^{t}(z-e^{i\widetilde{\theta}_k})^2P_{b}(z)=\prod\limits_{j=1}^{l}(z-e^{i\theta_j})^3\prod\limits_{k=1}^{t}(z-e^{i\widetilde{\theta}_k})^2P_{{\bf m}-3l-2t}(z),$$ and $P_b(z)=P_{{\bf m}-3l-2t}(z)$ has no zeros in $\overline{\Delta}$.
\end{remark}
Since any $R\in \mathcal{U}_R$ must satisfy both conditions (\ref{con-1}) and (\ref{con-2}), the above analysis leads us to consider the following subclasses $\mathcal{Q}_{I}$, $\mathcal{Q}_{II}$ and $\mathcal{Q}_{III}$ of all locally univalent functions in $\Delta$. 

\begin{definition}Let $R$ be a locally univalent function in $\Delta$, and it is analytic on $\mathbb{T}$ except finitely many singular points.

$\bullet$ We say $R$ belongs to the class $\mathcal{Q}_I$, if $R'$ can be written in the following form
 \begin{equation}\label{new-d-0}R'(z)=\frac{[(z-z_1)(z-z_2)\cdots(z-z_s)]P_{\bf n}(z)}{P_{\bf m}(z)}:=\frac{\Pi(z)P_{\bf n}(z)}{P_{\bf m}(z)},\,z\in \overline{\Delta}.\end{equation}
Here $s\in \mathbb{N}\cup \{0\}$. When $s\geq 1$, $z_j, j=1,2,\cdots, s$ are all distinct critical points on $\mathbb{T}$ of $R$. The numerator and denominator of $R'$ have no common factors except $1$. $P_{\bf n}(z)$ and $P_{\bf m}(z)$ have no zeros in $\overline{\Delta}$. The case $s=0$ means that $\Pi(z)\equiv1$ and $R$ has no critical points on $\mathbb{T}$.

$\bullet$ We say $R$ belongs to the class $\mathcal{Q}_{II}$, if $R'$ can be written in the following form
\begin{eqnarray}\label{new-d-1}
R'(z)=\frac{[(z-z_1)(z-z_2)\cdots(z-z_s)]P_{\bf n}(z)}{\prod\limits_{j=1}^{l}(z-e^{i\theta_j})^2 P_{{\bf m}}(z)},\,z\in \overline{\Delta}. \end{eqnarray}
Here $s\in \mathbb{N}\cup \{0\}, l\in \mathbb{N}$. When $s\geq 1$, $z_j, j=1,2,\cdots, s$ are all distinct critical points on $\mathbb{T}$ of $R$. The numerator and denominator of $R'$ in (\ref{new-d-1}) have no common factors except $1$. $P_{\bf n}(z)$ and $P_{\bf m}(z)$ have no zeros in $\overline{\Delta}$. In particular, $R'$ has no zeros on $\mathbb{T}$ when $s=0$. 

$\bullet$ We say $R$ belongs to the class $\mathcal{Q}_{III}$, if $R'$ can be written in the following form
\begin{eqnarray}\label{new-d-2}
R'(z)&=& \frac{[(z-z_1)(z-z_2)\cdots(z-z_s)]P_{\bf n}(z)}{\prod\limits_{j=1}^{l}(z-e^{i\theta_j})^3\prod\limits_{k=1}^{t}(z-e^{i\widetilde{\theta}_k})^2P_{\bf m}(z)},\,z\in \overline{\Delta}.\end{eqnarray}
Here $s, t\in \mathbb{N}\cup \{0\}, l\in \mathbb{N}$. When $s\geq 1$, $z_j, j=1,2,\cdots, s$ are all distinct critical points on $\mathbb{T}$ of $R$. The numerator and denominator of $R'$ in (\ref{new-d-2}) have no common factors except $1$. $P_{\bf n}(z)$ and $P_{\bf m}(z)$ have no zeros in $\overline{\Delta}$. In particular, $R'$ has no zeros on $\mathbb{T}$ when $s=0$ and $R'$ has no 
poles of order $2$ on $\mathbb{T}$ when $t=0$. 
\end{definition}
\begin{remark}
We see that $\mathcal{U}_R$ is contained in ${\mathcal{Q}_{I}} \cup  {\mathcal{Q}_{II}} \cup  {\mathcal{Q}_{III}}.$  
\end{remark}

We will show the Brennan conjecture $B(-2)=1$ is true for the class $\mathcal{U}_R$. Actually, we shall completely determine the {\em generalized} integral means spectra of functions belonging to the class $({\mathcal{Q}_{I}} \cup  {\mathcal{Q}_{II}} \cup  {\mathcal{Q}_{III}})$ for all real $\tau$. 
Let $\tau\in \mathbb{R}$ and $R$ be a locally univalent function in $\Delta$. We define the {\em generalized} integral means spectrum ${\widetilde{\beta}}_R$ of $R$ as 
\begin{equation}\label{defi-n}
{\widetilde{\beta}}_R(\tau):=\limsup\limits_{r\rightarrow 1^{-}}\frac{\log \int_{-\pi}^{\pi} |R'(re^{i\theta})|^{\tau} d\theta}{|\log(1-r)|}. 
\end{equation}
\begin{remark}
Note that, for any $\tau\in \mathbb{R}$, this definition is the same as in (\ref{defi}) when $R$ belongs to $\mathcal{U}_R$.
\end{remark}
We now present the main result of this section. 
\begin{theorem}\label{l-pro}
Let $\tau\in \mathbb{R}$ and let $R\in  ({\mathcal{Q}_{I}} \cup  {\mathcal{Q}_{II}} \cup  {\mathcal{Q}_{III}})$. 
Then

$(\bf A)$ when $\tau=0$, we have ${\widetilde{\beta}}_{R}(\tau)=0$ for all $R\in   ({\mathcal{Q}_{I}} \cup  {\mathcal{Q}_{II}} \cup  {\mathcal{Q}_{III}})$,

$(\bf B)$ when $\tau<0$, we have 

$({\bf b_1})$ if $R$ has no critical points on $\mathbb{T}$, then ${\widetilde{\beta}}_{R}(\tau)=0$,

$({\bf b_2})$ if $R$ has at least one critical point on $\mathbb{T}$, then  
\begin{equation}{\widetilde{\beta}}_{R}(\tau)=
\begin{cases}
0, \; \;\;\quad\,\,\,\, \text{for} \,\; \tau\in (-1, 0), \\
|\tau|-1, \; \text{for}\, \; \tau\leq -1,
\end{cases} \nonumber
\end{equation}

$(\bf C)$ when $\tau>0$, we have 

$({\bf c_1})$ if $R$ belongs to the class $ \mathcal{Q}_{I}$, then ${\widetilde{\beta}}_{R}(\tau)=0$,

$({\bf c_2})$ if $R$ belongs to the class $\mathcal{Q}_{II}$, then  
\begin{equation}{\widetilde{\beta}}_{R}(\tau)=
\begin{cases}
2\tau-1, \;  \text{for} \,\; \tau> \frac{1}{2}, \\
0, \;\quad\quad\,  \text{for}\, \; \tau\in (0, \frac{1}{2}],
\end{cases} \nonumber
\end{equation}

$({\bf c_3})$ if $R$ belongs to the class $\mathcal{Q}_{III}$, then  
\begin{equation}{\widetilde{\beta}}_{R}(\tau)=
\begin{cases}
3\tau-1, \;  \text{for} \,\; \tau>\frac{1}{3}, \\
0, \;\quad\quad\,  \text{for}\, \; \tau\in (0, \frac{1}{3}],
\end{cases} \nonumber
\end{equation}
\end{theorem}
 
\begin{remark}Since ${\widetilde{\beta}}_R(\tau)=\beta_{R}(\tau)$ for all $R\in \mathcal{U}_R$, $\tau \in \mathbb{R}$, as a direct consequence of Theorem \ref{l-pro}, we have

\begin{corollary}\label{coro-n}
 Let $\tau\in \mathbb{R}$ and $R\in \mathcal{U}_R$.  Then all the statements $(\bf A)$, $(\bf B)$, and $(\bf C)$ in Theorem \ref{l-pro} hold. \end{corollary}

In particular, we see from Corollary \ref{coro-n} that $\beta_{R}(\tau)={\widetilde{\beta}}_{R}(\tau)\leq |\tau|-1$ for all $R\in \mathcal{U}_R$ when $\tau\leq-2$, which implies that 
\begin{corollary}\label{coro-l}
The Brennan conjecture $B(-2)=1$ is true for the class $\mathcal{U}_R$.
\end{corollary}
\end{remark}
\subsection{Proof of Theorem \ref{l-pro}}
To prove Theorem \ref{l-pro}, we need the following lemmas.
\begin{lemma}\label{ll-1}
 Let $\varkappa\in \mathbb{R}, r\in (0, 1)$. For any fixed $\alpha\in (0,{\pi}]$, we have
  
(1) when $\varkappa>1$,  \begin{eqnarray}\label{eqn-2}
\int_{-\alpha}^{\alpha}\frac{d\theta}{|1-re^{i\theta}|^{\varkappa}}\asymp \frac{1}{(1-r)^{\varkappa-1}}, \, \text{as}\,\, r\rightarrow 1^{-},  
\end{eqnarray}

(2) when $\varkappa=1$,  \begin{eqnarray}\label{eqn-2-2}
\int_{-\alpha}^{\alpha}\frac{d\theta}{|1-re^{i\theta}|^{\varkappa}}\asymp \log \frac{1}{1-r}, \, \text{as}\,\, r\rightarrow 1^{-},  
\end{eqnarray}

(3) when $\varkappa<1$,  \begin{eqnarray}\label{eqn-2-3}
\int_{-\alpha}^{\alpha}\frac{d\theta}{|1-re^{i\theta}|^{\varkappa}}\asymp 1, \, \text{as}\,\, r\rightarrow 1^{-}.
\end{eqnarray}
\end{lemma}
Here and later, the notation $A\asymp B$ for $A>0, B>0$ means that there are two constants $C_1>0, C_2>0$ such that $C_1 A \leq B \leq C_2 A$.
\begin{remark}This lemma seems to be known in the literature. For completeness, we will include a detailed proof for it.  \end{remark}
\begin{proof}[Proof of Lemma \ref{ll-1}] We only prove case (1) when $\varkappa>1$, the other two cases can be shown similarly.  
We first prove the case $\alpha\in (0, \frac{\pi}{2}]$. Recall that, for $\theta\in [0,\frac{\pi}{2}]$, 
\begin{equation}\label{cos-e}
1-\frac{\theta^2}{2}\leq \cos\theta\leq 1-\frac{\theta^2}{4}.
\end{equation}
For $\alpha\in (0, \frac{\pi}{2}]$, we have
\begin{eqnarray}
\int_{-\alpha}^{\alpha}\frac{d\theta}{|1-re^{i\theta}|^{\varkappa}}=\int_{-\alpha}^{\alpha}\frac{d\theta}{|1+r^2-2r\cos\theta|^{\frac{\varkappa}{2}}}.\nonumber 
\end{eqnarray}
By (\ref{cos-e}), we obtain that
\begin{eqnarray}\label{ll-eq-1}
\lefteqn{\int_{-\alpha}^{\alpha}\frac{d\theta}{|1-re^{i\theta}|^{\varkappa}}\leq \int_{-\alpha}^{\alpha}\frac{d\theta}{[1+r^2-2r(1-\frac{\theta^2}{4})]^{\frac{\varkappa}{2}}}}\\ \nonumber 
&& =2 \int_{0}^{\alpha}\frac{d\theta}{[(1-r)^2+\frac{r}{2}\theta^2]^{\frac{\varkappa}{2}}} \leq  2\int_{0}^{\alpha}\Big[\frac{(1-r+\sqrt{\frac{r}{2}}\theta)^{2}}{2}\Big]^{-\frac{\varkappa}{2}} d\theta\\ \nonumber 
&& =\frac{2^{1+\varkappa}}{r^{\frac{\varkappa}{2}}}\int_{0}^{\alpha}\Big[\theta+\frac{\sqrt{2}(1-r)}{\sqrt{r}}\Big]^{-\varkappa} d\theta
\asymp \frac{1}{(1-r)^{\varkappa-1}}, \, \text{as}\,\, r\rightarrow 1^{-}, \nonumber
\end{eqnarray}
and \begin{eqnarray}\label{ll-eq-2}
\lefteqn{\int_{-\alpha}^{\alpha}\frac{d\theta}{|1-re^{i\theta}|^{\varkappa}}\geq \int_{-\alpha}^{\alpha}\frac{d\theta}{[1+r^2-2r(1-\frac{\theta^2}{2})]^{\frac{\varkappa}{2}}}}\\
\nonumber 
&& = 2 \int_{0}^{\alpha}\frac{d\theta}{[(1-r)^2+r\theta^2]^{\frac{\varkappa}{2}}} \geq 2\int_{0}^{\alpha}[(1-r)+\sqrt{r}\theta]^{-\varkappa}d\theta\\ \nonumber 
&& =\frac{2}{r^{\frac{\varkappa}{2}}}\int_{0}^{\alpha}(\theta+\frac{1-r}{\sqrt{r}})^{-\varkappa}d\theta\asymp \frac{1}{(1-r)^{\varkappa-1}}, \, \text{as}\,\, r\rightarrow 1^{-}.\nonumber
\end{eqnarray}
Combining (\ref{ll-eq-1}) and (\ref{ll-eq-2}), we get that, when $\varkappa >1$, \begin{eqnarray}\label{fangc-0}
\int_{-\alpha}^{\alpha}\frac{d\theta}{|1-re^{i\theta}|^{\varkappa}}\asymp \frac{1}{(1-r)^{\varkappa-1}}, \, \text{as}\,\, r\rightarrow 1^{-},  
\end{eqnarray}
for any $\alpha\in (0, \frac{\pi}{2}]$. 

We now consider the case $\alpha\in (\frac{\pi}{2}, \pi]$. On the one hand, we have 
\begin{eqnarray}\label{fangc-1}
\int_{-\alpha}^{\alpha}\frac{d\theta}{|1-re^{i\theta}|^{\varkappa}}>\int_{-\frac{\pi}{2}}^{\frac{\pi}{2}}\frac{d\theta}{|1-re^{i\theta}|^{\varkappa}}. 
\end{eqnarray}
On the other hand, when $\theta \in (\frac{\pi}{2}, \alpha]\cup [-\alpha, -\frac{\pi}{2})$, for any $r\in (\frac{1}{2},1)$, we have
\begin{equation}
|1-re^{i\theta}|^2=1+r^2-2r\cos\theta \geq \frac{5}{4}.\nonumber 
\end{equation} 
It follows that, when $r\in (\frac{1}{2},1)$,
\begin{eqnarray}\label{fangc-2}
\int_{-\alpha}^{\alpha}\frac{d\theta}{|1-re^{i\theta}|^{\varkappa}}&=&\int_{-\frac{\pi}{2}}^{\frac{\pi}{2}}\frac{d\theta}{|1-re^{i\theta}|^{\varkappa}}+\int_{\frac{\pi}{2}}^{\alpha}\frac{d\theta}{|1-re^{i\theta}|^{\varkappa}}
+\int_{-\alpha}^{-\frac{\pi}{2}}\frac{d\theta}{|1-re^{i\theta}|^{\varkappa}}\nonumber \\ 
&\leq & \int_{-\frac{\pi}{2}}^{\frac{\pi}{2}}\frac{d\theta}{|1-re^{i\theta}|^{\varkappa}}+(\frac{4}{5})^{\frac{\varkappa}{2}}(2\alpha-\pi).
\end{eqnarray}
Consequently, we see from  (\ref{fangc-1}), (\ref{fangc-2}), and (\ref{fangc-0}) that (\ref{eqn-2}) still hold for any $\alpha\in (\frac{\pi}{2}, \pi]$.
This proves the lemma. \end{proof}

\begin{lemma}\label{ll-2}
Let $\varkappa>0$, $r\in (0,1)$, $n\in \mathbb{N}$ and $z_j=e^{i\arg z_j}$, $j=1,2,\cdots, n$ be distinct points on $\mathbb{T}$. We assume that $g$ is an analytic function in $\Delta$ and satisfies that, 

{\bf(1)} there are two constants $c_0>0$, $r_0\in (0,1)$ such that $|g^{-1}(z)|\leq c_0$ for all $|z|\in (r_0, 1)$, and

{\bf(2)} for each $z_j$, $j=1,2,\cdots, n$, there are constants $r_j\in (0,1)$, $\vartheta_j>0$, $C_j>0$ such that
$|g(z)|\leq C_j$ for all $$z\in \Omega_j:=\{z=re^{i\theta}: r_j<r<1, \arg z_j-\vartheta_j<\theta<\arg z_j+\vartheta_j\},$$
and $\Omega_{j_1} \cap \Omega_{j_2}=\emptyset$ whenever $j_1\neq j_2$, $1\leq j_1, j_2 \leq n$.

We define
$$f(z)=g(z)\prod_{j=1}^{n}(z-z_j), z\in \Delta.$$
Then, as $r\rightarrow 1^{-}$, we have 
\begin{equation}\label{est} \int_{-\pi}^{\pi} \frac{d\theta}{|f(re^{i\theta})|^{\varkappa}}\asymp
\begin{cases}
\frac{1}{(1-r)^{\varkappa-1}}, \; \text{if} \; \varkappa>1, \\
\log \frac{1}{1-r}, \;\;\;  \text{if} \; \varkappa =1,\\
1, \quad\quad\quad\;\, \text{if} \; \varkappa\in (0,1).
\end{cases}
\end{equation}
\end{lemma}

\begin{proof}
We only prove (\ref{est}) for the case when $\varkappa>1$. Other two cases can be proved by the similar way. For $\varkappa >1$. First we note that, for any $0<r<1$,
\begin{eqnarray}\label{eqn--1} 
 \int_{-\pi}^{\pi} \frac{d\theta}{|f(re^{i\theta})|^{\varkappa}}\geq \sum_{j=1}^{n} \int_{-{\alpha}_j+\arg z_j}^{\alpha_j+\arg z_j} \frac{d\theta}{|f(re^{i\theta})|^{\varkappa}}.
\end{eqnarray}
Meanwhile, in view of the conditions satisfied by $g$, we see that, for any $1\leq j \leq n$, it holds that $|g(z)|\leq C^*$ for all  $z\in \widehat{\Omega}_j$. Here $C^*=\max\limits_{1\leq j\leq n}\{C_j\}$ and $r^{*}=\max\limits_{1\leq j\leq n}\{r_j\}$, and
$$\widehat{\Omega}_j=\{z=re^{i\theta}: r^*<r<1, \arg z_j-\vartheta_j<\theta<\arg z_j+\vartheta_j\}.$$
Hence, for each $1\leq j \leq n$, we have that, for any $z=re^{i\theta} \in \widehat{\Omega}_j$, 
\begin{eqnarray}\label{eqn--2} |f(z)|=|f(re^{i\theta})|&=& \prod_{j=1}^{n}|re^{i\theta}-e^{i\arg z_j}||g(re^{i\theta})|\\ \nonumber
&\leq&  C^* \prod_{j=1}^{n}|re^{i\theta}-e^{i\arg z_j}| \\ \nonumber &\leq& 2^{n-1}C^* |re^{i\theta}-e^{i\arg z_j}|.\end{eqnarray}
Then it follows from (\ref{eqn--1}), (\ref{eqn--2}), and (\ref{eqn-2}) that, when $r>r^*$,
\begin{eqnarray}\label{e--1} 
 \int_{-\pi}^{\pi} \frac{d\theta}{|f(re^{i\theta})|^{\varkappa}}&\geq& [2^{n-1}C^*]^{-\varkappa}\sum_{j=1}^{n} \int_{-{\alpha}_j+\arg z_j}^{\alpha_j+\arg z_j} \frac{d\theta}{|re^{i\theta}-e^{i\arg z_j}|^{\varkappa}}\\ 
 &=& [2^{n-1}C^*]^{-\varkappa} \sum_{j=1}^{n} \int_{-{\alpha}_j}^{\alpha_j} \frac{d\theta}{|1-re^{i\theta}|^{\varkappa}} \nonumber \\
 &\asymp& \frac{1}{(1-r)^{\varkappa-1}}, \, \text{as}\,\, r\rightarrow 1^{-}. \nonumber 
\end{eqnarray}

On the other hand, we can write  
\begin{equation}
\prod_{j=1}^{n}(z-z_j)^{-1}=\sum_{j=1}^{n}\frac{A_j}{z-z_j}. \nonumber
\end{equation}
When $r>r_0$, we have  
\begin{eqnarray}\label{e--2} \int_{-\pi}^{\pi} \frac{d\theta}{|f(re^{i\theta})|^{\varkappa}} &=& \int_{-\pi}^{\pi} \Big|\sum_{j=1}^{n}\frac{A_j}{re^{i\theta}-z_j}\Big|^{\varkappa}\frac{d\theta}{|g(z)|^{\varkappa}}\\ \nonumber 
&\leq& {c_0}^{\varkappa}\int_{-\pi}^{\pi}  \Big[\sum_{j=1}^{n}\frac{|A_j|}{|re^{i\theta}-e^{i\arg z_j}|}\Big]^{\varkappa} d\theta \nonumber \\ 
&\leq & {c_0}^{\varkappa} n^{\varkappa-1} {\mathbb{A}}^{\varkappa} \int_{-\pi}^{\pi}\sum_{j=1}^n \frac{d\theta}{|re^{i\theta}-e^{i\arg z_j}|^{\varkappa}}\nonumber \\ 
&\asymp& \frac{1}{(1-r)^{\varkappa-1}}, \, \text{as}\,\, r\rightarrow 1^{-}.\nonumber
\end{eqnarray}
Here ${\mathbb{A}}=\max\limits_{1\leq j\leq n}\{|A_j|\}$. In (\ref{e--2}), we have used the following inequality \cite[Exercises 1.1.4, Page 11]{Gr}, 
$$\Big(\sum_{k=1}^{n}|a_k|\Big)^{p}\leq n^{p-1}\sum_{k=1}^{n}|a_k|^p, a_k\in \mathbb{C}, p>1. $$
Thus (\ref{est}) for the case when $\varkappa>1$ follows from (\ref{e--1}) and (\ref{e--2}). Finally, using the similar arguments above, and by (\ref{eqn-2-2}), (\ref{eqn-2-3}), we can obtain the remaining two cases of (\ref{est}). The lemma is proved.
\end{proof}
We now present the proof of Theorem \ref{l-pro}.
\begin{proof}[Proof of Theorem \ref{l-pro}] 
Let $R\in ({\mathcal{Q}_{I}} \cup  {\mathcal{Q}_{II}} \cup  {\mathcal{Q}_{III}})$. 

{({\bf A})} It is easy to see that Theorem \ref{l-pro} holds for $\tau=0$. 

{({\bf B})} We now consider the case $\tau<0$. 

$({\bf b_1})$ If $R$ has no critical points on $\mathbb{T}$. Then, we see that $|R'|^{\tau}$ is bounded above in a neighborhood of $\mathbb{T}$. Hence we obtain from the definition (\ref{defi-n}) of the generalized integral means spectrum that ${\widetilde{\beta}}_{R}(\tau)=0$ for any $\tau<0$. 

$({\bf b_2})$ We only prove the case for $R\in \mathcal{Q}_{III}$, other two cases can be proved similarly. If $R\in \mathcal{Q}_{III}$ has at east one critical point on $\mathbb{T}$.
Recall that, as in (\ref{new-d-2}), $R'$ has the from as 
\begin{eqnarray} 
R'(z)&=& \frac{[(z-z_1)(z-z_2)\cdots(z-z_s)]P_{\bf n}(z)}{\prod\limits_{j=1}^{l}(z-e^{i\theta_j})^3\prod\limits_{k=1}^{t}(z-e^{i\widetilde{\theta}_k})^2P_{\bf m}(z)}.\nonumber\end{eqnarray}
We set $$g(z)=\frac{P_{\bf n}(z)}{\prod\limits_{j=1}^{l}(z-e^{i\theta_j})^3\prod\limits_{k=1}^{t}(z-e^{i\widetilde{\theta}_k})^2P_{\bf m}(z)},$$
so that $R'(z)=g(z)\prod\limits_{n=1}^{s}(z-z_n)$. Then we see from the definition of $R$ that 

({\bf 1.1}) there are two constants $c_0>0$, $r_0\in (0,1)$ such that $|g^{-1}(z)|\leq c_0$ for all $|z|\in (r_0, 1)$,
and

({\bf 1.2}) for each $z_n=e^{i\arg z_n}$, $n=1,2,\cdots, s$, we can always find constants $r_n\in (0,1)$, $\vartheta_n>0$, $C_n>0$ such that
$|g(z)|\leq C_n$ for all $$z\in \Omega_n=\{z=re^{i\theta}: r_n \leq r \leq 1, \arg z_n-\vartheta_n\leq \theta \leq \arg z_n+\vartheta_n\},$$
and $\Omega_{n_1}\cap \Omega_{n_2}=\emptyset$, whenever $n_1\neq n_2$, $1\leq n_1, n_2\leq s$.  
Then by using Lemma \ref{ll-2}, we obtain that, as $r\rightarrow 1^{-}$, 
\begin{equation}\label{jia-1}\int_{-\pi}^{\pi} \frac{d\theta}{|R'(re^{i\theta})|^{-\tau}}\asymp
\begin{cases}
\frac{1}{(1-r)^{-\tau-1}}, \;\; \text{if} \; \tau<-1, \\
\log \frac{1}{1-r}, \;\;\; \;\; \text{if} \; \tau=-1,\\
1, \quad\quad\quad\;\,\;\; \text{if} \; \tau\in (-1,0).
\end{cases}  
\end{equation}
Note that for any small $\epsilon>0$,  $$\log \frac{1}{1-r}/\frac{1}{(1-r)^{\epsilon}}=(1-r)^{\epsilon}\log \frac{1}{1-r}\rightarrow 0,\,\, {\text {as}}\,\, r\rightarrow 1^{-}.$$
This implies that ${\widetilde{\beta}}_{R}(-1)=0$. Then it follows from (\ref{jia-1}) that 
\begin{equation}{\widetilde{\beta}}_{R}(\tau)=
\begin{cases}
0, \;\quad\,\,\,\;\, \text{for} \; \tau\in (-1,0) , \\
|\tau|-1, \text{for} \;\tau\leq -1.
\end{cases}\nonumber 
\end{equation}

{({\bf C})} We next consider the case $\tau> 0$. 

$({\bf c_1})$ When $R$ belongs to the class $\mathcal{Q}_I$, we easily see that ${\widetilde{\beta}}_{R}(\tau)=0$ for any $\tau>0$. 

$({\bf c_2})$  When $R$ belongs to the class $\mathcal{Q}_{II}$. Recall that, as in (\ref{new-d-1}), $R'$ has the form as
$$R'(z)=\frac{[(z-z_1)(z-z_2)\cdots(z-z_s)]P_{\bf n}(z)}{\prod\limits_{j=1}^{l}(z-e^{i\theta_j})^2 P_{{\bf m}}(z)}.$$ 
We set $$Q_1(z)=\frac{[(z-z_1)(z-z_2)\cdots(z-z_s)]P_{\bf n}(z)}{P_{{\bf m}}(z)}.$$ 
Then, from the definition of $R$, we can find constants $M>0$, $r_0\in (0,1)$,  $\eta_j\in (0,\frac{\pi}{2l})$, $M_{j,1}>0, M_{j,2}>0$, $j=1,2,\cdots, l$, such that

$(\bf 2.1)$ $\Omega_{j_1}\cap \Omega_{j_2}=\emptyset$ whenever $j_1\neq j_2$, $1\leq j_1, j_2 \leq l$, and $M_{j,2}\leq |Q_1(z)|\leq M_{j,1}$ for all $z\in \Omega_j$, here $\Omega_j=\{re^{\theta}: r_0 \leq r\leq1, \theta_j-\eta_j\leq \theta\leq \theta_j+\eta_j\},$ $j=1,2,\cdots, l$, and $R$ has no critical points in $\bigcup\limits_{1\leq j\leq l}\Omega_j$, and

$(\bf 2.2)$ $|R'(z)|\leq M$ for all $z \in \Omega_0-\bigcup\limits_{1\leq j\leq l}\Omega_j$, here $\Omega_0=\{re^{\theta}: r_0\leq r\leq 1\}.$

For $1\leq j \leq l$, we take
$$D^j:=\max\limits_{z\in \Omega_j}\{\prod\limits_{1\leq a \leq l, a\neq j} |z-e^{i\theta_{a}}|\}, \,\, D_j:=\min\limits_{z\in \Omega_j}\{\prod\limits_{1\leq a \leq l, a\neq j} |z-e^{i\theta_{a}}|\}.$$ In particular, when $l=1$, we take $D^1=D_1=1$. We define $$\mathbf{D}^{\ast}=\max\limits_{1\leq j\leq l}\{D^j\}, \,\,\, \mathbf{D}_{\ast}=\min\limits_{1\leq j\leq l}\{D_j\}.$$
It is easy to see that $0<\mathbf{D}_{\ast}<\mathbf{D}^{\ast}<+\infty$ since $\Omega_j, j=1,2,\cdots, l$ are all compact. Then, when $r>r_0$, we have
\begin{eqnarray}
\int_{-\pi}^{\pi}|R'(re^{i\theta})|^{\tau}d\theta &\geq & \sum_{j=1}^{l}\int_{\theta_j-\eta_j}^{\theta_j+\eta_j}\frac{|Q_1(re^{i\theta})|^{\tau}}{\prod\limits_{j=1}^{l}|re^{i\theta}-e^{i\theta_j}|^{2\tau}}d\theta\nonumber \\
&\geq& M_2^{\tau} [\mathbf{D}^{*}]^{-2\tau}\sum_{j=1}^{l}\int_{\theta_j-\eta_j}^{\theta_j+\eta_j}\frac{d\theta}{|re^{i\theta}-e^{i\theta_j}|^{2\tau}}\nonumber \\
&=& M_2^{\tau}[\mathbf{D}^{*}]^{-2\tau}\sum_{j=1}^{l}\int_{-\eta_j}^{\eta_j}\frac{d\theta}{|1-re^{i\theta}|^{2\tau}},\nonumber
\end{eqnarray}here $M_2=\min\limits_{1\leq j \leq l}\{M_{j,2}\}$, and 
\begin{eqnarray}
\int_{-\pi}^{\pi}|R'(re^{i\theta})|^{\tau}d\theta&\leq& \sum\limits_{j=1}^l \int_{\theta_j-\eta_j}^{\theta_j+\eta_j}|R'(re^{i\theta})|^{\tau}d\theta+2\pi M^{\tau} 
\nonumber \\
&\leq & M_1^{\tau}[{\bf{D}}_{*}]^{-2\tau}\sum\limits_{j=1}^l \int_{\theta_j-\eta_j}^{\theta_j+\eta_j}\frac{d\theta}{|re^{i\theta}-e^{i\theta_j}|^{2\tau}}+2\pi M^{\tau} \nonumber  \\
&=& M_1^{\tau}[{\bf{D}}_{*}]^{-2\tau} \sum\limits_{j=1}^l \int_{-\eta_j}^{\eta_j}\frac{d\theta}{|1-re^{i\theta}|^{2\tau}}+2\pi M^{\tau} ,\nonumber
\end{eqnarray}
here, $M_1=\max\limits_{1\leq j \leq l}\{M_{j,1}\}$.
Consequently, from Lemma \ref{ll-1}, when $r>r_0$, we obtain that,  for $\tau>\frac{1}{2}$, 
 $$\int_{-\pi}^{\pi}|R'(re^{i\theta})|^{\tau}d\theta \asymp \sum\limits_{j=1}^l \int_{-\eta_j}^{\eta_j}\frac{d\theta}{|1-re^{i\theta}|^{2\tau}}\asymp \frac{1}{(1-r)^{2\tau-1}}, \text{as}\,\, r\rightarrow 1^{-}, $$
and for $\tau\in (0,\frac{1}{2}]$, $ \text{as}\,\, r\rightarrow 1^{-},$
\begin{equation}\int_{-\pi}^{\pi}|R'(re^{i\theta})|^{\tau}d\theta\asymp
\begin{cases}
\log \frac{1}{1-r}, \;\;  \text{if} \; \tau=\frac{1}{2},\\
1, \quad\quad\quad\, \text{if} \; \tau\in (0,\frac{1}{2}).
\end{cases}\nonumber 
\end{equation} It follows that \begin{equation}{\widetilde{\beta}}_{R}(\tau)=
\begin{cases}
2\tau-1, \, \text{for} \; \tau>\frac{1}{2}, \\
0, \quad\,\,\;\;\;\; \text{for} \;\tau\in (0, \frac{1}{2}].
\end{cases}\nonumber 
\end{equation} 

$({\bf c_3})$ When $R$ belongs to the class $\mathcal{Q}_{III}$. The main idea of the proof of this case is similar to the case $({\bf c_2})$.  Recall that, as in (\ref{new-d-2}), 
\begin{eqnarray} 
R'(z)&=& \frac{[(z-z_1)(z-z_2)\cdots(z-z_s)]P_{\bf n}(z)}{\prod\limits_{j=1}^{l}(z-e^{i\theta_j})^3\prod\limits_{k=1}^{t}(z-e^{i\widetilde{\theta}_k})^2P_{\bf m}(z)}.\end{eqnarray} 
We let $$Q_2(z)=\frac{[(z-z_1)(z-z_2)\cdots(z-z_s)]P_{\bf n}(z)}{P_{\bf m}(z)}.$$
Then, when $t\in \mathbb{N}$, we see from the definition of $R$ that there are constants $M>0$, $r_0\in (0,1)$, and $\eta_j\in (0,\frac{\pi}{2(l+t)})$, $M_{j,1}>0, M_{j,2}>0$, $j=1,2,\cdots, l$, and $\widetilde{\eta}_k\in (0,\frac{\pi}{2(l+t)})$, $\widetilde{M}_{k,1}>0, \widetilde{M}_{k,2}>0$, $k=1,2,\cdots t$, such that 

$(\bf 3.1)$ $\Omega_{j_1}\cap \Omega_{j_2}=\emptyset$ whenever $j_1\neq j_2$ and $M_{j,2}\leq |Q_2(z)|\leq M_{j,1}$ for all $z\in \Omega_j$, here $\Omega_j=\{re^{\theta}: r_0\leq r \leq 1, \theta_j-\eta_j\leq\theta\leq\theta_j+\eta_j\},$ $1\leq j_1, j_2\leq l,$ and

$(\bf 3.2)$ $\widetilde{\Omega}_{k_1}\cap \widetilde{\Omega}_{k_2}=\emptyset$ whenever $k_1\neq k_2$, $1\leq k_1, k_2 \leq t$, and ${\Omega}_{j}\cap \widetilde{\Omega}_{k}=\emptyset$ for all $1\leq j \leq l, 1\leq k\leq t$, and $\widetilde{M}_{k,2}\leq |Q_2(z)|\leq \widetilde{M}_{k,1}$ for all $z\in \widetilde{\Omega}_k$, here $\widetilde{\Omega}_k=\{re^{\theta}: {r}_0\leq r \leq 1, \widetilde{\theta}_k-\widetilde{\eta}_k\leq\theta\leq\widetilde{\theta}_k+\widetilde{\eta}_k\},$ and $R$ has no critical points in $\bigcup\limits_{1\leq j\leq l, 1\leq k\leq t}(\Omega_j \bigcup\widetilde{\Omega}_k)$, and

$(\bf 3.3)$ $|R'(z)|\leq M$ for all $z \in \Omega_0-\bigcup\limits_{1\leq j\leq l, 1\leq k\leq t}(\Omega_j \bigcup\widetilde{\Omega}_k)$, here 
$\Omega_0=\{re^{\theta}: r_0\leq r\leq 1\}.$

For $1\leq j \leq l$, we let $\widetilde{\pi}(z):=\prod\limits_{1\leq k \leq t}|z-e^{i\widetilde{\theta}_{k}}|^{2}\}$ and take
$$\mathcal{D}^j:=\max\limits_{z\in \Omega_j}\{\widetilde{\pi}(z)\prod\limits_{1\leq a \leq l, a\neq j} |z-e^{i\theta_{a}}|^{3}\}, \,\, \mathcal{D}_j:=\min\limits_{z\in \Omega_j}\{\widetilde{\pi}(z)\prod\limits_{1\leq a \leq l, a\neq j} |z-e^{i\theta_{a}}|^{3}\}.$$
In particular, when $l=1$, we take $\mathcal{D}^1=\max\limits_{z\in \Omega_1}\widetilde{\pi}(z), \,\, \mathcal{D}_1=\min\limits_{z\in \Omega_1}\widetilde{\pi}(z).$

For $1\leq k \leq t$, we let $\pi(z):=\prod\limits_{1\leq j \leq l}|z-e^{i{\theta}_{j}}|^{3}$ and take
$$\widetilde{D}^k:=\max\limits_{z\in \widetilde{\Omega}_k}\{\pi(z)\prod\limits_{1\leq b \leq t, b\neq k} |z-e^{i\widetilde{\theta}_{b}}|^{2}\},\,\, \widetilde{D}_k:=\min\limits_{z\in  \widetilde{\Omega}_k}\{\pi(z)\prod\limits_{1\leq b \leq t, b\neq k} |z-e^{i\widetilde{\theta}_{b}}|^{2}\}.$$
In particular, when $t=1$, we take $\widetilde{D}^1=\max\limits_{z\in \widetilde{\Omega}_1}\pi(z),\,\, \widetilde{D}_1=\min\limits_{z\in \widetilde{\Omega}_1}\pi(z).$
We define $$\mathbf{D}^{\star}=\max\limits_{1\leq j\leq l, 1\leq k\leq t}\{\mathcal{D}^j, \widetilde{D}^k\},\,\, \mathbf{D}_{\star}=\min\limits_{1\leq j\leq l, 1\leq k\leq t}\{\mathcal{D}_j, \widetilde{D}_k\}.$$
We easily see that $0<\mathbf{D}_{\star}<\mathbf{D}^{\star}<+\infty$ from the compactness of $\Omega_j, j=1,2,\cdots, l$, and $\widetilde{\Omega}_k, k=1,2,\cdots, t$. Then, when $r>r_0$, we have
\begin{eqnarray}
 \lefteqn{\int_{-\pi}^{\pi}|R'(re^{i\theta})|^{\tau}d\theta\geq \sum_{j=1}^{l}\int_{\theta_j-\eta_j}^{\theta_j+\eta_j}|R'(re^{i\theta})|^{\tau}d\theta+\sum_{k=1}^{t}\int_{\widetilde{\theta}_k-\widetilde{\eta}_k}^{\widetilde{\theta}_k+
 \widetilde{\eta}_k}|R'(re^{i\theta})|^{\tau}d\theta} \nonumber \\
 &&\geq  M_2^{\tau}[\mathbf{D}^{\star}]^{-\tau}\Big[\sum_{j=1}^{l}\int_{\theta_j-\eta_j}^{\theta_j+\eta_j}\frac{d\theta}{|re^{i\theta}-e^{i\theta_j}|^{3\tau}}+\sum_{k=1}^{t}\int_{\widetilde{\theta}_k-\widetilde{\eta}_k}^{\widetilde{\theta}_k+
 \widetilde{\eta}_k}\frac{d\theta}{|re^{i\theta}-e^{i\widetilde{\theta}_k}|^{2\tau}}\Big]
\nonumber \\
&&= M_2^{\tau}[\mathbf{D}^{\star}]^{-\tau}\Big[\sum_{j=1}^{l}\int_{-\eta_j}^{\eta_j}\frac{d\theta}{|1-re^{i\theta}|^{3\tau}}+\sum_{k=1}^{t}\int_{-\widetilde{\eta}_k}^{\widetilde{\eta}_k}\frac{d\theta}{|1-re^{i\theta}|^{2\tau}}\Big],\nonumber
\end{eqnarray}
here $M_2=\min\limits_{1\leq j \leq l, 1\leq k \leq t}\{M_{j,2}, \widetilde{M}_{k,2}\}$, and 
\begin{eqnarray}
\lefteqn{\int_{-\pi}^{\pi}|R'(re^{i\theta})|^{\tau}d\theta\leq\sum\limits_{j=1}^l \int_{\theta_j-\eta_j}^{\theta_j+\eta_j}|R'(re^{i\theta})|^{\tau}d\theta+\sum\limits_{k=1}^t \int_{\widetilde{\theta}_k-\widetilde{\eta}_k}^{\widetilde{\theta}_k+\widetilde{\eta}_k}|R'(re^{i\theta})|^{\tau}d\theta+2\pi M^{\tau} } 
\nonumber \\
&&\leq M_1^{\tau}[\mathbf{D}_{\star}]^{-\tau}\Big[\sum\limits_{j=1}^l \int_{\theta_j-\eta_j}^{\theta_j+\eta_j}\frac{d\theta}{|re^{i\theta}-e^{i\theta_j}|^{3\tau}}+\sum\limits_{k=1}^t\int_{\widetilde{\theta}_k-\widetilde{\eta}_k}^{\widetilde{\theta}_k+
\widetilde{\eta}_k}\frac{d\theta}{|re^{i\theta}-e^{i\widetilde{\theta}_k}|^{2\tau}}\Big]+2\pi M^{\tau} \nonumber  \\
&&= M_1^{\tau}[\mathbf{D}_{\star}]^{-\tau}\Big[\sum_{j=1}^{l}\int_{-\eta_j}^{\eta_j}\frac{d\theta}{|1-re^{i\theta}|^{3\tau}}+\sum_{k=1}^{t}\int_{-\widetilde{\eta}_k}^{\widetilde{\eta}_k}\frac{d\theta}{|1-re^{i\theta}|^{2\tau}}\Big]+2\pi M^{\tau} ,\nonumber
\end{eqnarray}
here, $M_1=\max\limits_{1\leq j \leq l, 1\leq k \leq t}\{M_{j,1}, \widetilde{M}_{k,1}\}$.
Consequently, from Lemma \ref{ll-1}, when $r>r_0$, we obtain that, for $\tau>\frac{1}{3}$, \begin{eqnarray}\int_{-\pi}^{\pi}|R'(re^{i\theta})|^{\tau}d\theta &\asymp&  \sum_{j=1}^{l}\int_{-\eta_j}^{\eta_j}\frac{d\theta}{|1-re^{i\theta}|^{3\tau}}+\sum_{k=1}^{t}\int_{-\widetilde{\eta}_k}^{\widetilde{\eta}_k}\frac{d\theta}{|1-re^{i\theta}|^{2\tau}}\nonumber \\
&\asymp& \frac{1}{(1-r)^{3\tau-1}},\,\, \text{as}\,\, r\rightarrow 1^{-},\nonumber \end{eqnarray}
and for $\tau\in (0,\frac{1}{3}]$, $ \text{as}\,\, r\rightarrow 1^{-},$
\begin{equation}\int_{-\pi}^{\pi}|R'(re^{i\theta})|^{\tau}d\theta\asymp
\begin{cases}
\log \frac{1}{1-r}, \;\;  \text{if} \; \tau=\frac{1}{3},\\
1, \quad\quad\quad \text{if} \; \tau\in (0,\frac{1}{3}).
\end{cases}\nonumber 
\end{equation} 
It follows that \begin{equation}\label{xg-1}{\widetilde{\beta}}_{R}(\tau)=
\begin{cases}
3\tau-1, \, \text{for} \; \tau>\frac{1}{3}, \\
0, \quad\,\,\;\;\;\; \text{for} \;\tau\in (0, \frac{1}{3}].
\end{cases}
\end{equation} 
For the remaining case when $t=0$, by the similar way as in the case $({\bf c_2})$, we still can prove that (\ref{xg-1}) holds. Finally, combining all above arguments, we conclude that Theorem \ref{l-pro} is true. This finishes‌ the proof of Theorem \ref{l-pro}. 
\end{proof}

\begin{remark}
We recall the Schwarz-Christoffel formula for the unit disk, see \cite[Page 236]{Ahlfors}. Let $f$ be a univalent function (conformal map) from the unit disk $\Delta$ onto the interior of a bounded polygon $\mathcal{P}$ with vertices $z_1, \dots, z_n, n\geq 3$, and corresponding interior angles $\alpha_1 \pi, \dots, \alpha_n \pi$. Then the derivative of $f$ is given by
\begin{equation} 
f'(z) = C \prod_{k=1}^{n} \left(z-{a_k}\right)^{\alpha_k - 1}, \quad z \in \Delta, \nonumber
\end{equation}
where $a_1, a_2, \dots, a_n$ are points on the unit circle $\mathbb{T}$ and $C$ is a non-zero complex constant.

Note that for $r\in (0,1)$, 
$$\int_{-\pi}^{\pi}\frac{d\theta}{|f'(re^{i\theta})|^2}=|C|^{-2}\int_{-\pi}^{\pi} \frac{d\theta}{\prod\limits_{k=1}^{n} \left|re^{i\theta}-{e^{i\arg a_k}}\right|^{2\alpha_k-2}}.$$
Since $\alpha_k < 2$ for each $k$, a direct estimate using Lemma \ref{ll-1} and similar arguments in the proof of part $(\bf c_3)$ of Theorem \ref{l-pro} shows that $\beta_f(-2)<1$. Actually, let $${\bf {\pi}}(z)={\bf {\pi}}(re^{i\theta})=\frac{1}{\prod\limits_{k=1}^{n} \left|re^{i\theta}-{e^{i\arg a_k}}\right|^{2\alpha_k-2}}.$$
For a given $r_0\in (0,1)$, we can find constants $M>0$, $\eta_k\in (0,\frac{\pi}{2n})$, $k=1,2,\cdots, n$ such that $\Omega_{k_1}\cap \Omega_{k_2}=\emptyset$ whenever $k_1\neq k_2$, $1\leq k_1, k_2\leq n,$ here $\Omega_k=\{re^{\theta}: r_0\leq r \leq 1, \arg a_k-\eta_k\leq\theta\leq\arg a_k+\eta_k\},$ and $|{\bf {\pi}}(z)|\leq M$ for all $z\in \Omega-\bigcup\limits_{1\leq k\leq n}\Omega_k$, here $\Omega=\{re^{i\theta}:r_0\leq r\leq 1\}$.

For $1\leq k \leq n$, we take
$$D^k:=\max\limits_{z\in \Omega_k}\{\prod\limits_{1\leq b \leq n, b\neq k} |z-e^{i\arg a_k}|^{2-2\alpha_k}\}.$$ We define ${\bf \mathfrak{D}}^{\ast}=\max\limits_{1\leq k\leq n}\{D_k\}.$ Then we easily see that $0<\mathfrak{D}^{\ast}<+\infty$ since $\Omega_k, k=1,2,\cdots, n$ are all compact. Hence, when $r>r_0$, we have
\begin{eqnarray}
\int_{-\pi}^{\pi}\pi(re^{i\theta})d\theta&\leq& \sum\limits_{k=1}^n \int_{\arg a_k-\eta_k}^{\arg a_k+\eta_k}\pi(re^{i\theta})d\theta+2\pi M 
\nonumber \\
&\leq & [ \mathfrak{D}^{\ast}]^{2}\sum\limits_{j=1}^n \int_{\arg a_k-\eta_k}^{\arg a_k+\eta_k}\frac{d\theta}{|re^{i\theta}-e^{i\arg a_k}|^{2\alpha_k-2}}+2\pi M \nonumber  \\
&=& [ \mathfrak{D}^{\ast}]^{2} \sum\limits_{k=1}^n \int_{-\eta_k}^{\eta_k}\frac{d\theta}{|1-re^{i\theta}|^{2\alpha_k-2}}+2\pi M.\nonumber
\end{eqnarray}
Consequently, from Lemma \ref{ll-1} and the fact that $\alpha_k<2$ for every $k$, we conclude that $\beta_f(-2)<1$. That is to say, the Brennan conjecture holds for all univalent functions from $\Delta$ onto the interior of a bounded polygon $\mathcal{P}$.
\end{remark}

\section{Remarks and open problems}
\subsection{The first part of remarks}
We know that 
\begin{proposition}[\cite{Jin}, {\bf Proposition 3.2, 3.3}]\label{bers} The mapping $\Lambda_1: [\mu]_T \mapsto  N_{f_{\mu}}$ from $(T, d_{T})$ to its image $T_1$ in $E_1$  is a homeomorphism, and the mapping $\Lambda_2: [\mu]_T \mapsto  S_{f_\mu}$ from $(T, d_T)$  to its image $T_2$ in $E_2$ is a homeomorphism.
\end{proposition}

\begin{remark} In view of Propposition \ref{bers}, we can identify the universal Teichm\"uller space with $T_1$ or $T_2$.
We let $\mathcal{S}_q^{\infty}$ be the class of all functions $f\in \mathcal{S}_q$ with $f(\infty)=\infty$. We set
$$\mathbf{N}_q: =\{\phi: \phi=N_f(z), f\in \mathcal{S}_q^{\infty}\}, \,\,\, \mathbf{S}_q:=\{\phi: \phi=S_f(z), f\in \mathcal{S}_q^{\infty}\}.$$
It is easy to see that $T_1=\mathbf{N}_q$ and $T_2= \mathbf{S}_q$. We now identify the closure of universal Teichm\"uller space with $\overline{T}_1=\overline{\mathbf{N}}_q$ in $\mathbf{N}$ or $\overline{T}_2=\overline{\mathbf{S}}_q$ in $\bf{S}$.  On the other hand, we know that $\overline{\mathbf{N}}_q$ is contained in $\overline{\mathcal{T}}$, and for any $\phi \in {\mathbf{S}}_q$, we can take a unique univalent function $f_{\phi}(z)\in \mathcal{S}_q^{\infty} $ with $\phi(z)=S_{f_{\phi}}(z)$. Hence, the IMS functional $I_{{T}_2}: \phi  \mapsto  \beta_{f_{\phi}}(\tau), \phi \in {{\mathbf{S}}}_q$ is well-defined. It follows from Theorem \ref{Em-1} and \ref{Em-3} and Proposition \ref{bers} that 
\begin{corollary}
For each $\tau\in \mathbb{C}$, the IMS functional $I_{\overline{T}_1}: \phi  \mapsto  \beta_{f_{\phi}}(\tau)$ is continuous on $\overline{{T}}_1=\overline{\mathbf{N}}_q$.
\end{corollary}

\begin{corollary}
For each $\tau\in \mathbb{C}$, the IMS functional $I_{T_2}: \phi  \mapsto  \beta_{f_{\phi}}(\tau)$ is continuous on $T_2={\mathbf{S}}_q$.
\end{corollary}
\end{remark}

\begin{remark}\label{r-1}
We let $\mathcal{S}_{Q}$ be the class of all univalent functions $f$ which belong to $\mathcal{S}$ and admit a quasiconformal extension to $\widehat{\mathbb{C}}$. It is easy to see that $\mathcal{S}_q$ is a proper subset of $\mathcal{S}_{Q}$.
 For each $\theta\in (-\pi, \pi]$, let $\mathcal{S}_{\theta}$ be the subclass of $\mathcal{S}_{Q}$  which consists of all the functions $f$ satisfying that $\lim\limits_{\Delta \ni z\rightarrow e^{i\theta}}f(z)=\infty$. For each $\theta\in (-\pi, \pi]$,  $\mathcal{S}_{\theta}$ is a copy of the universal Teichm\"uller space.
By using \cite[Lemma 1]{Zhur} and repeating the arguments in the proof of Theorem \ref{Em-4}, we can similarly prove that
 \begin{proposition}\label{add-p}
Let $\tau\in \mathbb{C}$ with $\tau\neq 0$. For each $\theta\in (-\pi, \pi]$, we have $\beta_{f}(\tau)<B(\tau)$ for any $f\in \mathcal{S}_{\theta}$.
  \end{proposition}
\begin{remark}For $f, g \in \mathcal{S}$, by checking the proof of Proposition \ref{well}, we see that $\beta_f(\tau)=\beta_g(\tau)$ for each $\tau \in \mathbb{C}$ if $N_f-N_g\in E_{1,0}$. Then, combining Theorem \ref{Em-4}, Proposition \ref{add-p}, we obtain that, if $f$ satisfies that $N_{f}$ is equivalent to some $N_{f_q}$ in $E_1$, here $f_q \in \mathcal{S}_Q$, then $\beta_{f}(\tau)<B(\tau)$ for any $\tau\neq 0$ so that $f$ can not be an extremal function for any $B(\tau)$ with $\tau\neq 0$. \end{remark}
  We let $$\mathbf{S}_{\theta}:=\{\phi, \phi=S_f(z), \, f\in \mathcal{S}_{\theta}\}.$$
If $\phi\in \mathbf{S}_{\theta}$, then there is a unique univalent function $f_{\phi}$ with $f_{\phi}\in \mathcal{S}_{\theta}$ and such that $\phi(z)=S_{f_{\phi}}(z).$ If $\phi\in \partial\mathbf{S}_{\theta}$, we know that there is a sequence $\{f_n\}_{n=1}^{\infty}$, $f_{n}\in \mathcal{S}_{\theta}$, such that $\lim_{n\rightarrow\infty}\|S_{f_n}-\phi\|_{E_2}=0,$ and the sequence $\{f_n\}_{n=1}^{\infty}$ converges for every $z\in \Delta$. Then, taking $f_\phi(z)=\lim_{n\rightarrow\infty}f_n(z), z\in \Delta$, we see that $f_{\phi}\in \mathcal{S}$ with $\phi(z)=S_{f_{\phi}}(z),$ and $f_{\phi}$ is unique by the normalization. Here the statement  $f_{\phi}$ is unique means that, if there is another sequence $\{\widehat{f}_n(z)\}_{n=1}^{\infty}$, $\widehat{f}_{n}\in \mathcal{S}_{\theta}$, such that $\lim_{n\rightarrow\infty}\|S_{\widehat{f}_n}-\phi\|_{E_2}=0,$ and the sequence $\{\widehat{f}_n(z)\}_{n=1}^{\infty}$ converges for every $z\in \Delta$, then, take $\widehat{f}_\phi(z)=\lim_{n\rightarrow\infty}\widehat{f}_n(z)$, we have $\widehat{f}_{\phi}(z)=f_{\phi}(z)$ for all $z\in \Delta$.  From these, for any $\phi \in \overline{\mathbf{S}}_{\theta}$, we can take a unique univalent function $f_{\phi}(z)$ with $f_{\phi}\in \mathcal{S}$ and such that $\phi(z)=S_{f_{\phi}}(z).$ Thus, the IMS functional $I_{\overline{T}_{\theta}}: \phi \mapsto  \beta_{f_{\phi}}(\tau), \phi \in \overline{\mathbf{S}}_{\theta}$ is well-defined. From \cite{AG2}, we know that the boundary $\partial\mathbf{S}_{\theta}$ of $\mathbf{S}_{\theta}$ is larger than the one of $\mathbf{N}_q$. It is interesting to study
\begin{problem}
For each $\theta\in (-\pi, \pi]$, is the IMS functional $I_{\overline{T}_{\theta}}: \phi  \mapsto  \beta_{f_{\phi}}(\tau)$ continuous on $\overline{T}_2=\overline{{\mathbf{S}}}_{\theta}$ when $\tau\neq 0$?
\end{problem}
\end{remark}

\begin{remark}
Let $\tau\in \mathbb{C}$. We know from \cite{Jin} that $\mathcal{E}_1(z)=-\log(1-z)$ is an extremal function for $B_b(\tau)$ when $\tau \geq 2$, and $\mathcal{E}_2(z)=z-\frac{1}{2}z^2$ is an extremal function for $B_b(\tau)$ when $\tau \leq \tau_{*}$. Here $\tau_{*}$ is the same as in Theorem \ref{CM}.  Also, both $N_{\mathcal{E}_1}(z)$ and $N_{\mathcal{E}_2}$ lie on the boundary of $\mathcal{T}$. In view of these facts and Theorem \ref{Em-4}, we will say an extremal function for $B_b(\tau)$ is {\em regular} if its Pre-Schwarzian derivative lies on the boundary of $\mathcal{T}$.  We then raise the following general 
\begin{conjecture}\label{con-add}
For each $\tau\in \mathbb{C}$, $B_{b}(\tau)$ has at least one regular extremal function.
\end{conjecture} 
If Conjecture \ref{con-add} is true, then the following conjecture will follow.
\begin{conjecture}
For each $\tau\in \mathbb{C}$, $B_{b}(\tau)$ has  at least one extremal function whose Schwarzian derivative lies on the boundary $\partial \mathbf{S}_{Q}$ of $\mathbf{S}_{Q}$.
\end{conjecture}
Here, ${\mathbf{S}}_{Q}$ is defined as  \begin{equation}\label{kongs}{\mathbf{S}}_{Q}:=\{\phi:\phi=S_f(z), f\in \mathcal{S}_{Q}\}.\end{equation}
We easily see that ${\mathbf{S}}_{Q}$ coincides with $\mathbf{S}_q$ by using (\ref{chain-1}).
\end{remark}

\subsection{The second part of remarks}When $\tau\leq -2$, we find that there are some univalent functions $f$ satisfying that $\beta_f(\tau)=|\tau|-1$ and $\|S_f\|_{E_2}=6$.  For instance, $$\kappa(z)=\frac{z}{(1-z)^2}, \,\,\mathcal{E}_2(z)=P_2(z)=z-\frac{1}{2}z^2, \,\, P_3(z)=z-\frac{1}{3}z^3, z\in \Delta,$$
which have been mentioned before. We have known that $\beta_{\kappa}(\tau)=\beta_{\mathcal{E}_2}(\tau)=\beta_{P_3}(\tau)=|\tau|-1$ for $\tau\leq -2$. On the other hand, by a simply computation, we obtain that
 $$S_{\kappa}(z)=-\frac{6}{(1-z^2)^2},\,S_{\mathcal{E}_2}(z)=-\frac{3}{2}\frac{1}{(1-z)^2},\,S_{P_3}(z)=\frac{-2-4z^2}{(1-z^2)^2},\, z\in \Delta.$$
 It follows that $\|S_{\kappa}\|_{E_2}=\|S_{\mathcal{E}_2}\|_{E_2}=\|S_{P_3}\|_{E_2}=6$. We will provide more functions that have these characteristics.
\begin{example}\label{exam-1}
 We consider the $T$-symmetric Koebe functions $\kappa_{T}$, see \cite{jmaa}, which are defined as
$$\kappa_{T}(z):=\frac{z}{(1-z^{T})^{\frac{2}{T}}}, z\in \Delta, T\in \mathbb{N}.$$
We see that $\kappa_{1}(z)=\kappa(z)$ and 
\begin{equation}\label{kapp}
 \kappa'_{T}(z)=\frac{(1+z^T)}{(1-z^T)^{1+\frac{2}{T}}}, z\in \Delta, 
\end{equation}
 $$S_{\kappa_{T}}(z)=\frac{2(T^2-1)z^{3T-2}-2(T^2+2)z^{2T-2}+2(T^2-1)z^{T-2}}{(1-z^{2T})^2}, z\in \Delta.$$ 
In particular, $$S_{\kappa_2}(z)=\frac{6}{(1+z^2)^2}.$$ 
In view of Lemma \ref{ll-2}, we obtain from (\ref{kapp}) that $\beta_{\kappa_{T}}(\tau)=|\tau|-1$ for each $T\in \mathbb{N}$ when $\tau\leq -2$. Also, it is not hard to check that, for each $T\in \mathbb{N}$,
$$(1-|z|^2)^2|S_{\kappa_{T}}(z)| \rightarrow 6,\,\, {\text{as}}\,\, z\rightarrow e^{\frac{k\pi}{T}i}\,\text{radially}.$$ Here $k=1, 3, \cdots, 2T-1$. Hence $\|S_{\kappa_{T}}\|_{E_2}=6$ for each $T\in \mathbb{N}$.
\end{example}

\begin{example}
Now, we consider the so-called $\gamma$-spiral Koebe functions $f_{\gamma}$, which are defined as
 $$f_{\gamma}(z):={z}(1-z)^{-2e^{i\gamma}cos\gamma}, z\in \Delta,$$
for $\gamma \in (-\frac{\pi}{2}, \frac{\pi}{2})$, see \cite{Ok}. We see that $f_0$ is the Koebe function. A direct computation yields that   
$$f'_{\gamma}(z)=\frac{1+ze^{2i\gamma}}{(1-z)^{1+2e^{i\gamma}cos\gamma}}, z\in \Delta.$$
It follows from Lemma \ref{ll-2} again that $\beta_{f_{\gamma}}(\tau)=|\tau|-1$ for each $\gamma \in (-\frac{\pi}{2}, \frac{\pi}{2})$ when $\tau\leq -2$, and it has been proved in \cite{Ok} that $\|S_{f_{\gamma}}\|_{E_2}=6$. 
\end{example}

Moreover, we observe that
\begin{proposition}
Let $\tau\leq -2$, $R\in \mathcal{U}_R$. If $R$ has at least one critical point on $\mathbb{T}$, then $\beta_{R}(\tau)=|\tau|-1$ and $\|S_{R}\|_{E_2}=6.$  
\end{proposition}
\begin{proof}
 Note that $\beta_{R}(\tau)=|\tau|-1$ follows from Theorem \ref{l-pro}. We only need to prove $\|S_{R}\|_{E_2}=6.$   If $R$ has at least one critical point on $\mathbb{T}$, we write $$R'(z)=\frac{(z-e^{i\theta})P_{\bf n}(z)}{P_{\bf m}(z)},$$
here $\theta\in (-\pi,\pi]$, the polynomials $(z-e^{i\theta})P_{\bf n}$ and $P_{\bf m}$ have no common factors except $1$, and $P_{\bf n}(e^{i\theta})\neq 0$. We let 
$$\varphi(z)=\frac{P_{\bf n}(z)}{P_{\bf m}(z)}.$$ 
Then $\varphi$ is analytic in $\Delta$, and is analytic at the point $e^{i\theta}$ with $\varphi(e^{i\theta})\neq 0$. Consequently, from $$R''(z)=\varphi(z)+(z-e^{i\theta})\varphi'(z),$$
and 
$$R'''(z)=2\varphi'(z)+(z-e^{i\theta})\varphi''(z),$$
we obtain that 
\begin{eqnarray}|S_{R}(z)|&=&\Big|\frac{R'''(z)}{R'(z)}-\frac{3}{2}\Big[\frac{R''(z)}{R'(z)}\Big]^2\Big|\nonumber 
\\
&=&\frac{|[\varphi'(z)+(z-e^{i\theta})\varphi''(z)](z-e^{i\theta})\varphi(z)-\frac{3}{2}[\varphi(z)+(z-e^{i\theta})
\varphi'(z)]^2|}{|1-ze^{-i\theta}|^2|\varphi(z)|^2}.\nonumber 
\end{eqnarray}
Hence, we see that $|S_{R}(z)|(1-|z|^2)^2 \rightarrow 6$, as $z\rightarrow e^{i\theta}$ radially. The proposition is proved.
\end{proof}

Meanwhile, if Conjecture \ref{conj} is true, then $\kappa_{T}$, $f_{\gamma}$ and any $R\in \mathcal{U}_{R}$ that have at least one critical point on $\mathbb{T}$, are all the extremal functions for $B_b(\tau)$ when $\tau \leq -2$. From these, it is reasonable to guess that 
\begin{conjecture}When $\tau \leq -2$, the extremal function $f$ for $B_b(\tau)$ exists and satisfies that $\|S_f\|_{E_2}=6$.
\end{conjecture}

 \subsection{The third part of remarks} 
\begin{remark}\label{pn}
Let $R\in \mathcal{U}_P$. If $R$ has at least one critical point on $\mathbb{T}$, then, from Theorem \ref{Em-4}, Proposition \ref{add-p}, Theorem \ref{CM}, and the fact that $\beta_{R}(\tau)=|\tau|-1$ for $\tau\leq -2$, we know that $R$ does not belong to $\mathcal{S}_{Q}$ (defined as in Remark \ref{r-1}), and the Schwarzian derivative $S_R$ of $R$ may lie on $\partial \mathbf{S}_{Q}$ (here $\mathbf{S}_{Q}$ is defined as in (\ref{kongs})). 
 For example, the function $\mathcal{E}_2(z)=z-\frac{1}{2}z^2$ has a critical point on $\mathbb{T}$, and we know from \cite[Remark 5.9]{Jin} that $S_{\mathcal{E}_2}$ lies on  $\partial \mathbf{S}_{Q}$. Hence we believe that $S_R$ should lie on the boundary of $\mathbf{S}_{Q}$ if $R\in \mathcal{U}_R$ has at least one critical point on $\mathbb{T}$, although we have not found a proof for this claim. We leave it as the following 
\begin{conjecture}\label{conj-2}
Let $R\in \mathcal{U}_R$. Then $S_R$ lies on $\partial \mathbf{S}_{Q}$ if $R$ has at least one critical point on $\mathbb{T}$.
\end{conjecture}
\end{remark}
We define the closed subspace $E_{1,0}$ of $E_1$ as
$$E_{1,0}:=\{\phi\in E_1: \lim\limits_{|z|\rightarrow 1^{-}}\phi(z)(1-|z|^2)=0\}.$$
We say two elements $\phi_1, \phi_2 \in E_1$ are equivalent, if $\phi_1-\phi_2\in E_{1,0}$. The equivalence class of $\phi\in E_1$  is denoted by $[\phi]_{E_1}$.  The set of all equivalence classes $[\phi]_{E_1}$ will be denote by $E_1/E_{1,0}$. $E_1/E_{1,0}$ is a Banach space with the quotient norm
$$\|[\phi]_{E_1}\|:=\inf\limits_{\psi \in [\phi]_{E_1}} \|\psi\|_{E_1}=\inf\limits_{\psi \in E_{1,0}} \|\phi+\psi\|_{E_1}.$$
We let $\widetilde{\mathcal{B}}: [\mu]_{AT}  \mapsto  [N_{f_{\mu}}]_{E_1}.$
The mapping $\widetilde{\mathcal{B}}$ is called {\em asymptotic Pre-Bers map}. It has been proved in \cite{Jin} that
\begin{proposition}\label{abers-1}
The mapping $\widetilde{\mathcal{B}}$ from $(AT, d_{AT})$ to $\widetilde{\mathbf{N}}_q$ in $E_1/E_{1,0}$ is a homeomorphism.
Here, $\widetilde{\mathbf{N}}_q:=\{[\phi]_{E_1}: \phi=N_f(z),\, f\in \mathcal{S}_q^{\infty}\}$ is an open subset of $E_1/E_{1,0}$.
\end{proposition}
\begin{remark}\label{above}
In view of Proposition \ref{abers-1}, we can identify the universal asymptotic Teichm\"uller space $AT$ with $\widetilde{\mathbf{N}}_q$.  \end{remark}
\begin{remark}\label{rem-2}
Let $P_n\in \mathcal{U}_R$. If $P_n$ has no critical points on $\mathbb{T}$, then it is easy to check that $N_{P_n}$ is contained in $E_{1,0}$ so that $N_{P_n}$ belongs to the class $[0]_{E_1}$ in $E_1/E_{1,0}$. On the other hand, if $P_n$ has at least one critical point on $\mathbb{T}$, then we see that there is no function $f\in \mathcal{S}_q$ such that $N_{P_n}-N_{f}\in E_{1,0}$. Actually, if there is a function $f_0\in \mathcal{S}_q$ such that $N_{P_n}-N_{f_0}\in E_{1,0}$, then, by using the same arguments as in the proof of Proposition \ref{well}, we obtain that $\beta_{f_0}(\tau)=\beta_{P_n}(\tau)=|\tau|-1$ for $\tau \leq -2$, which contradicts Theorem \ref{CM} and \ref{Em-4}. Hence, we further raise the following {\em dichotomy conjecture for univalent polynomials}. 
\begin{conjecture}\label{dich}
Let $P_n \in {\mathcal{U}}_P$. Then (1) $[N_{P_n}]_{E_1}$ lies on the origin, i.e., $[0]_{E_1}$, of the universal asymptotic Teichm\"uller space $\widetilde{\mathbf{N}}_q$, if $P_n$ has no critical points on $\mathbb{T}$; (2) 
$[N_{P_n}]_{E_1}$ lies on the boundary of $\widetilde{\mathbf{N}}_q$, if $P_n$ has at least one critical point on $\mathbb{T}$.
\end{conjecture}  \end{remark}
\begin{remark}\label{add-1}
As remarked above, the proof of part (1) of Conjecture \ref{dich} is easy. However, checking the validity of part (2) of Conjecture \ref{dich} seems not easy. We will give an example for part (2) of  Conjecture \ref{dich}. We consider again the function $\mathcal{E}_2(z)=z-\frac{1}{2}z^2$ which has one critical point on $\mathbb{T}$ and belongs to $\partial\mathcal{T}$. 
It is known from \cite[Remark 5.9 $({\bf II})$]{Jin} that \begin{equation}\label{add-0}
\lim\limits_{\gamma\rightarrow 1^{-}} \|N_{g_{\gamma}}-N_{\mathcal{E}_2}\|_{E_1}=0.\end{equation} Here, $g_{\gamma}(z):=[(1-z)^{1+\gamma}-1]/(-\gamma-1), \gamma \in (0,1)$ are univalent functions in $\Delta$ which admit a quasiconformal extension to $\widehat{\mathbb{C}}$. Consequently, note that $g_{\gamma}, \gamma\in (0,1)$ are all bounded in $\Delta$, we can find a family of functions $f_{\gamma}\in \mathcal{S}_q^{\infty}$ such that $\varsigma_{\gamma} \circ f_{\gamma}=g_{\gamma}, z\in \widehat{\mathbb{C}}$ so that $N_{f_{\gamma}}-N_{g_{\gamma}}\in E_{1,0}$ for each $\gamma \in (0,1)$. Here $\varsigma_{\gamma}$ are the M\"obius transformations with the form $\varsigma_{\gamma}(\eta)=\frac{\eta}{1+\eta a_{\gamma}}$, $a_{\gamma}\in \widehat{\mathbb{C}}, a_{\gamma}^{-1}\neq \overline{g_{\gamma}(\Delta)}$. Then, from the fact $$\|[N_{g_{\gamma}}]_{E_1}-[N_{\mathcal{E}_2}]_{E_1}\|\leq \|N_{g_{\gamma}}-N_{\mathcal{E}_2}\|_{E_1},$$
and (\ref{add-0}), we obtain that 
$$\lim\limits_{\gamma\rightarrow 1^{-}} \|[N_{f_{\gamma}}]_{E_1}-[N_{\mathcal{E}_2}]_{E_1}\|=\lim\limits_{\gamma\rightarrow 1^{-}} \|[N_{g_{\gamma}}]_{E_1}-[N_{\mathcal{E}_2}]_{E_1}\|=0.$$
On the other hand, from Remark \ref{above}, we know that there is no functions $f\in \mathcal{S}_q$ such that $N_{\mathcal{E}_2}-N_{f}\in E_{1,0}.$ Combining these facts, we have checked that $[N_{\mathcal{E}_2}]_{E_1}$ lies on the boundary of $\widetilde{\mathbf{N}}_q$.  
\end{remark}
\begin{remark}
When $P_n \in \mathcal{U}_P$ has no critical points on $\mathbb{T}$, we have pointed out that $N_{P_n}$ belongs to ${E_{1,0}}$. On the other hand, for a univalent function (conformal mapping) $f$ from $\Delta$ to a bounded Jordan domain in $\mathbb{C}$, we know that $f$ is an asymptotically conformal mapping if and only if $N_f$ belongs to $E_{1,0}$, see \cite{GS}, \cite{Po-1} or \cite{Po-2}. Hence it is natural to ask whether $P_n$ is an asymptotically conformal mapping in $\Delta$ when $P_n \in \mathcal{U}_P$ has no critical points on $\mathbb{T}$. However, we find that there exists $P_n\in \mathcal{U}_P$ with no critical points on $\mathbb{T}$ for which $P_n(\mathbb{T})$ is not a bounded Jordan curve so that $P_n$ cannot be asymptotically conformal. As an example, we consider ${\bf{P}}(z):=z+\frac{\sqrt{3}}{2}z^2+\frac{1}{4}z^3, z\in \Delta.$ Note that ${\bf{P}}(z)=\pi_2 \circ \overline{{\bf{P}}} \circ \pi_1(z), z\in \Delta$. Here, $w=\pi_1(z)=\frac{\sqrt{3}}{2}z$, which  maps $\Delta$ conformally to the disk $\Delta(\frac{\sqrt{3}}{2})=\{w\in \mathbb{C}: |w|<\sin(\frac{\pi}{3})=\frac{\sqrt{3}}{2}\}$,  $\zeta=\overline{{\bf{P}}}(w)=(1+w)^3-1$, $\pi_2(\zeta)=\frac{2}{3\sqrt{3}}\zeta$. Then, by Corollary 5.7.3 in \cite[Page 197]{Sheil}, we obtain that ${\bf{P}}$ belongs to $\mathcal{U}_P$, or we can check this fact by using Dieudonné's univalence criterion, see \cite[Page 75]{Du}. Meanwhile, we see from ${\bf{P}}'(z)=(\frac{\sqrt{3}}{2}z+1)^2$ that ${\bf{P}}'\neq 0$ for any $|z|<2/\sqrt{3}$. Also, we easily check that ${\bf{P}}(-\frac{\sqrt{3}}{2}+\frac{1}{2}i)={\bf{P}}(-\frac{\sqrt{3}}{2}-\frac{1}{2}i)$, which means that the curve ${\bf P}(\mathbb{T})$ is not a bounded Jordan curve in $\mathbb{C}$. 
\end{remark}
\begin{remark}
If Conjecture \ref{con-add} holds, namely, $B_{b}(\tau)$ has an extremal function $f$ whose Pre-Schwarzian derivative lies on $\partial\mathcal{T}$, then, by using the same arguments as in Remark \ref{add-1}, we can obtain that $[N_{f}]_{E_1}$ lies on the boundary of $\widetilde{\mathbf{N}}_q$. This means that Conjecture \ref{con-add} implies that
\begin{conjecture}
For each $\tau\in \mathbb{C}$, $B_{b}(\tau)$ has  at least one extremal function $f\in \mathcal{S}$ such that $[N_f]_{E_1}$ lies on the boundary of the universal asymptotic Teichm\"uller space $\widetilde{\mathbf{N}}_q$. 
\end{conjecture}
\end{remark}
 
 \subsection{The final part of remarks} 
In the paper \cite{SS}, Shimorin introduced and studied the following multiplication operator $\mathcal{M}_f$, induced by the Schwarzian derivative of univalent function $f\in \mathcal{S}$,
 $$\mathcal{M}_f(\phi)(z):=S_f(z)\phi(z),\, \phi\in \mathcal{A}(\Delta).$$
There is a close connection between the norm of multiplication operator $\mathcal{M}_f$ and the Brennan conjecture. It has been proved in \cite{SS} that 
 \begin{proposition}\label{last-p}
 If some univalent function $f\in \mathcal{S}$ satisfies that
\begin{equation}\label{mo}\|\mathcal{M}_f\|^2\leq \frac{36(\alpha+3)(\alpha+5)}{(\alpha+2)(\alpha+4)},\end{equation}
 for any $\alpha>0$. Then $\beta_f(-2)\leq 1$, that is to say, the Brennan conjecture is true for the function $f$. 
 Here,
 $$\|\mathcal{M}_f\|=\sup_{\|\phi\|_{\alpha}\neq0, \phi\in \mathcal{H}_{\alpha}^{2}(\Delta)}\frac{\|S_f(z)\phi(z)\|_{\alpha+4}}{\|\phi(z)\|_{\alpha}}:=\mathbb{M}_f(\alpha).$$
 \end{proposition}
\begin{remark}It is easy to see that, if for any $\alpha>0$, (\ref{mo}) holds for all $f\in \mathcal{S}_q$, then the Brennan conjecture is true. On the other hand, it has been pointed out in \cite[Page 1633]{SS} that there exists univalent function $f\in \mathcal{S}$ such that 
$\mathbb{M}_f(\alpha)>\mathbb{M}_{\kappa}(\alpha)$ when $\alpha>-1$. Meanwhile, for any $\alpha>-1$, it has been shown in \cite{SS} that, 
\begin{equation}\label{last-e}\mathbb{M}_{\kappa_{\Theta}}^2(\alpha)=\frac{36(\alpha+3)(\alpha+5)}{(\alpha+2)(\alpha+4)}.\end{equation}
 Here, $\kappa_{\Theta}$ are the Koebe function $\kappa$ and its rotations, which are defined as
\begin{equation}\label{mob}\kappa_{\Theta}(z):=\frac{z}{(1-e^{i\Theta} z)^2},\, \Theta \in (-\pi, \pi],\, z\in \Delta. \end{equation}
Let $\varrho$ be a M\"obius transformation.  We conclude that $\varrho \circ \kappa_{\Theta} \in \mathcal{S}$,  if and only if  $\varrho$ has the form as 
$$\varrho(\zeta):=\varrho_{c}(\zeta)=\frac{\zeta}{1+ce^{i\Theta}\zeta}, 0\leq c\leq 4.$$
Since the Schwarzian derivative is invariant under the M\"obius transformation, we see that, for any $\alpha>-1$, it holds that $\mathbb{M}_{\varrho_c \circ \kappa_{\Theta}}(\alpha)=\mathbb{M}_{\kappa}(\alpha),$ for any $(c, \Theta)\in [0,4]\times(-\pi, \pi]$. Because the Koebe function and its rotations are the unique class of extremal functions in $\mathcal{S}$ for some extremal problems in the theory of univalent functions, thus, it is natural to guess that the functions $\varrho_c \circ \kappa_{\Theta}$, $(c, \Theta)\in [0,4]\times(-\pi, \pi]$,  may be the unique class that induce the biggest norm of the operator $\mathcal{M}_f$ in the class $\mathcal{S}$ for any $\alpha>0$. However, we find that there are other univalent functions $f\in \mathcal{S}$ such that (\ref{last-e}) still holds for any $\alpha>-1$. 

We consider again the functions $\mathcal{E}_2$ and $P_3$. We will check that 

\begin{proposition}\label{p2}For each $\alpha>-1$, $$\mathbb{M}_{\mathcal{E}_2}^2(\alpha)=\mathbb{M}_{P_3}^2(\alpha)=\frac{36(\alpha+3)(\alpha+5)}{(\alpha+2)(\alpha+4)}.$$
  \end{proposition}  
 
 \begin{proof}[Proof of Proposition \ref{p2}]
Recall that
 $$S_{\mathcal{E}_2}(z)=-\frac{3}{2}\frac{1}{(1-z)^2},\,S_{P_3}(z)=\frac{-2-4z^2}{(1-z^2)^2},\, z\in \Delta.$$
Then, for any $z\in \Delta$,
  $$|S_{\mathcal{E}_2}(z)|\leq |S_{\kappa}(z)|,\, {\text{and}}\,\, |S_{P_3}(z)|\leq |S_{\kappa}(z)|=\frac{6}{|1-z^2|^2}.$$ 
Hence we see that $\mathbb{M}_{\mathcal{E}_2}(\alpha)\leq\mathbb{M}_{\kappa}(\alpha)$, and $\mathbb{M}_{P_3}(\alpha)\leq\mathbb{M}_{\kappa}(\alpha)$  for any $\alpha>-1$. 

Next, we first show that $\mathbb{M}_{\mathcal{E}_2}(\alpha)\geq \mathbb{M}_{\kappa}(\alpha)$. For $\gamma>0$, we have 
\begin{equation}\label{l-eq-1}
  \frac{1}{(1-z)^{\gamma}}=\sum_{n=0}^{\infty}\frac{\Gamma(n+\gamma)}{n!\Gamma(\gamma)}z^n, z\in \Delta.\end{equation}
Here $\Gamma$ is the usual Gamma function, see \cite{AAR}. For any $\phi=\sum_{n=0}^{\infty}a_nz^n\in \mathcal{H}_{\alpha}^2(\Delta)$, we have
\begin{equation}\label{l-eq-2}\|\phi\|_{\alpha}^2=\sum_{n=0}^{\infty}\frac{n!\Gamma(\alpha+2)}{\Gamma(n+\alpha+2)}|a_n|^{2}.
\end{equation}
Let $r\in (0,1)$ and $2\lambda>\alpha+2$. Then, from (\ref{l-eq-1}) and (\ref{l-eq-2}), we obtain that
 \begin{equation}\label{l-eq-3}
 \|{(1-rz)^{-\lambda}}\|_{\alpha}^2=\sum_{n=0}^{\infty}\frac{n!\Gamma(\alpha+2)}{\Gamma(n+\alpha+2)}\Big|\frac{\Gamma(n+\lambda)}{n!\Gamma(\lambda)}\Big|^2 r^{2n}.
 \end{equation}
By Stirling's formula, we have
 \begin{equation}
 \frac{\Gamma(n+\lambda)}{n!}=(n+1)^{\lambda-1}[1+o(1)],\,\, {\textup{as}} \,\, n\rightarrow \infty.\nonumber 
 \end{equation}
 Here and later, we use $o(1)$ to denote the general term of a sequence with $o(1)\rightarrow 0, \, {\textup{as}}\,\, n \rightarrow \infty$, which may be different in different places. Consequently, it follows from (\ref{l-eq-3}) that 
 \begin{eqnarray}\label{l-eq-4}
 \|{(1-rz)^{-\lambda}}\|_{\alpha}^2&=&\sum_{n=0}^{\infty}\frac{\Gamma(\alpha+2)}{(n+1)^{\alpha+1}}[1+o(1)]\cdot\frac{(n+1)^{2\lambda-2}}{[\Gamma(\lambda)]^2}[1+o(1)]r^{2n}\nonumber \\
 &=&\sum_{n=0}^{\infty}\frac{\Gamma(\alpha+2)}{[\Gamma(\lambda)]^2}\cdot {(n+1)^{2\lambda-\alpha-3}}[1+o(1)]r^{2n} \nonumber \\
 &=&\frac{\Gamma(\alpha+2)\Gamma(2\lambda-\alpha-2)}{[\Gamma(\lambda)]^2}\sum_{n=0}^{\infty} \frac{\Gamma(n+2\lambda-\alpha-2)}{n!\Gamma(2\lambda-\alpha-2)}[1+o(1)]r^{2n}.
 \end{eqnarray}
 On the other hand, when $2\lambda>\alpha+2$, we have 
  \begin{eqnarray}\label{l-eq-5}
 \lefteqn{[{(1-r^2)^{2\lambda-\alpha-2}}]\sum_{n=0}^{\infty}\frac{\Gamma(n+2\lambda-\alpha-2)}{n!\Gamma(2\lambda-\alpha-2)}[1+o(1)]r^{2n}}\nonumber \\
 &&=1+ [{(1-r^2)^{2\lambda-\alpha-2}}]\sum_{n=0}^{\infty}\frac{\Gamma(n+2\lambda-\alpha-2)}{n!\Gamma(2\lambda-\alpha-2)}\cdot [o(1)]\cdot r^{2n}.
 \end{eqnarray}
Note that $o(1):=b_n\rightarrow 0, \, {\textup{as}}\,\, n \rightarrow \infty$ in (\ref{l-eq-5}). Then, for any small $\varepsilon>0$, there are two constants $n_0\in \mathbb{N}$ and $c_0>0$ such that $|b_n|\leq \varepsilon$ for $n>n_0$ and $|b_n|\leq c_0$ for all $n\in \mathbb{N}\cup \{0\}$. 
It follows that
\begin{eqnarray}\label{l-eq-6}
\lefteqn{[{(1-r^2)^{2\lambda-\alpha-2}}]\sum_{n=0}^{\infty}\frac{\Gamma(n+2\lambda-\alpha-2)}{n!\Gamma(2\lambda-\alpha-2)}\cdot [o(1)]\cdot r^{2n}} \nonumber \\
  && \leq [{(1-r^2)^{2\lambda-\alpha-2}}]c_0\sum_{n=0}^{n_0} \frac{\Gamma(n+2\lambda-\alpha-2)}{n!\Gamma(2\lambda-\alpha-2)}r^{2n}\nonumber \\&& \quad\quad\qquad+\varepsilon[{(1-r^2)^{2\lambda-\alpha-2}}]\sum_{n={{n_0}+1}}^{\infty}\frac{\Gamma(n+2\lambda-\alpha-2)}{n!\Gamma(2\lambda-\alpha-2)}r^{2n} \nonumber \\
  && \leq c_0[{(1-r^2)^{2\lambda-\alpha-2}}]\sum_{n=0}^{n_0} \frac{\Gamma(n+2\lambda-\alpha-2)}{n!\Gamma(2\lambda-\alpha-2)}+\varepsilon.
  \end{eqnarray}
  Meanwhile, \begin{eqnarray}\label{l-eq-7}
\lefteqn{[{(1-r^2)^{2\lambda-\alpha-2}}]\sum_{n=0}^{\infty}\frac{\Gamma(n+2\lambda-\alpha-2)}{n!\Gamma(2\lambda-\alpha-2)}\cdot [o(1)]\cdot r^{2n}} \nonumber \\
  && \geq-[{(1-r^2)^{2\lambda-\alpha-2}}]\sum_{n=0}^{n_0} \frac{\Gamma(n+2\lambda-\alpha-2)}{n!\Gamma(2\lambda-\alpha-2)}|b_n| r^{2n}\nonumber \\&& \quad\quad\qquad-\varepsilon[{(1-r^2)^{2\lambda-\alpha-2}}]\sum_{n={{n_0}+1}}^{\infty}\frac{\Gamma(n+2\lambda-\alpha-2)}{n!\Gamma(2\lambda-\alpha-2)}r^{2n} \nonumber \\
  && \geq -[{(1-r^2)^{2\lambda-\alpha-2}}]\sum_{n=0}^{n_0}\frac{\Gamma(n+2\lambda-\alpha-2)}{n!\Gamma(2\lambda-\alpha-2)}|b_n|-\varepsilon.
  \end{eqnarray}
Combining (\ref{l-eq-6}) and (\ref{l-eq-7}), we obtain that  
    \begin{eqnarray}
[{(1-r^2)^{2\lambda-\alpha-2}}]\sum_{n=0}^{\infty}\frac{\Gamma(n+2\lambda-\alpha-2)}{n!\Gamma(2\lambda-\alpha-2)}\cdot[o(1)]\cdot r^{2n}\rightarrow 0, {\textup{as}}\,\, r\rightarrow 1^{-},\nonumber
 \end{eqnarray}
 so that  \begin{eqnarray}
[{(1-r^2)^{2\lambda-\alpha-2}}]\sum_{n=0}^{\infty}\frac{\Gamma(n+2\lambda-\alpha-2)}{n!\Gamma(2\lambda-\alpha-2)}[1+o(1)]r^{2n}\to 1,  {\textup{as}}\,\, r\rightarrow 1^{-}.\nonumber
 \end{eqnarray}
This means that 
 \begin{equation}
\sum_{n=0}^{\infty}\frac{\Gamma(n+2\lambda-\alpha-2)}{n!\Gamma(2\lambda-\alpha-2)}[1+o(1)]r^{2n}=\frac{1+\hat{o}(1)}{(1-r^2)^{2\lambda-\alpha-2}},\, {\textup{as}}\,\, r\rightarrow 1^{-}.\nonumber
 \end{equation}
 Here and later, $\hat{o}(1)$ denotes a function with respect to $r$ with $\hat{o}(1)\rightarrow 0, \,{\textup{as}}\,\, r\rightarrow 1^{-},$ which may be different in different places. Then, we obtain from (\ref{l-eq-4}) that, for $2\lambda>\alpha+2$, 
 \begin{equation}
   \|{(1-rz)^{-\lambda}}\|_{\alpha}^2=\frac{\Gamma(\alpha+2)\Gamma(2\lambda-\alpha-2)}{[\Gamma(\lambda)]^2}\cdot\frac{1+\hat{o}(1)}{(1-r^2)^{2\lambda-\alpha-2}},\, {\textup{as}}\,\, r\rightarrow 1^{-}.\nonumber 
 \end{equation} 
 Since 
 $$\mathbb{M}_{\mathcal{E}_2}^2(\alpha)\geq  \sup_{\|\phi\|_{\alpha}\neq0, \phi\in \mathcal{H}_{\alpha}^{2}(\Delta)}\frac{\|S_{\mathcal{E}_2}(rz)\phi(z)\|_{\alpha+4}^2}{\|\phi(z)\|_{\alpha}^2}$$ for any $r\in (0,1)$, see \cite[Page 1631]{SS}, hence, by taking $\phi(z)=(1-rz)^{-\lambda}$, we obtain that, for $2\lambda>\alpha+2$,
 \begin{eqnarray}
 \lefteqn{\mathbb{M}_{\mathcal{E}_2}^2(\alpha)\geq \sup_{r\in (0,1)} \frac{9}{4}\frac{\|(1-rz)^{-\lambda-2}\|_{\alpha+4}^2}{\|(1-rz)^{-\lambda}\|_{\alpha}^2}}\nonumber 
 \\
 && \geq
 \frac{9}{4}\frac{\Gamma(\alpha+6)}{[\Gamma(\lambda+2)]^2}\cdot\frac{[\Gamma(\lambda)]^2}{\Gamma(\alpha+2)}=\frac{9(\alpha+2)(\alpha+3)(\alpha+4)(\alpha+5)}{4[\lambda(\lambda+1)]^2}. \nonumber 
 \end{eqnarray}
Let $\lambda\rightarrow {(\frac{\alpha}{2}+1)}^{+}$, we get that
 \begin{equation}
\mathbb{M}_{\mathcal{E}_2}^2(\alpha) \geq \frac{36(\alpha+3)(\alpha+5)}{(\alpha+2)(\alpha+4)}.\nonumber
 \end{equation}

Now, we prove $\mathbb{M}_{P_3}(\alpha)\geq\mathbb{M}_{\kappa}(\alpha)$. First, for $r\in (0,1)$ and $2\lambda>\alpha+2$, similarly as above, we get that
  \begin{equation}\label{a-e-1}
 \|{(1-rz^2)^{-\lambda}}\|_{\alpha}^2=\frac{\Gamma(\alpha+2)\Gamma(2\lambda-\alpha-2)}{2^{\alpha+1}[\Gamma(\lambda)]^2}\cdot\frac{1+\hat{o}(1)}{(1-r^2)^{2\lambda-\alpha-2}},\, {\textup{as}}\,\, r\rightarrow 1^{-},\nonumber 
 \end{equation}
and
\begin{eqnarray}
\lefteqn{S_{P_3}(\sqrt{r}z)(1-rz^2)^{-\lambda}}\nonumber \\
&&=-2\sum_{n=0}^{\infty}\frac{\Gamma(n+\lambda+2)}{n!\Gamma(\lambda+2)}r^nz^{2n}-4\sum_{n=0}^{\infty}\frac{\Gamma(n+\lambda+2)}{n!\Gamma(\lambda+2)}r^{n+1}z^{2n+2} \nonumber \\
&&=-2-\sum_{n=0}^{\infty}\Big[2\frac{\Gamma(n+\lambda+3)}{(n+1)!\Gamma(\lambda+2)}+4\frac{\Gamma(n+\lambda+2)}{n!\Gamma(\lambda+2)}\Big]r^{n+1}z^{2n+2}.\nonumber
\end{eqnarray} 
Then, from (\ref{l-eq-2}) and by using similar arguments above, we obtain that
\begin{equation}
\|S_{P_3}(\sqrt{r}z)(1-rz^2)^{-\lambda}\|_{\alpha+4}^2=36\frac{\Gamma(\alpha+6)\Gamma(2\lambda-\alpha-2)}{2^{\alpha+5}[\Gamma(\lambda+2)]^2}\cdot\frac{1+\hat{o}(1)}{(1-r^2)^{2\lambda-\alpha-2}},\, {\textup{as}}\,\, r\rightarrow 1^{-}.\nonumber 
\end{equation}
Consequently,
\begin{eqnarray}
 \lefteqn{\mathbb{M}_{P_3}^2(\alpha)\geq \sup_{r\in (0,1)} \frac{\|S_{P_3}(\sqrt{r}z)(1-rz^2)^{-\lambda}\|_{\alpha+4}^2}{\|(1-rz^2)^{-\lambda}\|_{\alpha}^2}}\nonumber 
 \\
 && \geq
36\frac{\Gamma(\alpha+6)}{2^{4}[\Gamma(\lambda+2)]^2}\cdot\frac{[\Gamma(\lambda)]^2}{\Gamma(\alpha+2)}=\frac{9(\alpha+2)(\alpha+3)(\alpha+4)(\alpha+5)}{4[\lambda(\lambda+1)]^2}. \nonumber 
 \end{eqnarray} 
Let $\lambda\rightarrow {(\frac{\alpha}{2}+1)}^{+}$, we obtain that
 \begin{equation}
\mathbb{M}_{P_3}^2(\alpha)\geq \frac{36(\alpha+3)(\alpha+5)}{(\alpha+2)(\alpha+4)}.\nonumber
 \end{equation}
This finishes the proof of Proposition \ref{p2}. 
 \end{proof} 
\end{remark}

We continue to investigate the properties of $\|\mathcal{M}_f\|$. Let $\alpha>-1$, we define the functional $\Lambda_{\alpha}$ as
$$ \Lambda_{\alpha}: S_f  \mapsto \|\mathcal{M}_f\|=\mathbb{M}_f(\alpha),\,\, f\in \mathcal{S}.$$  
We obtain that
\begin{proposition}\label{pro-c}
For any fixed $\alpha>-1$, $\Lambda_{\alpha}$ is continuous on $\overline{\mathbf{S}}$, the closure of ${\mathbf{S}}$. 
\end{proposition}
\begin{proof}
Let $\varepsilon>0$ be small. For $f_1, f_2 \in \mathcal{S}$ with $S_{f_1}, S_{f_2} \in \overline{\mathbf{S}}$, we assume that $\|S_{f_1}-S_{f_2}\|_{E_2}=\varepsilon$ so that
\begin{equation}
|S_{f_1}(z)-S_{f_2}(z)|(1-|z|^2)^2\leq \varepsilon,  \nonumber 
\end{equation}
for all $z\in \Delta$. Then,
\begin{eqnarray}
\lefteqn{\mathbb{M}_{f_1}^2(\alpha)=\sup_{\|\phi\|_{\alpha}=1}\|S_{f_1}(z)\phi(z)\|_{\alpha+4}^2}\nonumber  \\
&& =\sup_{\|\phi\|_{\alpha}=1}\|[S_{f_1}(z)-S_{f_2}(z)+S_{f_2}(z)]\phi(z)\|_{\alpha+4}^2\nonumber  \\
&& \leq \sup_{\|\phi\|_{\alpha}=1}\||[S_{f_1}(z)-S_{f_2}(z)]\phi(z)|+|S_{f_2}(z)\phi(z)|\|_{\alpha+4}^2 \nonumber  \\
&& = \sup_{\|\phi\|_{\alpha}=1}\|[S_{f_1}(z)-S_{f_2}(z)]\phi(z)\|_{\alpha+4}^2+\|S_{f_2}(z)\phi(z)\|_{\alpha+4}^2 \nonumber \\
&&\quad\qquad\qquad +(\alpha+5)\iint_{\Delta}2|S_{f_1}(z)-S_{f_2}(z)||S_{f_2}(z)||\phi(z)|^2 (1-|z|^2)^{\alpha+4}\frac{dxdy}{\pi}.  \nonumber  \end{eqnarray} 
Note that, for $\|\phi\|_{\alpha}=1$,
\begin{equation}
 \|[S_{f_1}(z)-S_{f_2}(z)]\phi(z)\|_{\alpha+4}^2\leq {\varepsilon}^2\frac{\alpha+5}{\alpha+1}, \nonumber 
\end{equation}
and 
\begin{equation}
(\alpha+5)\iint_{\Delta}2|S_{f_1}(z)-S_{f_2}(z)||S_{f_2}(z)||\phi(z)|^2 (1-|z|^2)^{\alpha+4}\frac{dxdy}{\pi}\leq 12\varepsilon \frac{\alpha+5}{\alpha+1}.  \nonumber 
\end{equation}
It follows that
\begin{equation}
\mathbb{M}_{f_1}^2(\alpha) \leq \mathbb{M}_{f_2}^2(\alpha)+{\varepsilon}^2\frac{\alpha+5}{\alpha+1}+12\varepsilon \frac{\alpha+5}{\alpha+1}. \nonumber 
\end{equation}
Similarly, we have 
\begin{equation}
\mathbb{M}_{f_2}^2(\alpha) \leq \mathbb{M}_{f_1}^2(\alpha)+{\varepsilon}^2\frac{\alpha+5}{\alpha+1}+12\varepsilon \frac{\alpha+5}{\alpha+1}. \nonumber 
\end{equation}
Consequently, we obtain that
\begin{equation}
\Big|\mathbb{M}_{f_1}^2(\alpha)-\mathbb{M}_{f_2}^2(\alpha) \Big|\leq(12+\varepsilon)\frac{\alpha+5}{\alpha+1}\|S_{f_1}-S_{f_2}\|_{E_2}. \nonumber 
\end{equation}
This implies that $\Lambda_{\alpha}$ is continuous on $\overline{\mathbf{S}}$. The proposition is proved.
\end{proof}
For $\alpha>-1$, we define $${\mathbb{M}}(\alpha):=\sup_{f\in \mathcal{S}_{Q}} \|\mathcal{M}_f\|=\sup_{f\in \mathcal{S}_{Q}}\mathbb{M}_f(\alpha).$$
Then we have 
\begin{proposition}\label{pro-c-1}
For each $\alpha>-1$, we have $\mathbb{M}_f(\alpha)< \mathbb{M}(\alpha)$ for any $f\in \mathcal{S}_{Q}$.  
\end{proposition}
\begin{proof}
Let $f\in \mathcal{S}_{Q}$.  If $\|S_f\|_{E_2}=0$, then $\mathbb{M}_f(\alpha)=0$ and the proposition holds. Next we will assume that $\|S_f\|_{E_2}>0$ so that $\mathbb{M}_f(\alpha)>0$. Then we know from a classical result of the theory of universal Teichm\"uller space (see \cite{Ahl}, \cite{Gehr}) that there is a constant $\delta>0$ such that $\mathcal{O}_f:=\{\varphi\in \mathcal{A}(\Delta): \|\varphi-S_f\|_{E_2}<\delta\}$ is contained in ${\mathbf{S}}_{Q}$, here, ${\mathbf{S}}_{Q}$ is defined as in (\ref{kongs}). It follows that there is a function $f_0 \in \mathcal{S}_{Q}$ such that $S_{f_0}=(1+\frac{\delta}{2}\|S_f\|_{E_2}^{-1})S_f$ since $(1+\frac{\delta}{2}\|S_f\|_{E_2}^{-1})S_f$ belongs to $\mathcal{O}_f$. Then $$\mathbb{M}_{f_{0}}(\alpha)=(1+\frac{\delta}{2}\|S_f\|_{E_2}^{-1})\mathbb{M}_f(\alpha)>\mathbb{M}_f(\alpha).$$ 
The proof of Proposition \ref{pro-c-1} is complete. 
\end{proof}

\begin{remark}
Proposition \ref{pro-c} and \ref{pro-c-1} suggest us that, for each $\alpha>-1$, if the biggest norm $\mathbb{M}(\alpha)$ can be attained at some point $S_f$ in $\overline{\mathbf{S}}_{Q}$, the closure of $\mathbf{S}_Q$, then $S_f$ must lie on $\partial \mathbf{S}_{Q}$. However, we do not know the answer to the following problem. 
\begin{problem}
Can the functional $\Lambda_{\alpha}$ attain at least one maximum in $\overline{\mathbf{S}}_{Q}$ for each $\alpha>-1$?
\end{problem}

\end{remark}

\begin{remark}
In light of Proposition \ref{last-p}, to find more functions satisfying the Brennan conjecture, it is natural to study the problem that in which subset of $\mathcal{S}$ the Koebe function and its rotations provide the biggest norm for the multiplication operator when $\alpha>0$. That is
\begin{problem}\label{p-1}
For each $\alpha>0$, which functions $f\in \mathcal{S}$ satisfy $\mathbb{M}_f(\alpha)\leq \mathbb{M}_{\kappa}(\alpha)$? 
\end{problem} 
It should be pointed out that it has been checked in \cite{Jin-1} that, for each $\alpha>0$, $\mathbb{M}_f(\alpha)\leq \mathbb{M}_{\kappa}(\alpha)$ for a large class of univalent functions $f$ belonging to $\mathcal{S}_Q$.  
\end{remark}

\begin{remark}Let $\alpha>0, f\in \mathcal{S}$. By checking carefully the arguments in the proof of Proposition $8$ in \cite[Page 1632]{SS}, we see that, if there is a constant $\varepsilon>0$ such that $$\mathbb{M}_f(\alpha)+\mathbb{M}_f(\alpha+2)\leq \mathbb{M}_{\kappa}(\alpha)+\mathbb{M}_{\kappa}(\alpha+2),$$
for $\alpha\in (0,\varepsilon)$, then $\beta_f({-2})\leq 1$. Moreover, for $R\in \mathcal{U}_R$ that have at least one critical point on $\mathbb{T}$, if there are two small constants $\epsilon>0, \eta>0$ such that $\mathbb{M}_R(\alpha)<\mathbb{M}_{\kappa}(\alpha)-\eta$
for $\alpha \in (-\epsilon, \epsilon)\cup (2-\epsilon,2+\epsilon)$, then we will obtain that $\beta_{R}(-2)<1$, which is a contradiction. This leads us to conjecture that the following claim holds. Indeed, this is the case, and a proof of this claim can be found in \cite{Jin-p}.
\begin{claim}\label{p-2}
For each $\alpha>-1$, we have $\mathbb{M}_R(\alpha)\geq \mathbb{M}_{\kappa}(\alpha)$ for all $R\in \mathcal{U}_R$ that have at least one critical point on $\mathbb{T}$. 
\end{claim}\end{remark}

\begin{remark}
We have proved that,  for each $\alpha>-1$, 
$$\mathbb{M}_{\mathcal{E}_2}^2(\alpha)=\mathbb{M}_{P_3}^2(\alpha)=\mathbb{M}_{\kappa}^2(\alpha)=\frac{36(\alpha+3)(\alpha+5)}{(\alpha+2)(\alpha+4)}.$$ It is interesting to consider the following
\begin{problem}
For each $\alpha>-1$, which functions $f\in \mathcal{U}_{R}$ satisfy $\mathbb{M}_f(\alpha)=\mathbb{M}_{\kappa}(\alpha)$? 
\end{problem} 
\end{remark}

\begin{remark}If our guess that Problem~\ref{p-1} has an affirmative answer for every $f$ in $\mathcal{S}$ is true, i.e., $\mathbb{M}_f(\alpha) \leq \mathbb{M}_{\kappa}(\alpha)$ holds for every $\alpha>0$ and all $f\in\mathcal{S}$, then the Brennan conjecture follows. Furthermore, since Claim~\ref{p-2} holds, it follows that $\mathbb{M}_R(\alpha)=\mathbb{M}_{\kappa}(\alpha)$ for every $\alpha>0$ and all $R\in\mathcal{U}_R$ that have at least one critical point on $\mathbb{T}$. Although these predictions may seem surprising, no counterexamples are known to the author. Finally, based on the above findings, we further conjecture that
\begin{conjecture}
For every $\alpha>0$, if the functional $\Lambda_{\alpha}$ attains the maximum on $S_f$ in $\overline{\mathbf{S}}$ for some $f\in \mathcal{S}$, then $\|S_f\|_{E_2}=6$.
\end{conjecture}
 \end{remark}
 
 \section{{Acknowledgements}}
The author was supported by the National Natural Science Foundation of China (Grant No. 11501157).

\begin{spacing}{1.3} 

\end{spacing}

\end{document}